\documentclass[10pt, a4paper, fleqn]{article}

\title{Fully Discrete Positivity-Preserving and Energy-Dissipating Schemes for Aggregation-Diffusion Equations with a Gradient-Flow Structure}

\usepackage{authblk}

\author[1]{Rafael Bailo}
\author[1]{Jos\'{e} A. Carrillo}
\author[2]{Jingwei Hu}

\affil[1]{
	Department of Mathematics, Imperial College London
}
\affil[ ]{
	180 Queen's Gate, SW7 2AZ London, United Kingdom
}
\affil[ ]{\textit{
		r.bailo@imperial.ac.uk, carrillo@imperial.ac.uk
	}}
\affil[ ]{}
\affil[2]{
	Department of Mathematics, Purdue University
}
\affil[ ]{
	West Lafayette, IN 47907, United States of America
}
\affil[ ]{\textit{jingweihu@purdue.edu}}

\usepackage{geometry}

\usepackage[utf8x]{inputenc}
\usepackage[T1]{fontenc}
\usepackage{amsmath}
\usepackage{amsthm}
\usepackage{amsfonts}
\usepackage{amssymb}
\usepackage{graphicx}
\usepackage{mathtools}
\usepackage{ifthen}

\usepackage[UKenglish]{babel}

\usepackage[all]{nowidow}

\newcommand{\subjectclassification}[1]{

	{\small\textbf{\textit{AMS Subject Classification --- }} #1}

}
\newcommand{\keywords}[1]{

	{\small\textbf{\textit{Keywords --- }} #1}

}

\renewcommand\lll\MoveEqLeft

\usepackage{tikz}
\usetikzlibrary{patterns}
\usetikzlibrary{decorations.pathreplacing}
\usetikzlibrary{calc}

\usepackage{float}
\usepackage{subcaption}
\usepackage[section]{placeins}

\usepackage{array}
\newcolumntype{L}[1]{>{\raggedright\let\newline\\\arraybackslash\hspace{0pt}}m{#1}}
\newcolumntype{C}[1]{>{\centering\let\newline\\\arraybackslash\hspace{0pt}}m{#1}}
\newcolumntype{R}[1]{>{\raggedleft\let\newline\\\arraybackslash\hspace{0pt}}m{#1}}
\usepackage{booktabs}
\usepackage{tabularx}
\usepackage{multirow}

\usepackage[hypertexnames=false]{hyperref}
\hypersetup{
	colorlinks=true,
	linkcolor=blue,
	citecolor=blue,
	urlcolor=blue,
	linktocpage
}

\usepackage[capitalise]{cleveref}

 \usepackage{accents}
\newcommand{\shiftleft}{\hspace{-1pt}}

\newcommand\term\emph

\makeatletter
\let\newtitle\@title
\let\newauthor\@author
\let\newdate\@date
\makeatother
\usepackage{fancyhdr}
\numberwithin{equation}{section}

\usepackage{enumerate}

\makeatletter
\def\@maketitle{\newpage
	\begin{center}\let \footnote \thanks
		{\LARGE\bfseries \@title \par}\vskip 2.5em{\large
				\lineskip .5em\begin{tabular}[t]{c}\@author
				\end{tabular}\par}\vskip 1em{\large \@date}\end{center}\par
	\vskip 1.5em}
\makeatother 

\usepackage{color}

\newcounter{review}

\newcommand{\ntcreview}[3]{\refstepcounter{review}

	{\color{#2}{\textbf{[#1]}: #3}}}

\newcommand{\creview}[3]{\ntcreview{#1}{#2}{#3}
	\addcontentsline{tor}{subsection}{\thereview~\textbf{[#1]}:~#3
	}}

\newcommand{\review}[2]{\creview{#1}{blue}{#2}}

\makeatletter

\newcommand\listreviewname{List of Reviews}
\newcommand\listofreviews{\section*{\listreviewname}\@starttoc{tor}}
\makeatother

\usepackage{amsthm}
\theoremstyle{plain}
\newtheorem{theorem}{Theorem}[section]

\newtheorem{proposition}[theorem]{Proposition}

\theoremstyle{remark}
\newtheorem{remark}[theorem]{\bf Remark}
\newtheorem{definition}[theorem]{\bf Definition}

\usepackage{esint}

\def\XXint#1#2#3{{\setbox0=\hbox{$#1{#2#3}{\int}$ }
			\vcenter{\hbox{$#2#3$ }}\kern-.6\wd0}}

\DeclarePairedDelimiter{\prt}{(}{)}
\DeclarePairedDelimiter{\brk}{[}{]}

\DeclarePairedDelimiter{\abs}{|}{|}
\DeclarePairedDelimiter{\norm}{\|}{\|}
\DeclarePairedDelimiter{\set}{\{}{\}}

\DeclarePairedDelimiter{\inn}{\langle}{\rangle}

\usepackage{suffix}

\newcommand{\inner}[2]{\inn{#1,#2}}
\WithSuffix\newcommand\inner*[2]{\inn*{#1,#2}}

\DeclarePairedDelimiter{\positive}{(}{)^{+}}
\DeclarePairedDelimiter{\negative}{(}{)^{-}}

\newcommand\pos\positive
\renewcommand\neg\negative

\WithSuffix\newcommand\pos*{\positive*}
\WithSuffix\newcommand\neg*{\negative*}

\newcommand{\Rd}{{\mathbb{R}^d}}

\renewcommand{\L}[1]{{L^{#1}}}

\newcommand{\pnorm}[2]{\norm{#2}_{\L{#1}}}
\WithSuffix\newcommand\pnorm*[2]{\norm*{#2}_{\L{#1}}}

\newcommand{\psnorm}[3]{\norm{#3}_{\L{#1}(#2)}}
\WithSuffix\newcommand\psnorm*[3]{\norm*{#3}_{\L{#1}(#2)}}

\newcommand{\pnormp}[2]{\pnorm{#1}{#2}^{#1}}
\WithSuffix\newcommand\pnormp*[2]{\pnorm*{#1}{#2}^{#1}}

\newcommand{\psnormp}[3]{\psnorm{#1}{#2}{#3}^{#1}}
\WithSuffix\newcommand\psnormp*[3]{\psnorm*{#1}{#2}{#3}^{#1}}

\renewcommand{\vec}{\mathbf}
\newcommand{\bx}{\vec{x}}
\newcommand{\by}{\vec{y}}

\newcommand{\conv}{\ast}

\renewcommand{\d}{\mathrm{d}}
\newcommand{\dd}{\mathop{}\!\d}

\newcommand{\grad}{\nabla}
\renewcommand{\div}{\nabla\cdot}

\newcommand{\pt}{\partial_t}

\newcommand{\tini}{t_{\mathrm{initial}}}
\newcommand{\tfin}{t_{\mathrm{final}}}

\newcommand{\Dt}{\Delta t}
\newcommand{\Dx}{\Delta x}
\newcommand{\Dy}{\Delta y}

\newcommand{\nhalf}{1/2}

\renewcommand{\i}{_{i}}
\newcommand{\ip}{_{i+1}}

\newcommand{\ih}{_{i+\nhalf}}
\newcommand{\imh}{_{i-\nhalf}}

\renewcommand{\j}{_{j}}

\newcommand{\jh}{_{j+\nhalf}}
\newcommand{\jmh}{_{j-\nhalf}}

\renewcommand{\k}{_{k}}

\renewcommand{\l}{_{l}}

\renewcommand{\ij}{_{i,\,j}}
\newcommand{\ipj}{_{i+1,\,j}}
\newcommand{\imj}{_{i-1,\,j}}
\newcommand{\ihj}{_{i+\nhalf,\,j}}
\newcommand{\imhj}{_{i-\nhalf,\,j}}

\newcommand{\ijp}{_{i,\,j+1}}
\newcommand{\ijm}{_{i,\,j-1}}
\newcommand{\ijh}{_{i,\,j+\nhalf}}
\newcommand{\ijmh}{_{i,\,j-\nhalf}}

\newcommand{\ipjp}{_{i+1,\,j+1}}
\newcommand{\ipjm}{_{i+1,\,j-1}}
\newcommand{\imjp}{_{i-1,\,j+1}}
\newcommand{\imjm}{_{i-1,\,j-1}}

\newcommand{\ir}{_{i,\,r}}
\newcommand{\ipr}{_{i+1,\,r}}
\newcommand{\ihr}{_{i+\nhalf,\,r}}
\newcommand{\imhr}{_{i-\nhalf,\,r}}

\newcommand{\rj}{_{r,\,j}}
\newcommand{\rjp}{_{r,\,j+1}}
\newcommand{\rjh}{_{r,\,j+\nhalf}}

\newcommand{\kl}{_{k,\,l}}
\newcommand{\kr}{_{k,\,r}}

\newcommand{\n}{^{n}}
\newcommand{\np}{^{n+1}}

\newcommand{\nss}{^{**}}
\newcommand{\nsss}{^{***}}

\newcommand{\ppr}{(r)}
\newcommand{\pprm}{(r-1)}
\newcommand{\ppi}{(i)}
\newcommand{\ppj}{(j)}
\newcommand{\ppzero}{(0)}
\newcommand{\pptwoM}{(2M)}

\newcommand{\nr}{^{n,\,\ppr}}
\newcommand{\nrm}{^{n,\,\pprm}}
\newcommand{\nzero}{^{n,\,\ppzero}}
\newcommand{\ntwoM}{^{n,\,\pptwoM}}

\newcommand{\Nr}{^{N,\,\ppr}}
\newcommand{\Er}{^{E,\,\ppr}}
\newcommand{\Sr}{^{S,\,\ppr}}
\newcommand{\Wr}{^{W,\,\ppr}}

\newcommand{\nh}{^{n+\nhalf}}
\newcommand{\nhr}{^{n+\nhalf,\,\ppr}}
\newcommand{\nhrm}{^{n+\nhalf,\,\pprm}}
\newcommand{\nhzero}{^{n+\nhalf,\,\ppzero}}
\newcommand{\nhtwoM}{^{n+\nhalf,\,\pptwoM}}

\newcommand{\nj}{^{n,\,\ppj}}
\newcommand{\nhi}{^{n+\nhalf,\,\ppi}}

\newcommand{\nssr}{^{**,\,\ppr}}
\newcommand{\nsssr}{^{***,\,\ppr}}

\newcommand{\Wik}{W_{i-k}}
\newcommand{\Wjk}{W_{j-k}}

\newcommand{\Wikjl}{W_{i-k,\,j-l}}

\newcommand{\Wikrl}{W_{i-k,\,r-l}}
\newcommand{\Wikjr}{W_{i-k,\,j-r}}
\newcommand{\Wikrr}{W_{i-k,\,r-r}}
\newcommand{\Wkirl}{W_{k-i,\,r-l}}

\newlength{\dhatheight}
\newcommand{\doublehat}[1]{\settoheight{\dhatheight}{\ensuremath{\hat{#1}}}\addtolength{\dhatheight}{-0.35ex}\hat{\vphantom{\rule{1pt}{\dhatheight}}\smash{\hat{#1}}}}

\newcommand{\N}{^{N}}
\newcommand{\E}{^{E}}
\renewcommand{\S}{^{S}}
\newcommand{\W}{^{W}}

\ifthenelse{\isundefined{\Wr}}{
	\newcommand{\Wr}{^{W,\,\ppr}}
}{
	\renewcommand{\Wr}{^{W,\,\ppr}}
}

\newcommand{\rx}{\prt{\rho_x}}
\newcommand{\ry}{\prt{\rho_y}}

\DeclareMathOperator*{\minmod}{minmod}

\newcommand{\curlyC}{\mathcal{C}}

\renewcommand{\review}[2]{}
\renewcommand{\tableofcontents}{}
\renewcommand{\listofreviews}{}

\expandafter\def\csname ver@etex.sty\endcsname{3000/12/31}

\usepackage{autonum}
\makeatletter
\patchcmd{\autonum@saveEnvironmentSubcommands}
{(0,0)\begin}
{(0,0)\hfuzz=\maxdimen\begin}
{}{}
\makeatother

\makeatletter
\autonum@generatePatchedReferenceCSL{ref}
\autonum@generatePatchedReferenceCSL{eqref}
\autonum@generatePatchedReferenceCSL{Cref}
\makeatother

\AtBeginDocument{

} 

\newcommand{\revisionnote}[1]{}

\newcommand{\revision}[1]{#1} 
\begin{document}
\maketitle

\begin{abstract}
	We propose fully discrete, implicit-in-time finite-volume schemes for a general family of non-linear and non-local Fokker-Planck equations with a gradient-flow structure, usually known as aggregation-diffusion equations, in any dimension. The schemes enjoy the positivity-preservation and energy-dissipation properties, essential for their practical use. The first-order scheme verifies these properties unconditionally for general non-linear diffusions and interaction potentials, while the second-order scheme does so provided a CFL condition holds. Sweeping dimensional splitting permits the efficient construction of these schemes in higher dimensions while preserving their structural properties. Numerical experiments validate the schemes and show their ability to handle complicated phenomena typical in aggregation-diffusion equations, such as free boundaries, metastability, merging and phase transitions.
\end{abstract}

\subjectclassification{35Q70; 35Q91; 45K05; 65M08.}

 
\tableofcontents
\listofreviews

\section{Introduction}

This work concerns the family of non-linear and non-local aggregation-diffusion equations
\begin{equation}\label{eq:continuous}
	\left\{
	\begin{aligned}
		 & \pt\rho=\div\brk{\rho\grad\prt{H'(\rho)+V+W\conv\rho}}, \quad \bx\in\Rd,\ t>0, \\
		 & \rho(0,\bx)=\rho_0(\bx),
	\end{aligned}
	\right.
\end{equation}
where $\rho=\rho(t,\bx)\geq 0$ is the unknown \textit{particle density}, $H(\rho)$ is the \textit{density of internal energy}, $V(\bx)$ is the \textit{confinement potential}, and $W(\bx)$ is the so-called \textit{interaction potential}; see \cite{C.M.V2003,Villani2003}. $H$ is a convex function by definition, in order to ensure that the first term, $\nabla\cdot[\rho H''(\rho)\nabla \rho]$, is a (possibly degenerate) non-linear density-dependent diffusion. The drift terms $\nabla\cdot [\rho\nabla V]$ and $\nabla\cdot [\rho\nabla(W\ast\rho)]$ correspond respectively to forces acting on the particles due to external sources $V(\bx)$, and to attractive-repulsive forces between particles given by a potential $W(\bx)$.

These equations can be derived as mean-field limits of particle systems or as upscalings of cellular automata, and have applications in granular materials \cite{B.C.P1997,B.C.P1999,B.C.C+1998}, cell migration and chemotaxis \cite{B.C.M2007,C.C2006,L.C.A2008,C.C.Y2019}, collective motion of animals (swarming) \cite{T.B.L2006,K.C.B+2013,C.F.T+2010}, opinion formation \cite{M.T2014,G.P.Y2017}, self-assembly of nanoparticles \cite{H.P2006}, and mathematical finance \cite{G.P.Y2013}, among others. In the case of linear diffusion, these models correspond to macroscopic limits of interacting particle systems in terms of Fokker-Planck type equations; see, for instance, \cite{Sznitman1991,B.C.C2011} and the references therein. They can lead to interesting evolutionary phenomena, such as noise-induced phase transitions or metastability behaviour \cite{B.C.C+2016,G.P2018,B.D.Z2017}, present too in non-linear diffusion models \cite{B.F.H2014,C.C.H2015,C.C.Y2019}. Typical interaction potentials $W(\bx)$ which appear in applications are radial, and can be fully attractive, such as the Newtonian and Bessel potentials in chemotaxis \cite{C.C.Y2019} or power-laws in granular materials \cite{C.M.V2003}; repulsive in the short range and attractive in the long range, such as combinations of power-law or Morse-type potentials in swarming \cite{T.B.L2006,C.H.M2014}; or compactly supported potentials in many biological applications, such as networks and cell sorting \cite{B.D.Z2017,C.C.S2018}.

\Cref{eq:continuous} possesses interesting properties. First, its solution is always a non-negative density. Second, it has a variational structure: it is the \textit{gradient flow} of a free energy, as discovered in \cite{C.M.V2003}. We can define the \textit{free-energy functional} by
\begin{equation}
	E(\rho)=\int_{\Rd} \brk{ H(\rho)+V\rho +\frac{1}{2}(W\ast \rho )\rho } \, \dd{\bx}\, ,
\end{equation}
whose formal variation for zero mass perturbations is given by $\xi:=\frac{\delta E}{\delta \rho}=H'(\rho)+V+W\ast\rho$; the evolution of the free energy along a solution of \eqref{eq:continuous} is given by
\begin{equation}\label{eq:energyDissipation}
	\frac{\dd{E}}{\dd{t}}=\int_{\Rd}\frac{\delta E}{\delta \rho}\frac{\partial \rho}{\partial t}\,\dd{\bx}=\int_{\Rd} \xi \nabla \cdot[\rho \nabla \xi]\,\dd{\bx}=-\int_{\Rd} \rho|\nabla \xi|^2\,\dd{\bx}\leq 0.
\end{equation}
This dissipation property has another interpretation: the solution to \eqref{eq:continuous} is the gradient flow or the curve of steepest descent for the free energy functional $E$ in the sense of the Euclidean transport distance between probability measures, as discussed in \cite{A.G.S2005,C.M.V2003,C.C.Y2019} and the references therein.  We remark that these structural properties persist when the equation is solved in a bounded domain $\Omega$ with no-flux boundary conditions, provided the convolution is understood by extending the density as zero outside $\Omega$, and that $\nabla \xi \cdot \eta(x)=0$ is satisfied for all $x\in\partial\Omega$, where $\eta(x)$ is the outwards unit normal vector to the boundary of $\Omega$.

It is important to notice that the dissipation property entails a full characterisation of the set of stationary states: they are given by non-negative densities such that $\xi$ is constant (possibly different) in each connected component of their support. Therefore, the free energy is a Lyapunov functional for \eqref{eq:continuous}, and will be useful to discuss the stability of the equilibria in many particular cases, see \cite{C.M.V2003,C.G.P+2020}.

These non-linear Fokker-Planck equations have drawn much attention by the numerical analysis and simulation communities in recent times. Indeed, a central question has been the design of numerical schemes capable of preserving the structural properties of the gradient flow \eqref{eq:continuous}: the non-negativity of the solution, the dissipation property \eqref{eq:energyDissipation}, and a corresponding discrete set of stationary states which accurately capture the long-time asymptotics of these equations. First and second-order-accurate finite-volume schemes which treat \cref{eq:continuous} as a non-linear continuity equation (with a vector field given by $-\nabla\xi$) were proposed in \cite{B.F2012} for non-linear diffusions, and in \cite{C.C.H2015} for aggregation-diffusion equations. Their explicit discretisations preserve the positivity-preservation property under a CFL condition; entropy dissipation is shown for the semi-discrete schemes (continuous in time). A generalisation of these ideas for high-order approximations has been proposed in \cite{S.C.S2018}, using a discontinuous Galerkin approach with a suitable quadrature rule which employs Gauss-Lobatto formulas.

In the case of linear diffusion, schemes which enjoy the semi-discrete energy-dissipation property were known for Fokker-Planck equations in granular media \cite{B.C.S2004,B.D2010}, based on the Chang-Cooper discretisation approach \cite{C.C1970}. Such schemes have been generalised and improved for equations of the form \eqref{eq:continuous} in \cite{P.Z2018}, leading to second-order-accurate finite-difference schemes, again with a semi-discrete dissipation. Linear Fokker-Planck equations have many other entropies; for instance, spectral schemes were used in \cite{G.J.Y2012} to achieve the dissipation of the weighted $L^2$ entropies. Finally, implicit-in-time semi-discretisations were proposed in \cite{A.U2003,C.G.J2008} which dissipate the discrete-in-time energy for non-linear Fokker-Planck equations corresponding to \eqref{eq:continuous} with $W=0$. These schemes are reminiscent of the convex splitting ideas in the variational $L^2$ framework, as developed in \cite{S.X.Y2019}; however, they are not directly applicable in the setting of gradient flows with respect to measures. Other numerical schemes used for non-linear Fokker-Planck equations include finite-element schemes \cite{B.C.W2010}; particle/blob methods \cite{C.B2015,C.C.C+2018,C.C.P2019}; and approaches based on the gradient flow formulation, in terms of the steepest descent of Euclidean transport distances \cite{B.C.C2008,C.M2010,M.O2014,J.M.O2017,C.R.W2016,C.D.M+2018}.

In a recent work \cite{A.B.P+2018}, several fully discrete, implicit-in-time discretisations have been proposed for the Keller-Segel model in one dimension (with linear diffusion and non-linear chemosensitivity) which possess the energy dissipation property. Among them, they presented an implicit-in-time, fully discrete scheme based on the gradient flow ideas of \cite{B.F.H2014,C.C.H2015,S.C.S2018}. In the present work we generalise these ideas proposing fully discrete (in both space and time), implicit finite-volume schemes for \Cref{eq:continuous} which are both positivity-preserving and energy-dissipating; these properties are met unconditionally by a scheme with first-order accuracy in time and space, and met under a parabolic CFL condition by a second-order-in-space scheme. Both schemes work for general non-linear diffusions and general interaction potentials $W$, through a careful combination of the implicit-in-time discretisations as hinted in \cite{A.B.P+2018}. Special care has to be taken in using the implicit schemes as the Jacobian matrix can be ill-conditioned in the presence of vacuum, due to the non-linear diffusion terms. A detailed study of the positivity for these schemes is done via M-matrix arguments.

Our next contribution is to propose for the first time a combination of these gradient-flow schemes with a generalised dimensional splitting technique; this results in fully discrete, implicit finite-volume schemes for \cref{eq:continuous}, which are positivity-preserving and energy-dissipating in higher dimensions with a reasonable computational cost. We remark that a direct generalisation of the one-dimensional schemes presented here and in \cite{A.B.P+2018} to multiple dimensions is possible, but comes with a very high computational cost caused by the inversion of large Jacobian matrices. Our approach takes advantage of the dimensional splitting to drastically reduce this cost without sacrificing the fundamental properties of the scheme.

The rest of this work is organised as follows: \cref{sec:timediscretisation} will discuss the time discretisation for \eqref{eq:continuous} in a way that preserves the energy dissipation for general non-linearities and interaction potentials; \cref{sec:1Dschemes} is devoted to the fully discrete scheme in the one-dimensional setting; \revision{\cref{sec:2DschemesDS,sec:2DschemesDSCoupled} respectively introduce the dimensional splitting and the \textit{sweeping dimensional splitting} formulations for higher dimensions;} \cref{sec:validation} is aimed at validating the numerical scheme on explicit solutions for both non-linear diffusions and aggregations; finally, \cref{sec:experiments} presents numerical experiments that showcase the effectiveness of the new scheme in dealing with complicated phenomena, such as metastability or phase transitions, both in one and two dimensions, with linear and non-linear diffusions. \section{Time Discretisation}\label{sec:timediscretisation}
The choice of time discretisation is an integral step in the construction a fully discrete, energy-dissipating scheme. In this section, we shall consider two discretisations, which will lead to two fully discrete schemes, discussed later in the work.

To begin, we define the semi-discrete energy at time $t\n$ as follows:
\begin{equation}
	\begin{aligned}
		E(\rho^{n})=\int_{\Rd} H(\rho^{n})+V\rho\n+\frac{1}{2}(W\ast \rho^{n})\rho\n\,\dd{\bx}.
	\end{aligned}
\end{equation}
For the sake of generality, consider the following first-order scheme:
\begin{equation}\label{eq:semidiscrete}
	\frac{\rho\np-\rho\n}{\Dt}=\nabla\cdot (\rho^*\nabla(H'(\rho\np)+V+W\ast\rho\nss)),
\end{equation}
where $\rho^*$ and $\rho\nss$ can be either $\rho\n$ or $\rho\np$ and will be specified later.
\revision{Here, we have chosen $H'\prt{\rho\np}$ over $H'\prt{\rho\n}$ because the explicit case has been explored in \cite{A.U2003,C.G.J2008} and does not lead to a fully discrete dissipation of the free energy.}
We now attempt to show that scheme \eqref{eq:semidiscrete} verifies $E(\rho\np)\leq E(\rho\n)$, provided $\rho\n\geq 0$, at any time step $t\n$. To this end, we first multiply both sides of \eqref{eq:semidiscrete} by $H'(\rho\np)+V+W\ast\rho\nss$ and integrate to obtain
\begin{align}
	\lll\int_{\Rd}(\rho\np-\rho\n)(H'(\rho\np)+V+W\ast\rho\nss)\,\dd{\bx}        \\
	 & = -\Dt\int_{\Rd} \rho^*|\nabla(H'(\rho\np)+V+W\ast\rho\nss)|^2\,\dd{\bx},
	\label{eq:energyDisProof1}
\end{align}
using integration by parts on the right-hand side. From \eqref{eq:energyDisProof1}, we have
\begin{align}
	\int_{\Rd}(\rho\np-\rho\n)V\,\dd{\bx}
	 & = -\int_{\Rd}(\rho\np-\rho\n)(H'(\rho\np)+W\ast\rho\nss)\,\dd{\bx}             \\
	 & \quad - \Dt\int_{\Rd} \rho^*|\nabla(H'(\rho\np)+V+W\ast\rho\nss)|^2\,\dd{\bx}.
	\label{eq:energyDisProof2}
\end{align}
Then,
\begin{align}
	\lll E(\rho\np)-E(\rho\n)                                                                                                                \\
	 & = \int_{\Rd} H(\rho\np)-H(\rho^{n})+(\rho\np-\rho\n)V+\frac{1}{2}(W\ast \rho\np)\rho\np-\frac{1}{2}(W\ast \rho^{n})\rho^{n}\,\dd{\bx} \\
	 & = \int_{\Rd} H(\rho\np)-H(\rho^{n})-(\rho\np-\rho\n)H'(\rho\np)\,\dd{\bx}                                                             \\
	 & \quad + \int_{\Rd}\frac{1}{2}(W\ast \rho\np)\rho\np-\frac{1}{2}(W\ast \rho^{n})\rho^{n}-(\rho\np-\rho\n)(W\ast\rho\nss)\,\dd{\bx}     \\
	 & \quad - \Dt\int_{\Rd} \rho^*|\nabla(H'(\rho\np)+V+W\ast\rho\nss)|^2\,\dd{\bx}                                                         \\\
	 & = I+II+III,
\end{align}
having used \eqref{eq:energyDisProof2} in the second equality. Here,
\begin{align}
	I   & :=\int_{\Rd} H(\rho\np)-H(\rho^{n})-(\rho\np-\rho\n)H'(\rho\np)\,\dd{\bx};                                                   \\
	II  & :=\int_{\Rd}\frac{1}{2}(W\ast \rho\np)\rho\np-\frac{1}{2}(W\ast \rho^{n})\rho^{n}-(\rho\np-\rho\n)(W\ast\rho\nss)\,\dd{\bx}; \\
	III & :=-\Dt\int_{\Rd} \rho^*|\nabla(H'(\rho\np)+V+W\ast\rho\nss)|^2\,\dd{\bx}.
\end{align}

Our goal is to show that $I+II+III\leq 0$ in order to arrive at $E(\rho\np)\leq E(\rho\n)$. In part $I$,
\begin{equation}
	H(\rho\np)-H(\rho^{n})\leq (\rho\np-\rho\n)H'(\rho\np),
\end{equation}
which follows from the convexity of $H$; hence, $I\leq 0$, as already pointed out in \cite{A.U2003,C.G.J2008}.

In part $II$, there are several possible choices for $\rho\nss$, explored below.
\paragraph{Explicit Case.} If $\rho\nss=\rho^{n}$, we have
\begin{align}
	II(\rho\n) & :=\int_{\Rd}\frac{1}{2}(W\ast \rho\np)\rho\np-\frac{1}{2}(W\ast \rho^{n})\rho^{n}-(\rho\np-\rho\n)(W\ast\rho\n)\,\dd{\bx}\nonumber \\
	           & =\frac{1}{2}\int_{\Rd} \left[ (W\ast \rho\np)\rho\np-2(W\ast \rho^{n})\rho\np+(W\ast \rho^{n})\rho^{n}\right]\,\dd{\bx}\nonumber   \\
	           & =\frac{1}{2}\iint_{\mathbb{R}^{2d}} W(\bx-\by)(\rho\np(\bx)-\rho\n(\bx))(\rho\np(\by)-\rho\n(\by))\,\dd{\bx}\dd{\by},
\end{align}
since $W$ is symmetric. If $W(\bx)=\frac{|\bx|^2}{2}$:
\begin{align}
	II(\rho\n) & = \frac{1}{2}\iint_{\mathbb{R}^{2d}} \left(\frac{|\bx|^2}{2}-\bx\cdot \by+\frac{|\by|^2}{2}\right)(\rho\np(\bx)-\rho^{n}(\bx))(\rho\np(\by)-\rho\n(\by))\,\dd{\bx}\dd{\by} \\
	           & = \frac{1}{2}\left[\left(\int_{\Rd}\frac{|\bx|^2}{2}(\rho\np(\bx)-\rho\n(\bx))\,\dd{\bx}\right)\left(\int_{\Rd} \rho\np(\by)-\rho^{n}(\by)\,\dd{\by}\right)\right.         \\
	           & \quad +\left(\int_{\Rd}\frac{|\by|^2}{2}(\rho\np(\by)-\rho\n(\by))\,\dd{\by}\right)\left(\int_{\Rd} \rho\np(\bx)-\rho^{n}(\bx)\,\dd{\bx}\right)                            \\
	           & \quad \left.-\left(\int_{\Rd} \bx(\rho\np(\bx)-\rho\n(\bx))\,\dd{\bx}\right)\cdot\left(\int_{\Rd} \by(\rho\np(\by)-\rho^{n}(\by))\,\dd{\by}\right)\right]                  \\
	           & = -\frac{1}{2}\left|\int_{\Rd} \bx(\rho\np(\bx)-\rho\n(\bx))\,\dd{\bx}\right|^2 \leq 0,
\end{align}
where we used the conservation of mass property,
\begin{align}
	\int_{\Rd} \rho\np(\bx)\,\dd{\bx}=\int_{\Rd}\rho\n(\bx)\,\dd{\bx}.
\end{align}
If $\hat{W}(\xi)\leq 0$, where $\hat{W}(\xi)$ is the Fourier transform of $W(\bx)$:
\begin{equation}
	\begin{aligned}
		II(\rho\n) & =\frac{1}{2}\int_{\Rd}(W\ast (\rho\np-\rho\n))(\bx)(\rho\np(\bx)-\rho\n(\bx))\,\dd{\bx},                                                                    \\
		           & =\frac{1}{2}\iint_{\mathbb{R}^{2d}}\hat{W}(\xi)(\hat{\rho}\np(\xi)-\hat{\rho}\n(\xi))(\rho\np(\bx)-\rho\n(\bx))e^{2\pi i \bx\cdot \xi}\,\dd{\xi}\,\dd{\bx}, \\
		           & =\frac{1}{2}\int_{\mathbb{R}^{d}}\hat{W}(\xi)(\hat{\rho}\np(\xi)-\hat{\rho}\n(\xi))(\hat{\rho}\np(-\xi)-\hat{\rho}\n(-\xi))\,\dd{\xi},                      \\
		           & =\frac{1}{2}\int_{\mathbb{R}^{d}}\hat{W}(\xi)(\hat{\rho}\np(\xi)-\hat{\rho}\n(\xi))(\overline{\hat{\rho}\np(\xi)}-\overline{\hat{\rho}\n(\xi)})\,\dd{\xi},  \\
		           & =\frac{1}{2}\int_{\Rd} \hat{W}(\xi)|\hat{\rho}\np(\xi)-\hat{\rho}\n(\xi)|^2\,\dd{\xi}\leq 0.
	\end{aligned}
\end{equation}

\paragraph{Implicit Case.} If $\rho\nss=\rho\np$, we have
\begin{align}
	II(\rho\np) & :=\int_{\Rd}\frac{1}{2}(W\ast \rho\np)\rho\np-\frac{1}{2}(W\ast \rho^{n})\rho^{n}-(\rho\np-\rho\n)(W\ast\rho\np)\,\dd{\bx}\nonumber,         \\
	            & =-\frac{1}{2}\int_{\Rd} \left[ (W\ast \rho\np)\rho\np-2(W\ast \rho^{n})\rho\np+(W\ast \rho^{n})\rho^{n}\right]\,\dd{\bx}\nonumber,           \\
	            & =-\frac{1}{2}\iint_{\mathbb{R}^{2d}}\shiftleft W(\bx-\by)(\rho\np(\bx)-\rho\n(\bx))(\rho\np(\by)-\rho\n(\by))\,\dd{\bx}\dd{\by}=-II(\rho\n),
\end{align}
which is just the negative of the previous case. Thus, we conclude $II(\rho\np)\leq 0$ for $W(\bx)=-\frac{|\bx|^2}{2}$ and $\hat{W}(\xi)\geq 0$.

\paragraph{Unconditional Case.} If $\rho\nss=\prt{\rho\np+\rho\n}/{2}$, we have
\begin{align}
	II =\int_{\Rd}\frac{1}{2}(W\ast \rho\np)\rho\np-\frac{1}{2}(W\ast \rho^{n})\rho^{n}-\frac{1}{2}(\rho\np-\rho\n)(W\ast(\rho\np+\rho\n))\,\dd{\bx}\nonumber = 0.
\end{align}
Therefore, $II$ is exactly zero regardless of the choice of $W$.

Finally, $III\leq 0$ follows from the positivity of $\rho^{*}$, chosen as either $\rho\n$ or $\rho\np$. We now remark that the computations above remain unaltered when considered over a bounded domain $\Omega$ with no-flux boundary conditions. It is a simple exercise to check that the boundary terms left by the integration by parts vanish.

Henceforth, an interaction $W$ satisfying $II(\rho\n)=-II(\rho\np)\leq 0$ (respectively $II(\rho\n)=-II(\rho\np)\geq 0$) will be referred as a
\revision{\textit{negative-definite} (resp. \textit{positive-definite})} interaction potential. We have just shown that attractive (resp. repulsive) quadratic potentials, as well as certain catastrophic (resp. H-stable) potentials (according to the notation of classical statistical mechanics in \cite{Ruelle1969}), are \revision{negative definite (resp. positive definite)}. These potentials are very relevant in applications for the existence of both steady states and phase transitions, with and without (linear or non-linear) diffusion terms; see \cite{C.C.P2015,C.D.M2016,C.D.P2019,C.K.Y2013,B.D.Z2017,C.G.P+2020} and the references therein.

To summarise, a suitable choice of $\rho\nss$, given $W$, always exists, and we have obtained two time discretisation methods which satisfy an energy dissipation property:
\begin{proposition}\label{prop:timediscrete}
	\begin{enumerate}[I.]
		\item The implicit time discretisation
		      \begin{equation}\label{eq:S1discretisation}
			      \frac{\rho\np-\rho\n}{\Dt}=\nabla\cdot (\rho^{n}\nabla(H'(\rho\np)+V+W\ast\rho\nss)),
		      \end{equation}
		      satisfies the energy dissipation property, i.e., $E(\rho\np)\leq E(\rho\n)$, if one of the following conditions holds:
		      \begin{enumerate}[(i)]
			      \item $\rho\nss=\prt{\rho\np+\rho\n}/{2}$;
			      \item $\rho\nss=\rho^{n}$ and $W$ is a \revision{negative-definite} interaction potential;
			      \item $\rho\nss=\rho\np$ and $W$ is a \revision{positive-definite} interaction potential.
		      \end{enumerate}

		\item The implicit time discretisation
		      \begin{equation} \label{eq:S2discretisation}
			      \frac{\rho\np-\rho\n}{\Dt}=\nabla\cdot (\rho\np\nabla(H'(\rho\np)+V+W\ast\rho\nss)),
		      \end{equation}
		      satisfies the energy dissipation property, i.e., $E(\rho\np)\leq E(\rho\n)$, if one of the following conditions holds:
		      \begin{enumerate}[(i)]
			      \item $\rho\nss=\prt{\rho\np+\rho\n}/{2}$;
			      \item $\rho\nss=\rho^{n}$ and $W$ is a \revision{negative-definite} interaction potential;
			      \item $\rho\nss=\rho\np$ and $W$ is a \revision{positive-definite} interaction potential.
		      \end{enumerate}

	\end{enumerate}
\end{proposition}

It is an interesting open problem to show that the schemes \eqref{eq:S1discretisation} and \eqref{eq:S2discretisation} are well posed in the set of non-negative densities, for small enough $\Dt$, under reasonable assumptions on the potentials. \section{Fully Discrete Schemes in One Dimension}
\label{sec:1Dschemes}

In this section we present fully discrete schemes in one dimension, by coupling the time discretisations of \cref{sec:timediscretisation} with the finite-volume method in space. We introduce two schemes, corresponding to the discretisations \eqref{eq:S1discretisation} and \eqref{eq:S2discretisation}.

To begin, we consider a large computational domain $\Omega = [-L,L]$, for $L>0$, and divide it into $2M$ uniform cells of size $\Dx= L/M$. We denote the $i$-th cell by $C\i=[x\imh,x\ih]$ for $i=1,\dots,2M$; the cell centre is $x\i=-L+\Dx(i-1/2)$. No-flux boundary conditions are assumed at the boundaries.

\subsection{Scheme 1 (S1)}

This scheme is based on the implicit time discretisation (\ref{eq:S1discretisation}). The spatial discretisation follows the fully explicit finite-volume method in \cite{C.C.H2015}, where only semi-discrete energy dissipation was shown (continuous in time). By modifying certain terms into implicit form, we obtain a fully discrete, energy-dissipating scheme with second-order accuracy in space.

Assume $\rho\i$ is the cell average on $C\i$; the scheme reads
\begin{align}\label{eq:S1}
	\begin{split}
		&\frac{\rho\i\np-\rho\i\n}{\Dt}+\frac{F\ih\np-F\imh\np}{\Dx}=0,\\
		&F\ih\np=\rho\i\E(u\ih\np)^++\rho\ip\W(u\ih\np)^-,\\
		&(u\ih\np)^+=\max(u\ih\np,0),
		\qquad(u\ih\np)^-=\min(u\ih\np,0),\\
		&u\ih\np=-\frac{\xi\ip\np-\xi\i\np}{\Dx}, \quad
		\\ &
		\xi\i\np=H'(\rho\i\np)+V\i+(W\ast\rho\nss)\i,\\
		&\rho\i\E=\rho\i\n+\frac{\Dx}{2}\rx\n\i,
		\qquad\rho\i\W=\rho\i\n-\frac{\Dx}{2}\rx\n\i,\\
		&\rx\n\i=\text{minmod}\left(\theta\frac{\rho\ip\n-\rho\i\n}{\Dx}, \frac{\rho\ip\n-\rho_{i-1}\n}{2\Dx},\theta\frac{\rho_{i}\n-\rho_{i-1}\n}{\Dx}\right),
	\end{split}
\end{align}
for a symmetric potential $W$. As suggested previously, $\rho\nss$ could be $\rho\n$, $\rho\np$, or $\prt{\rho\np+\rho\n}/{2}$. The confining potential terms are defined as $V\i=V(x\i)$, and the convolution is given by $(W\ast\rho\nss)\i=\sum_{k=1}^{2M}\Wik\rho_k\nss\Dx$, where $\Wik=W(x\i-x\k)$. The minmod limiter is defined as
\begin{equation}
	\text{minmod} (z_1,z_2\cdots):=\left\{
	\begin{aligned}
		 & \min(z_1,z_2,\cdots), &  & \text{if} \ \text{all } z\i>0, \\
		 & \max(z_1,z_2,\cdots), &  & \text{if} \ \text{all } z\i<0, \\
		 & 0                     &  & \text{otherwise};
	\end{aligned}
	\right.
\end{equation}
we choose $\theta=2$.

\begin{remark}\label{rm:singularW}
	One might be interested in potentials $W$ with an integrable singularity at the origin, as is the case in \cite{C.C.H2015}. In that situation, the definition of $\Wik$ is modified to an integral form,
	\begin{equation}
		\Wik=\frac{1}{\Dx}\int_{C\k}W(x\i-s) \,\dd{s},
	\end{equation}
	along the corresponding cell.
\end{remark}

\subsubsection{Positivity Preservation Property}

\begin{theorem}\label{th:S1CFL}
	Scheme \eqref{eq:S1} is positivity preserving: if $\rho\i\n\geq 0$ for all $i$, then $\rho\i\np\geq 0$ for all $i$, provided the following CFL condition is satisfied:
	\begin{equation}\label{eq:S1CFL}
		\Dt\leq \frac{\Dx}{2\max\i\set*{(u\ih\np)^+,-(u\ih\np)^-}}.
	\end{equation}
\end{theorem}

\begin{proof}
	From the definition of the scheme in \eqref{eq:S1}, we have
	\begin{equation} \label{eq:S1CFLProof1}
		\frac{\rho\i\np-\rho\i\n}{\Dt}+\frac{\rho\i\E\pos{u\ih\np}+\rho\ip\W\neg{u\ih\np}-\rho_{i-1}\E\pos{u\imh\np}-\rho\i\W\neg{u\imh\np}}{\Dx}=0.
	\end{equation}
	Hence,
	\begin{align}
		\rho\i\np & = \frac{1}{2}\prt{\rho\i\E+\rho\i\W} -\frac{\Dt}{\Dx}\prt{\rho\i\E\pos{u\ih\np}+\rho\ip\W\neg{u\ih\np}-\rho_{i-1}\E\pos{u\imh\np}-\rho\i\W\neg{u\imh\np}} \\
		          & = \frac{\Dt}{\Dx}\pos{u\imh\np}\rho_{i-1}\E + \left(\frac{1}{2}-\frac{\Dt}{\Dx}\pos{u\ih\np}\right)\rho\i\E                                               \\
		          & \quad + \left(\frac{1}{2}+\frac{\Dt}{\Dx}\neg{u\imh\np}\right)\rho\i\W - \frac{\Dt}{\Dx}\neg{u\ih\np}\rho\ip\W.
	\end{align}
	By construction, if $\rho\n\i \geq 0$, then $\rho\i\E, \rho\i\W \geq 0$. Moreover, $\pos{u\imh\np}\geq 0$, $\neg{u\imh\np}\leq 0$. Provided condition \eqref{eq:S1CFL} is satisfied, $\rho\i\np$ is a sum of non-negative values, hence $\rho\i\np\geq 0$.
\end{proof}

\begin{remark}
	For the equivalent first-order scheme, where $\rho\i\E=\rho\i\W=\rho\i\n$, \cref{eq:S1CFLProof1} becomes
	\begin{equation}
		\frac{\rho\i\np-\rho\i\n}{\Dt}+\frac{\rho\i\n\pos{u\ih\np}+\rho\ip\n\neg{u\ih\np}-\rho_{i-1}\n\pos{u\imh\np}-\rho\i\n\neg{u\imh\np}}{\Dx}=0,
	\end{equation}
	which yields
	\begin{equation}
		\rho\i\np=\left(1-\frac{\Dt}{\Dx}\pos{u\ih\np}+\frac{\Dt}{\Dx}\neg{u\imh\np}\right)\rho\i\n+\frac{\Dt}{\Dx}\rho_{i-1}\n\pos{u\imh\np}-\frac{\Dt}{\Dx}\rho\ip\n\neg{u\ih\np}.
	\end{equation}
	Hence, a sufficient CFL condition to guarantee positivity can be
	\begin{equation}
		\Dt\leq \frac{\Dx}{\max\i \set*{(u\ih\np)^+ - (u\imh\np)^-}}.
	\end{equation}
\end{remark}

\begin{remark}
	Since $u\ih\np=\mathcal{O}(\Dx^{-1})$, a parabolic CFL condition $\Dt=\mathcal{O}(\Dx^2)$ is normally required for the scheme (in either its first or its second-order version) to guarantee the positivity preservation property.
\end{remark}

\subsubsection{Energy Dissipation Property}

We shall define the fully discrete free energy at time $t\n$ by
\begin{equation}\label{eq:discreteEnergy}
	E_\Delta\prt{\rho\n} =
	\Dx\prt*{
		\sum_{i=1}^{2M} H\prt{\rho\i\n}
		+\sum_{i=1}^{2M} V\i\rho\i\n
		+\frac{\Dx}{2}\sum_{i,k=1}^{2M}\Wik\rho\i\n\rho\k\n
	}.
\end{equation}
We would like to show that $E_\Delta(\rho\n)$ decreases at each time step. To that end, it is useful to introduce a classification for interaction potentials, following the discussion in Section 2:

\begin{definition}[Positive and negative-definite potential]
	Assume $\rho\i\n\geq 0$ on the interval $[-L,L]$ and zero outside of it at any time $t\n$. An interaction potential $W$ such that
	\begin{align}
		\sum_{i,k=1}^{2M}\Wik(\rho\i\np-\rho\i\n)(\rho\k\np-\rho\k\n)\leq 0
		\qquad \prt{\textrm{respectively } \geq 0}
	\end{align}
	is called a \revision{\textit{negative-definite}} (resp. \revision{\textit{positive-definite}}) \textit{potential}.
\end{definition}

We can classify the following potentials:
\begin{proposition}\label{th:potentialCharacterisation}
	\begin{enumerate}[I.]
		\item A quadratic potential $W(x)={x^2}/{2}$ is \revision{negative definite}.
		\item A quadratic potential $W(x)=-{x^2}/{2}$ is \revision{positive definite}.
		\item A potential such that $\hat{W}\l\leq 0$ (resp. $\hat{W}\l\geq 0$) for all $l=1,\dots,4M$, where
		      \begin{align}
			      \hat{W}\l=\sum_{k=-M+1}^{3M}W\k\exp\set*{-2\pi i \frac{(k+M-1)(l-1)}{4M}}
		      \end{align}
		      is the discrete Fourier transform of the sequence $W\k=W(x\k)$ defined on $[-2L,2L]$, is \revision{negative definite} (resp. \revision{positive definite}).
	\end{enumerate}
\end{proposition}

\begin{proof}
	If $W(x)={x^2}/{2}$, then
	\begin{align}
		 & \sum_{i,k=1}^{2M}\Wik(\rho\np\i-\rho\n\i)(\rho\k\np-\rho\k\n)
		\\=&
		\sum_{i,k=1}^{2M}\frac{(x\i-x\k)^2}{2} (\rho\np\i-\rho\n\i)(\rho\k\np-\rho\k\n)
		,                                                                \\=&
		\sum_{i=1}^{2M} x\i^2(\rho\i\np-\rho\i\n)\sum_{k=1}^{2M} (\rho\k\np-\rho\k\n)-\left(\sum_{i=1}^{2M} x\i(\rho\i\np-\rho\i\n)\right)^2
		,                                                                \\=&
		-\left(\sum_{i=1}^{2M} x\i(\rho\i\np-\rho\i\n)\right)^2\leq 0,
	\end{align}
	where we have used the fact that the scheme (\ref{eq:S1}) conserves the mass, $\sum_{i=1}^{2M} \rho\i\n=\sum_{i=1}^{2M} \rho\i\np$, under the no-flux boundary condition.

	If $W(x)$ is such that $\hat{W}_l\leq 0$ for all $l=1,\dots,4M$, then
	\begin{align}
		 & \sum_{j,k=1}^{2M}\Wjk (\rho\np\j-\rho\n\j)(\rho\k\np-\rho\k\n)
		\\=&
		\sum_{j=-M+1}^{3M}\sum_{k=-M+1}^{3M}\Wjk (\rho\np\j-\rho\n\j)(\rho\k\np-\rho\k\n)
		,                                                                 \\=&
		\frac{1}{4M}\sum_{j=-M+1}^{3M}\sum_{l=1}^{4M}\hat{W}_l(\hat{\rho}\np\l-\hat{\rho}\n\l)e^{2\pi i \frac{(j+M-1)(l-1)}{4M}} (\rho\np\j-\rho\n\j)
		,                                                                 \\=&
		\frac{1}{4M}\sum_{l=1}^{4M}\hat{W}\l(\hat{\rho}\np\l-\hat{\rho}\n\l)\overline{(\hat{\rho}\l\np-\hat{\rho}\l\n)}
		,                                                                 \\=&
		\frac{1}{4M}\sum_{l=1}^{4M}\hat{W}\l|\hat{\rho}\np\l-\hat{\rho}\n\l|^2\leq 0,
	\end{align}
	where the first equality extends summation to the interval $[-2L,2L]$, since $\rho\n\i=0$ outside $[-L,L]$, and the second equality uses the discrete circular convolution theorem by assuming that the sequence $W\i$ defined on $[-2L,2L]$ is periodically extended to the whole space $\mathbb{R}$.
\end{proof}

We can now state the following:

\begin{theorem}\label{th:energydissipationS1}
	Under the CFL condition \eqref{eq:S1CFL}, scheme \eqref{eq:S1} satisfies the energy dissipation property, i.e., $E_\Delta(\rho\np)\leq E_\Delta(\rho\n)$, if one of the following conditions holds:
	\begin{enumerate}[(i)]
		\item $\rho\nss=\prt{\rho\np+\rho\n}/{2}$;
		\item $\rho\nss=\rho\n$ and the potential $W$ is \revision{negative definite};
		\item $\rho\nss=\rho\np$ and the potential $W$ is \revision{positive definite}.
	\end{enumerate}
\end{theorem}

\begin{proof}
	From the definition of scheme \eqref{eq:S1}, it follows
	\begin{align}
		\sum_{i=1}^{2M}(\rho\i\np-\rho\i\n)\xi\i\np=-\frac{\Dt}{\Dx}\sum_{i=1}^{2M} (F\ih\np-F\imh\np)\xi\i\np,
	\end{align}
	with $\xi\i\np=H'(\rho\i\np)+V\i+\sum_{k=1}^{2M}\Wik\rho\k\nss\Dx$. We can rewrite it, as in \eqref{eq:energyDisProof2}, to obtain
	\begin{align}
		\lll \sum_{i=1}^{2M}(\rho\i\np-\rho\i\n)V\i                                                                                                                     \\
		 & = -\frac{\Dt}{\Dx}\sum_{i=1}^{2M} (F\ih\np-F\imh\np)\xi\i\np-\sum_{i=1}^{2M}(\rho\i\np-\rho\i\n)\left(H'(\rho\i\np)+\sum_{k=1}^{2M}\Wik\rho\k\nss\Dx\right).
	\end{align}
	Using the definition of the discrete energy, \cref{eq:discreteEnergy}, and the identity above, we deduce
	\begin{align}
		\lll E_\Delta(\rho\np)-E_\Delta (\rho\n)                                                                                                 \\
		 & = \Dx\left(\sum_{i=1}^{2M} H(\rho\i\np)+\sum_{i=1}^{2M} V\i\rho\i\np+\frac{\Dx}{2}\sum_{i,k=1}^{2M}\Wik\rho\i\np\rho\k\np\right)      \\
		 & \quad -\Dx\left(\sum_{i=1}^{2M} H\prt{\rho\i\n}+\sum_{i=1}^{2M} V\i\rho\i\n+\frac{\Dx}{2}\sum_{i,k=1}^{2M}\Wik\rho\i\n\rho\k\n\right) \\
		 & = \Dx \sum_{i=1}^{2M}\left( H(\rho\i\np)-H\prt{\rho\i\n}-H'(\rho\i\np)(\rho\i\np-\rho\i\n) \right)                                    \\
		 & \quad +\frac{\Dx^2}{2}\sum_{i,k=1}^{2M}\Wik \left(\rho\i\np\rho\k\np-\rho\i\n\rho\k\n-2(\rho\i\np-\rho\i\n)\rho\k\nss \right)         \\
		 & \quad - \Dt \sum_{i=1}^{2M} (F\ih\np-F\imh\np)\xi\i\np                                                                                \\
		 & = I+II+III,
	\end{align}
	where
	\begin{align}
		I   & :=\Dx \sum_{i=1}^{2M}\left( H(\rho\i\np)-H\prt{\rho\i\n}-H'(\rho\i\np)(\rho\i\np-\rho\i\n) \right);                       \\
		II  & :=\frac{\Dx^2}{2}\sum_{i,k=1}^{2M}\Wik \left(\rho\i\np\rho\k\np-\rho\i\n\rho\k\n-2(\rho\i\np-\rho\i\n)\rho\k\nss \right); \\
		III & :=-\Dt \sum_{i=1}^{2M} (F\ih\np-F\imh\np)\xi\i\np.
	\end{align}

	The first term is easily controlled since $H(\rho\i\np)-H\prt{\rho\i\n}-H'(\rho\i\np)(\rho\i\np-\rho\i\n)\leq 0$ since $H(\rho)$ is convex, hence $I\leq 0$.

	Considering $II$:
	\begin{itemize}
		\item If $\rho\nss=\rho\n$,
		      \begin{equation}
			      II=\frac{\Dx^2}{2}\sum_{i,k=1}^{2M}\Wik (\rho\np\i-\rho\n\i)(\rho\k\np-\rho\k\n).
		      \end{equation}
		      Then, $II\leq 0$ if $W$ is negative definite.

		\item If $\rho\nss=\rho\np$,
		      \begin{equation}
			      II=-\frac{\Dx^2}{2}\sum_{i,k=1}^{2M}\Wik (\rho\np\i-\rho\n\i)(\rho\k\np-\rho\k\n).
		      \end{equation}
		      Then, $II\leq 0$ if $W$ is positive definite.
		\item
		      If $\rho\nss=\prt{\rho\np+\rho\n}/{2}$, $II\equiv0$.
	\end{itemize}

	Therefore, for the aforementioned choices of $\rho\nss$ and corresponding choices of $W$, we find
	\begin{align}
		E_\Delta(\rho\np)-E_\Delta (\rho\n)\leq & -\Dt\sum_{i=1}^{2M} \xi\i\np(F\ih\np-F\imh\np)                                \\
		=                                       & -\Dt\sum_{i=1}^{2M-1} F\ih\np(\xi\i\np-\xi\ip\np)                             \\
		=                                       & -\Dt\Dx\sum_{i=1}^{2M-1} F\ih\np u\ih\np                                      \\
		=                                       & -\Dt\Dx\sum_{i=1}^{2M-1}(\rho\i\E\pos{u\ih\np}+\rho\ip\W\neg{u\ih\np})u\ih\np \\
		\leq                                    & -\Dt\Dx\sum_{i=1}^{2M-1} \min(\rho\i\E,\rho\ip\W)|u\ih\np|^2\leq 0,
	\end{align}
	where the first equality employs summation by parts as well as the no-flux boundary condition, and the last inequality relies on the positivity of $\rho\i\E$ and $\rho\i\W$.
\end{proof}

\subsection{Scheme 2 (S2)}

This scheme is based on the time discretisation (\ref{eq:S2discretisation}). The spatial discretisation is the same as in the first-order version of S1. By modifying certain terms into implicit form, we can obtain a fully discrete unconditionally positivity-preserving and energy-dissipating scheme. This scheme is related to the one introduced in \cite{A.B.P+2018} for the Keller-Segel model with non-linear chemosensitivity and linear diffusion.

Assume $\rho\i$ is the cell average on $C\i$; the scheme reads
\begin{align}\label{eq:S2}
	\begin{split}
		&\frac{\rho\i\np-\rho\i\n}{\Dt}+\frac{F\ih\np-F\imh\np}{\Dx}=0,\\
		&F\ih\np=
		\rho\i\np(u\ih\np)^++\rho\ip\np(u\ih\np)^-,\\
		&(u\ih\np)^+=\max(u\ih\np,0),
		\qquad(u\ih\np)^-=\min(u\ih\np,0),\\
		&u\ih\np=-\frac{\xi\ip\np-\xi\i\np}{\Dx}, \quad
		\\ &
		\xi\i\np=H'(\rho\i\np)+V\i+(W\ast\rho\nss)\i,
	\end{split}
\end{align}
for a symmetric potential $W$. Again, $\rho\nss$ may be chosen as $\rho\n$, $\rho\np$, or $\prt{\rho\np+\rho\n}/{2}$; $V\i$ and $(W\ast\rho\nss)\i$ are defined as above.

\subsubsection{Positivity Preservation Property}

\begin{theorem}\label{th:S2CFL}
	Scheme \eqref{eq:S2} is unconditionally positivity preserving: if $\rho\i\n\geq 0$ for all $i$, then $\rho\i\np\geq 0$ for all $i$.
\end{theorem}

\begin{proof}From the definition of the scheme, \cref{eq:S2}, we find
	\begin{equation}
		\frac{\rho\i\np-\rho\i\n}{\Dt}+\frac{\rho\i\np\pos{u\ih\np}+\rho\ip\np\neg{u\ih\np}-\rho_{i-1}\np\pos{u\imh\np}-\rho\i\np\neg{u\imh\np}}{\Dx}=0,
	\end{equation}
	i.e.,
	\begin{equation}
		\left(1+\frac{\Dt}{\Dx}\pos{u\ih\np}-\frac{\Dt}{\Dx}\neg{u\imh\np}\right)\rho\i\np+\frac{\Dt}{\Dx}\neg{u\ih\np}\rho\ip\np-\frac{\Dt}{\Dx}\pos{u\imh\np}\rho_{i-1}\np=\rho\i\n,
	\end{equation}
	which may be written as
	\begin{equation}
		A(\rho\np)\rho\np=\rho\n.
	\end{equation}
	We would like to show that the matrix $A$ is inverse positive, meaning that every entry of $A^{-1}$ is non-negative. We recall a sufficient condition: a matrix $M$ is inverse positive if $m_{ij}\leq 0$ for $i\neq j$, $m_{ii}>0$, and $M$ is strictly diagonally dominant ($m_{ii}> \sum_{i\neq j}|m_{ij}|$). The matrix $A$ defined above does not satisfy the condition; however, $A^T$ always does. Hence, every entry of $(A^T)^{-1}$ is non-negative, and thus so are those of $A^{-1}$. Thus, if $\rho\n\i\geq 0$ for all $i$, then $\rho\np\i\geq 0$ for all $i$.
\end{proof}

\subsubsection{Energy Dissipation Property}

\begin{theorem}\label{th:energydissipationS2}
	Scheme \eqref{eq:S2} satisfies the energy dissipation property unconditionally, i.e., the discrete energy \eqref{eq:discreteEnergy} satisfies $E_\Delta(\rho\np)\leq E_\Delta(\rho\n)$, if one of the following condition holds:
	\begin{enumerate}[(i)]
		\item $\rho\nss=\prt{\rho\np+\rho\n}/{2}$;
		\item $\rho\nss=\rho\n$ and the potential $W$ is \revision{negative definite};
		\item $\rho\nss=\rho\np$ and the potential $W$ is \revision{positive definite}.
	\end{enumerate}
\end{theorem}

\begin{proof}
	The proof of \cref{th:energydissipationS1} remains valid except for the last part, which instead is argued as
	\begin{equation}
		\begin{aligned}
			E_\Delta(\rho\np)-E_\Delta (\rho\n)\leq & -\Dt\sum_{i=1}^{2M} \xi\i\np(F\ih\np-F\imh\np)                                  \\
			=                                       & -\Dt\sum_{i=1}^{2M-1} F\ih\np(\xi\i\np-\xi\ip\np)                               \\
			=                                       & -\Dt\Dx\sum_{i=1}^{2M-1} F\ih\np u\ih\np                                        \\
			=                                       & -\Dt\Dx\sum_{i=1}^{2M-1}(\rho\i\np\pos{u\ih\np}+\rho\ip\np\neg{u\ih\np})u\ih\np \\
			\leq                                    & -\Dt\Dx\sum_{i=1}^{2M-1} \min(\rho\i\np,\rho\ip\np)|u\ih\np|^2
			\leq 0;
		\end{aligned}
	\end{equation}
	the unconditional positivity of $\rho\i$ is used in the last inequality.
\end{proof}
 \section{Higher Dimensions: Dimensional Splitting}
\label{sec:2DschemesDS}

\revision{
	The schemes presented above can be directly extended to any number of dimensions. In this section, instead, we propose higher-dimension schemes through the dimensional splitting technique, and show that the positivity preservation and energy dissipation properties are maintained.
}

\revision{
	The advantage of the dimensional splitting framework for numerical methods is a large reduction of the computational cost. The one-dimensional schemes from \cref{sec:1Dschemes} involve the solution of an implicit, non-linear problem at each time-step. For the sake of illustration, we consider the typical Newton-Raphson method. The computational cost of such a method arises from the need to invert a Jacobian matrix at each iteration; an $N\times N$ full matrix can be inverted with complexity $\mathcal{O}\prt{N^\gamma}$, for $2<\gamma\leq 3$ \cite{C.W1990, Strassen1969}. In a mesh of $N$ cells per dimension, the solution of a fully $d$-dimensional implicit scheme would require inverting $N^d\times N^d$ Jacobians, at cost $\mathcal{O}\prt{N^{d\gamma}}$. In comparison, the complexity of inverting the $dN^{d-1}$ Jacobians of size $N\times N$ required by the dimensionally split method is $\mathcal{O}\prt{dN^{d+\gamma-1}}$; this is an improvement at $d=2$ already, and certainly so for higher dimensions.
}

\revision{
	Some thought must be given to the convolution terms. First, we remark that the dimensional splitting technique does not reduce the overall cost of computing the convolution terms.
	Second, we observe that the Jacobian of the non-linear problem is a full matrix whenever the convolution terms are treated implicitly.
	More importantly, the estimation of the computational advantage shown above only holds provided the dimensionally split schemes decouple ``row-by-row'', i.e., the one-dimensional sub-problems given by the splitting are independent of each other. This is simply not the case when the interaction potential terms are treated implicitly or semi-implicitly, due to the convolution; a subtler approach for these cases is presented in \cref{sec:2DschemesDSCoupled}.
}

We proceed to describe two-dimensional versions of the schemes in the splitting framework; the extension to any number of dimensions can be done in a similar fashion. For simplicity, we assume a square domain $[-L,L]^2$ and partition both the $x$ and $y$-axes uniformly using $2M$ cells per axis. Then, $\Dx=\Dy=L/M$ and the $i,j$-th cell is denoted by $C\ij=[x\imh,x\ih]\times[y\jmh,y\jh] $, with centre at $x\ij=(-L+{\Dx}/{2}+\Dx(i-1),-L+{\Dy}/{2}+\Dy(j-1))$. Again, no-flux boundary conditions are assumed at the boundaries.

\subsection{Scheme 1 (S1)}

Assume $\rho\ij$ is the cell average on cell $C\ij$; the scheme reads:

\begin{itemize}
	\item \textbf{Step 1 --- Evolution in the $x$-direction}
	      \begin{align}\label{eq:S1SplittingX}
		      \begin{split}
			      &\frac{\rho\ij\nh-\rho\ij\n}{\Dt}+\frac{F\ihj\nh-F\imhj\nh}{\Dx}=0,\\
			      &F\ihj\nh=\rho\ij\E\pos{u\ihj\nh}+\rho\ipj\W\neg{u\ihj\nh},\\
			      &\pos{u\ihj\nh}=\max(u\ihj\nh,0),
			      \qquad\neg{u\ihj\nh}=\min(u\ihj\nh,0),\\
			      &u\ihj\nh=-\frac{\xi\ipj\nh-\xi\ij\nh}{\Dx}, \quad
			      \\&
			      \xi\ij\nh=H'( \rho\ij\nh)+V\ij+(W\ast\rho\nss)\ij,\\
			      &\rho\ij\E=\rho\ij\n+\frac{\Dx}{2}\rx\n\ij,
			      \qquad\rho\ij\W=\rho\ij\n-\frac{\Dx}{2}\rx\n\ij,\\
			      &\rx\n\ij=\minmod\left(\theta\frac{\rho\ipj\n-\rho\ij\n}{\Dx}, \frac{\rho\ipj\n-\rho\imj\n}{2\Dx},\theta\frac{\rho\ij\n-\rho\imj\n}{\Dx}\right).
		      \end{split}
	      \end{align}
	      Once again, the choice of $\rho\nss$ may be $\rho\n$, $\rho\nh$, or $\prt{\rho\n+\rho\nh}/{2}$. $V\ij=V(x\ij)$ and
	      \begin{equation}
		      (W\ast\rho\nss)\ij=\sum_{k,l=1}^{2M}W_{i-k,j-l}\rho\kl\nss\Dx\Dy,
	      \end{equation}
	      where $W_{i-k, j-l}=W(x\i-x\k,y\j-y\l)$. Remark \ref{rm:singularW} similarly applies for singular potentials in two dimensions.

	\item \textbf{Step 2 --- Evolution in the $y$-direction}
	      \begin{align}\label{eq:S1SplittingY}
		      \begin{split}
			      &\frac{\rho\ij\np-\rho\ij\nh}{\Dt}+\frac{G\ijh\np-G\ijmh\np}{\Dy}=0,\\
			      &G\ijh\np=\rho\ij\N\pos{v\ijh\np}+\rho\ijp\S\neg{v\ijh\np},\\
			      &\pos{v\ijh\np}=\max(v\ijh\np,0),
			      \qquad\neg{v\ijh\np}=\min(v\ijh\np,0),\\
			      &v\ijh\np=-\frac{\xi\ijp\np-\xi\ij\np}{\Dy}, \quad
			      \\&
			      \xi\ij\np=H'( \rho\ij\np)+V\ij+(W\ast\rho\nsss)\ij,\\
			      &\rho\ij\N=\rho\ij\nh+\frac{\Dx}{2}(\rho_y)\nh\ij,
			      \qquad\rho\ij\S=\rho\ij\nh-\frac{\Dx}{2}(\rho_y)\nh\ij,\\
			      &(\rho_y)\nh\ij=\minmod\left(\theta\frac{\rho\ijp\nh-\rho\ij\nh}{\Dy}, \frac{\rho\ijp\nh-\rho\ijm\nh}{2\Dy},\theta\frac{\rho\ij\nh-\rho\ijm\nh}{\Dy}\right).
		      \end{split}
	      \end{align}
	      Here, $\rho\nsss$ may be $\rho\nh$, $\rho\np$, or $\prt{\rho\nh+\rho\np}/{2}$. Other quantities, such as $(W\ast\rho\nsss)\ij$, are defined analogously.
\end{itemize}

\subsubsection{Positivity Preservation Property}

\begin{theorem}
	Scheme \textup{(\ref{eq:S1SplittingX},~\ref{eq:S1SplittingY})} is positivity preserving provided the following CFL condition is satisfied:
	\begin{equation}\label{eq:S1CFLSplitting}
		\Dt\leq \frac{1}{2}\min\set*{
			\frac{\Dx}{\max\ij\set*{\pos{u\ihj\nh}, \neg{u\ihj\nh}}},
			\frac{\Dy}{\max\ij\set*{\pos{v\ijh\np}, \neg{v\ijh\np}}}
		}.
	\end{equation}
\end{theorem}

\begin{proof}
	From the first step of the scheme, \eqref{eq:S1SplittingX}, follows
	\begin{equation}
		\frac{\rho\ij\nh-\rho\ij\n}{\Dt}+\frac{\rho\ij\E\pos{u\ihj\nh}+\rho\ipj\W\neg{u\ihj\nh}-\rho\imj\E\pos{u\imhj\nh}-\rho\ij\W\neg{u\imhj\nh}}{\Dx}=0.
	\end{equation}
	Hence,
	\begin{align}
		\rho\ij\nh & = \frac{1}{2} \prt{\rho\ij\E+\rho\ij\W} - \frac{\Dt}{\Dx} \left(\rho\ij\E\pos{u\ihj\nh} + \rho\ipj\W\neg{u\ihj\nh}\right. \\
		           & \quad - \left.\rho\imj\E\pos{u\imhj\nh} - \rho\ij\W\neg{u\imhj\nh}\right),                                                \\
		           & = \frac{\Dt}{\Dx}\pos{u\imhj\nh}\rho\imj\E+\left(\frac{1}{2}-\frac{\Dt}{\Dx}\pos{u\ihj\nh}\right)\rho\ij\E                \\
		           & \quad +\left(\frac{1}{2}+\frac{\Dt}{\Dx}\neg{u\imhj\nh}\right)\rho\ij\W-\frac{\Dt}{\Dx}\neg{u\ihj\nh}\rho\ipj\W.
	\end{align}
	Since $\rho\ij\nh$ are linear combinations of non-negative reconstructed values, $\rho\imj\E$, $\rho\ij\E$, $\rho\ij\W$, and $\rho\ipj\W$, and $\pos{u\imhj\nh}\geq 0$, $\neg{u\ihj\nh}\leq0$, we may conclude that $\rho\ij\nh\geq 0$, provided the condition
	\begin{equation}\label{eq:S1CFLSplittingProof1}
		\Dt\leq \frac{\Dx}{2\max\ij\set*{\pos{u\ihj\nh},-\neg{u\ihj\nh}}}
	\end{equation}
	is satisfied.

	The second step, \eqref{eq:S1SplittingY}, follows analogously:
	\begin{equation}
		\frac{\rho\ij\np-\rho\ij\nh}{\Dt}+\frac{\rho\ij\N\pos{v\ijh\np}+\rho\ijp\S\neg{v\ijh\np}-\rho\ijm\N\pos{v\ijmh\np}-\rho\ij\S\neg{v\ijmh\np}}{\Dy}=0,
	\end{equation}
	hence
	\begin{align}
		\rho\ij\np & = \frac{1}{2} \prt{\rho\ij\N+\rho\ij\S} - \frac{\Dt}{\Dy}\left(\rho\ij\N\pos{v\ijh\np} + \rho\ijp\S\neg{v\ijh\np} \right. \\
		           & \quad - \left.\rho\ijm\N\pos{v\ijmh\np} - \rho\ij\S\neg{v\ijmh\np}\right),                                                \\
		           & = \frac{\Dt}{\Dy}\pos{v\ijmh\np}\rho\ijm\N + \prt*{\frac{1}{2}-\frac{\Dt}{\Dy}\pos{v\ijh\np}}\rho\ij\N                    \\
		           & \quad + \prt*{\frac{1}{2}+\frac{\Dt}{\Dy}\neg{v\ijmh\np}}\rho\ij\S - \frac{\Dt}{\Dy}\neg{v\ijh\np}\rho\ijp\S.
	\end{align}
	Once more, $\rho\ij\np$ are linear combinations of non-negative reconstructed point values, $\rho\ijm\N$, $\rho\ij\N$, $\rho\ij\S$, and $\rho\ijp\S$ , and $\pos{v\ijmh\np}\geq 0$, $\neg{v\ijh\np}\leq0$; positivity, $\rho\ij\np\geq 0$, follows provided the condition
	\begin{equation}\label{eq:S1CFLSplittingProof2}
		\Dt\leq \frac{\Dy}{2\max\ij\set*{\pos{v\ijh\np},-\neg{v\ijh\np}}}
	\end{equation}
	is satisfied.

	The combination of \Cref{eq:S1CFLSplittingProof1,eq:S1CFLSplittingProof2} yields condition \cref{eq:S1CFLSplitting} in the theorem.
\end{proof}

\subsubsection{Energy Dissipation Property}

We define the fully discrete free energy at time $t\n$ by
\begin{equation}\label{eq:energyDiscrete2D}
	E_\Delta(\rho\n)=\Dx\Dy\left(\sum_{i,j=1}^{2M} H(\rho\ij\n)+\sum_{i,j=1}^{2M} V\ij\rho\ij\n+\frac{\Dx\Dy}{2}\sum_{i,j,k,l=1}^{2M}W_{i-k,j-l}\rho\ij\n\rho\kl\n\right).
\end{equation}
We aim to show the dissipation of $E_\Delta(\rho\n)$ at each time step. Again, we introduce the classification for interaction potentials:
\begin{definition}
	Assume $\bar{\rho}\ij, \tilde{\rho}\ij\geq 0$ on the domain $[-L,L]^2$ and zero outside of it. An interaction potential $W$ such that
	\begin{align}
		\sum_{i,j,k,l=1}^{2M}W_{i-k,j-l}(\bar{\rho}\ij-\tilde{\rho}\ij)(\bar{\rho}\kl-\tilde{\rho}\kl)\leq 0
		\qquad \prt{\textrm{respectively } \geq 0}
	\end{align}
	is called a \revision{\textit{negative-definite}} (resp. \revision{\textit{positive-definite}}) potential.
\end{definition}

A result analogous to \cref{th:potentialCharacterisation} can be obtained in higher dimensions to provide the same examples of negative and positive-definite potentials; this is left to the reader.

\begin{theorem}\label{th:energydissipationS1DS}
	Under the CFL condition \eqref{eq:S1CFLSplitting}, scheme \textup{(\ref{eq:S1SplittingX},~\ref{eq:S1SplittingY})} satisfies the energy dissipation property, i.e., $E_\Delta(\rho\np)\leq E_\Delta(\rho\n)$, if one of the following conditions holds:
	\begin{enumerate}[(i)]
		\item $\rho\nss=\prt{\rho\n+\rho\nh}/2$ and $\rho\nsss=\prt{\rho\nh+\rho\np}/2$;
		\item $\rho\nss=\rho\n$, $\rho\nsss=\rho\nh$, and the potential $W$ is \revision{negative-definite};
		\item $\rho\nss=\rho\nh$, $\rho\nsss=\rho\np$, and the potential $W$ is \revision{positive-definite}.
	\end{enumerate}
\end{theorem}

\begin{proof}
	From step 1, \eqref{eq:S1SplittingX}, it follows that
	\begin{align}
		\sum_{i,j=1}^{2M}(\rho\ij\nh-\rho\ij\n)\xi\ij\nh=-\frac{\Dt}{\Dx}\sum_{i,j=1}^{2M} (F\ihj\nh-F\imhj\nh)\xi\ij\nh,
	\end{align}
	where $\xi\ij\nh=H'(\rho\ij\nh)+V\ij+\sum_{k,l=1}^{2M}W_{i-k,j-l}\rho\kl\nss\Dx\Dy$. As in Section 3, we deduce
	\begin{align}
		\lll E_\Delta(\rho\nh)-E_\Delta (\rho\n)                                                                                                                                                              \\
		 & = \Dx\Dy\left(\sum_{i,j=1}^{2M} H(\rho\ij\nh)+\shiftleft\shiftleft\sum_{i,j=1}^{2M} V\ij\rho\ij\nh+\frac{\Dx\Dy}{2}\shiftleft\shiftleft\sum_{i,j,k,l=1}^{2M}W_{i-k,j-l}\rho\ij\nh\rho\kl\nh\right) \\
		 & \quad -\Dx\Dy\left(\sum_{i,j=1}^{2M} H(\rho\ij\n)+\sum_{i,j=1}^{2M} V\ij\rho\ij\n+\frac{\Dx\Dy}{2}\sum_{i,j,k,l=1}^{2M}W_{i-k,j-l}\rho\ij\n\rho\kl\n\right)                                        \\
		 & = \Dx \Dy\sum_{i,j=1}^{2M}\left( H(\rho\ij\nh)-H(\rho\ij\n)-H'(\rho\ij\nh)(\rho\ij\nh-\rho\ij\n) \right)                                                                                           \\
		 & \quad + \frac{(\Dx\Dy)^2}{2}\sum_{i,j,k,l=1}^{2M}W_{i-k,j-l} \left(\rho\ij\nh\rho\kl\nh-\rho\ij\n\rho\kl\n-2(\rho\ij\nh-\rho\ij\n)\rho\kl\nss \right)                                              \\
		 & \quad - \Dt\Dy \sum_{i,j=1}^{2M} (F\ihj\nh-F\imhj\nh)\xi\ij\nh.
	\end{align}
	The terms involving $W$ are either non-positive or identically zero, depending on the choice of $\rho\nss$, just as in the one-dimensional setting.

	From step 2, \eqref{eq:S1SplittingY}, we deduce
	\begin{align}
		\sum_{i,j=1}^{2M}(\rho\ij\np-\rho\ij\nh)\xi\ij\np=-\frac{\Dt}{\Dy}\sum_{i,j=1}^{2M} (G\ijh\np-G\ijmh\np)\xi\ij\np,
	\end{align}
	where $\xi\ij\np=H'(\rho\ij\np)+V\ij+\sum_{k,l=1}^{2M}W_{i-k,j-l}\rho\kl\nsss\Dx\Dy$, and similarly
	\begin{align}
		\lll E_\Delta(\rho\np)-E_\Delta (\rho\nh)                                                                                                                                                                                                                                                                            \\
		 & = \Dx\Dy\left(\sum_{i,j=1}^{2M} H(\rho\ij\np)+\sum_{i,j=1}^{2M} V\ij\rho\ij\np+\frac{\Dx\Dy}{2}\shiftleft\shiftleft\sum_{i,j,k,l=1}^{2M}W_{i-k,j-l}\rho\ij\np\rho\kl\np\right)                                                                                                                                    \\
		 & \quad - \Dx\Dy\left(\sum_{i,j=1}^{2M} H(\rho\ij\nh)+\shiftleft\shiftleft \sum_{i,j=1}^{2M} V\ij\rho\ij\nh+\frac{\Dx\Dy}{2} \shiftleft\shiftleft\shiftleft\shiftleft\shiftleft\shiftleft \sum_{i,j,k,l=1}^{2M} \shiftleft\shiftleft\shiftleft\shiftleft\shiftleft\shiftleft W_{i-k,j-l}\rho\ij\nh\rho\kl\nh\right) \\
		 & = \Dx\Dy\sum_{i,j=1}^{2M}\left( H(\rho\ij\np)-H(\rho\ij\nh)-H'(\rho\ij\np)(\rho\ij\np-\rho\ij\nh) \right)                                                                                                                                                                                                         \\
		 & \quad + \frac{(\Dx\Dy)^2}{2}\sum_{i,j,k,l=1}^{2M}W_{i-k,j-l} \left(\rho\ij\np\rho\kl\np-\rho\ij\nh\rho\kl\nh-2(\rho\ij\np-\rho\ij\nh)\rho\kl\nsss \right)                                                                                                                                                         \\
		 & \quad - \Dt\Dx\sum_{i,j=1}^{2M} (G\ijh\np-G\ijmh\np)\xi\ij\np.
	\end{align}
	Yet again, the terms involving $W$ are either non-positive or identically zero, depending on the choice of $\rho\nsss$, consistently with the first step. Combining the estimates, we finally conclude
	\begin{align}
		\lll E_\Delta(\rho\np)-E_\Delta (\rho\n)                                                                                 \\
		 & \leq -\Dt\Dy\sum_{i,j=1}^{2M} \xi\ij\nh(F\ihj\nh-F\imhj\nh)-\Dt\Dx\sum_{i,j=1}^{2M} \xi\ij\np(G\ijh\np-G\ijmh\np)     \\
		 & = -\Dt\Dy\sum_{i,j=1}^{2M-1} F\ihj\nh(\xi\ij\nh-\xi\ipj\nh) -\Dt\Dx\sum_{i,j=1}^{2M-1} G\ijh\np(\xi\ij\np-\xi\ijp\np) \\
		 & = -\Dt\Dx\Dy\sum_{i,j=1}^{2M-1} F\ihj\nh u\ihj\nh -\Dt\Dx\Dy\sum_{i,j=1}^{2M-1} G\ijh\np v\ijh\np                     \\
		 & = -\Dt\Dx\Dy\sum_{i,j=1}^{2M-1}(\rho\ij\E\pos{u\ihj\nh}+\rho\ipj\W\neg{u\ihj\nh})u\ihj\nh                             \\
		 & \quad -\Dt\Dx\Dy\sum_{i,j=1}^{2M-1}(\rho\ij\n\pos{v\ijh\np}+\rho\ijp\S\neg{v\ijh\np})v\ijh\np                         \\
		 & \leq -\Dt\Dx\Dy\sum_{i,j=1}^{2M-1} \min(\rho\ij\E,\rho\ipj\W)|u\ihj\nh|^2                                             \\
		 & \quad -\Dt\Dx\Dy\sum_{i,j=1}^{2M-1} \min(\rho\ij\n,\rho\ijp\S)|v\ijh\np|^2 \leq 0,
	\end{align}
	where we employ the summation by parts, the no-flux boundary condition, and the positivity of $\rho\ij$.
\end{proof}

\subsection{Scheme 2 (S2)}

Assume $\rho\ij$ is the cell average on cell $C\ij$; the scheme reads
\begin{itemize}
	\item \textbf{Step 1 --- Evolution in the $x$-direction}
	      \begin{align}\label{eq:S2SplittingX}
		      \begin{split}
			      &\frac{\rho\ij\nh-\rho\ij\n}{\Dt}+\frac{F\ihj\nh-F\imhj\nh}{\Dx}=0,\\
			      &F\ihj\nh=\rho\ij\nh\pos{u\ihj\nh}+\rho\ipj\nh\neg{u\ihj\nh},\\
			      &\pos{u\ihj\nh}=\max(u\ihj\nh,0), \ \ \neg{u\ihj\nh}=\min(u\ihj\nh,0),\\
			      &u\ihj\nh=-\frac{\xi\ipj\nh-\xi\ij\nh}{\Dx}, \quad
			      \\&
			      \xi\ij\nh=H'( \rho\ij\nh)+V\ij+(W\ast\rho\nss)\ij.
		      \end{split}
	      \end{align}
	      Where $\rho\nss$ may be $\rho\n$, $\rho\nh$, or $\prt{\rho\n+\rho\nh}/{2}$.

	\item \textbf{Step 2 --- Evolution in the $y$-direction}
	      \begin{align}\label{eq:S2SplittingY}
		      \begin{split}
			      &\frac{\rho\ij\np-\rho\ij\nh}{\Dt}+\frac{G\ijh\np-G\ijh\np}{\Dy}=0,\\
			      &G\ijh\np=\rho\ij\np\pos{v\ijh\np}+\rho\ijp\np\neg{v\ijh\np},\\
			      &\pos{v\ijh\np}=\max(v\ijh\np,0), \ \ \neg{v\ijh\np}=\min(v\ijh\np,0),\\
			      &v\ijh\np=-\frac{\xi\ijp\np-\xi\ij\np}{\Dy},\quad
			      \\&
			      \xi\ij\np=H'( \rho\ij\np)+V\ij+(W\ast\rho\nsss)\ij.
		      \end{split}
	      \end{align}
	      Where $\rho\nsss$ can be $\rho\nh$, $\rho\np$, or $\prt{\rho\nh+\rho\np}/{2}$.
\end{itemize}

\subsubsection{Positivity Preservation Property}

\begin{theorem}
	Scheme \textup{(\ref{eq:S2SplittingX},~\ref{eq:S2SplittingY})} is unconditionally positivity preserving.
\end{theorem}

\begin{proof}
	Just as in the proof of \cref{th:S2CFL}, we rewrite the first step \eqref{eq:S2SplittingX} to find
	\begin{align}
		\rho\ij\n & = \left(1+\frac{\Dt}{\Dx}\pos{u\ihj\nh}-\frac{\Dt}{\Dx}\neg{u\imhj\nh}\right)\rho\ij\nh     \\
		          & \quad + \frac{\Dt}{\Dx}\neg{u\ihj\nh}\rho\ipj\nh-\frac{\Dt}{\Dx}\pos{u\imhj\nh}\rho\imj\nh,
	\end{align}
	which may be written as
	\begin{equation}
		A(\tilde\rho\nh)\tilde\rho\nh=\tilde\rho\n, \quad \tilde\rho_{(j-1)N+i}=\rho\ij,
	\end{equation}
	where $A=(a_{ij})$ is an M-matrix. Therefore $\tilde\rho_{i}\nh\geq0$ if $\tilde\rho_{i}\n\geq 0$ for all $i$. Similarly, the second step (\ref{eq:S2SplittingY}) yields
	\begin{align}
		\rho\ij\nh & = \left(1+\frac{\Dt}{\Dy}\pos{v\ijh\np}-\frac{\Dt}{\Dy}\neg{v\ijmh\np}\right)\rho\ij\np     \\
		           & \quad + \frac{\Dt}{\Dy}\neg{v\ijh\np}\rho\ijp\np-\frac{\Dt}{\Dy}\pos{v\ijmh\np}\rho\ijm\np,
	\end{align}
	which can be written as
	\begin{equation}
		A(\tilde\rho\np)\tilde\rho\np=\tilde\rho\nh, \quad \tilde\rho_{(j-1)N+i}=\rho\ij,
	\end{equation}
	where $A=(a_{ij})$ is an M-matrix also. Therefore $\tilde\rho_{i}\np\geq0$ if $\tilde\rho_{i}\nh\geq 0$ for all $i$.
\end{proof}

\subsubsection{Energy Dissipation Property}

\begin{theorem}\label{th:energydissipationS2DS}
	Scheme \textup{(\ref{eq:S2SplittingX},~\ref{eq:S2SplittingY})} satisfies the energy dissipation property unconditionally, i.e., the discrete energy \eqref{eq:energyDiscrete2D} satisfies $E_\Delta(\rho\np)\leq E_\Delta(\rho\n)$, if one of the following conditions holds:
	\begin{enumerate}[(i)]
		\item $\rho\nss=\prt{\rho\n+\rho\nh}/2$ and $\rho\nsss=\prt{\rho\nh+\rho\np}/2$;
		\item $\rho\nss=\rho\n$, $\rho\nsss=\rho\nh$, and the potential $W$ is \revision{negative-definite};
		\item $\rho\nss=\rho\nh$, $\rho\nsss=\rho\np$, and the potential $W$ is \revision{positive-definite}.
	\end{enumerate}
\end{theorem}

\begin{proof}
	The proof of the result in the previous section carries over except for the last part:
	\begin{align}
		\lll E_\Delta(\rho\np)-E_\Delta (\rho\n)                                                                                 \\
		 & \leq -\Dt\Dy\sum_{i,j=1}^{2M} \xi\ij\nh(F\ihj\nh-F\imhj\nh) -\Dt\Dx\sum_{i,j=1}^{2M} \xi\ij\np(G\ijh\np-G\ijmh\np)    \\
		 & = -\Dt\Dy\sum_{i,j=1}^{2M-1} F\ihj\nh(\xi\ij\nh-\xi\ipj\nh) -\Dt\Dx\sum_{i,j=1}^{2M-1} G\ijh\np(\xi\ij\np-\xi\ijp\np) \\
		 & = -\Dt\Dx\Dy\sum_{i,j=1}^{2M-1} F\ihj\nh u\ihj\nh -\Dt\Dx\Dy\sum_{i,j=1}^{2M-1} G\ijh\np v\ijh\np                     \\
		 & = -\Dt\Dx\Dy\sum_{i,j=1}^{2M-1}(\rho\ij\nh\pos{u\ihj\nh}+\rho\ipj\nh\neg{u\ihj\nh})u\ihj\nh                           \\
		 & \quad -\Dt\Dx\Dy\sum_{i,j=1}^{2M-1}(\rho\ij\np\pos{v\ijh\np}+\rho\ijp\np\neg{v\ijh\np})v\ijh\np                       \\
		 & \leq -\Dt\Dx\Dy\sum_{i,j=1}^{2M-1} \min(\rho\ij\nh,\rho\ipj\nh)|u\ihj\nh|^2                                           \\
		 & \quad -\Dt\Dx\Dy\sum_{i,j=1}^{2M-1} \min(\rho\ij\np,\rho\ijp\np)|v\ijh\np|^2 \leq 0.
	\end{align}
\end{proof}
 \section{Higher Dimensions: Sweeping Dimensional Splitting}
\label{sec:2DschemesDSCoupled}

\revisionnote{This section is new}

The schemes presented in \cref{sec:2DschemesDS} offer a dimensionally split generalisation of the S1 and S2 schemes, chosen over their direct generalisation to higher dimensions due to their reduced computational complexity. However, as it was noted at the beginning of the section, this improvement only holds whenever the split schemes decouple \textit{row-by-row}; for instance, in the case of S1, whenever the update \eqref{eq:S1SplittingX} can be performed separately for each value of $j$, and the update \eqref{eq:S1SplittingY} can similarly be computed independently for each $i$.

In pursuit of the energy dissipation property shown in \cref{th:energydissipationS1,th:energydissipationS2,th:energydissipationS1DS,th:energydissipationS2DS}, we have proposed schemes which deal with the interaction terms implicitly. The resulting convolutions of the $\rho\np$ variable leave us with schemes without the desired decoupling. To overcome the computational cost of direct higher-dimensional schemes or coupled dimensionally-split generalisations, we now propose the \textit{sweeping dimensional splitting} framework.

Once again, we describe only the two-dimensional schemes, discretising the mesh as in \cref{sec:2DschemesDS}; the generalisation to three and higher-dimensional settings is immediate.

\subsection{Scheme 1 (S1)}

Assume $\rho\n$ is the approximate solution at time $t\n$. We define the scheme, in its sweeping dimensional splitting form, in two steps: an update in the $x$-direction, followed by the $y$-direction. Each of the updates happens through a sequence of individual row-by-row (or column-by-column) iterations.

\begin{itemize}
	\item \textbf{Evolution in the $x$-direction}

	      We let $\rho\nzero \coloneqq \rho\n$ and define, for $1\leq i, j, r\leq 2M$, the implicit sequence
	      \begin{align}\label{eq:2DS1X}
		      \rho\ij\nr =
		      \begin{cases}\displaystyle
			      \rho\ij\nrm - \frac{\Dt}{\Dx}\prt*{F\ihj\nr - F\imhj\nr} & \text{if }j=r ,
			      \\ \rho\ij\nrm & \text{otherwise} ;
		      \end{cases}
	      \end{align}
	      where
	      \begin{equation}\label{eq:2DS1Xterms}
		      \begin{aligned}
			      F\ihj\nr & = \rho\ij\Er \pos{u\ihj\nr} + \rho\ipj\Wr \neg{u\ihj\nr} ,
			      \\ u\ihj\nr &= - \frac{\xi\ipj\nr - \xi\ij\nr}{\Dx} ,
			      \\ \xi\ij\nr &= H'\prt{\rho\ij\nr} + V\ij + \prt{W\conv\hat{\rho}\nr}\ij ;
			      \\ \rho\ij\Er &= \rho\ij\nr + \frac{\Dx}{2}\rx\ij\nr,
			      \qquad \rho\ij\Wr = \rho\ij\nr - \frac{\Dx}{2}\rx\ij\nr,
			      \\ \rx\n\ij &= \minmod\prt*{
				      \theta \frac{\rho\ipj\n - \rho\ij\n}{\Dx}, \;
				      \frac{\rho\ipj\n - \rho\imj\n}{2\Dx}, \;
				      \theta \frac{\rho\ij\n - \rho\imj\n}{\Dx}
			      };
		      \end{aligned}
	      \end{equation}

	      and the discrete convolution is defined by the sum
	      \begin{align}\label{eq:2Ddiscreteconvolution}
		      \prt{W\conv\hat{\rho}\nr}\ij =
		      \sum_{k,l=1}^{2M} \Wikjl \hat{\rho}\kl\nr \Dx\Dy .
	      \end{align}
	      Furthermore, the convolution variable is given by
	      \begin{align}\label{eq:2DXconvolutionvariable}
		      \hat{\rho}\kl\nr =
		      \begin{cases}
			      \rho\kl\nssr & \text{if }l = r ,
			      \\ \rho\kl\nrm & \text{otherwise} ,
		      \end{cases}
	      \end{align}
	      where once again the choice of $\rho^{**}$ depends on the properties of the interaction potential $W$: throughout the scheme, $\rho\nssr$ is chosen as one of $\rho\nr$, $\rho\nrm$, or $\prt{\rho\nr + \rho\nrm}/{2}$.

	      The last term of the sequence defines the semi-update $\rho\nh\coloneqq \rho\ntwoM$. See \cref{fig:sweeping} for a diagram of the update.

	      \begin{figure}
		      \centering
		      \begin{tikzpicture}[scale=2]
	\newcommand{\gw}{1}
	\newcommand{\gang}{55}
	\newcommand{\gh}{0.8*sin((\gang))}
	\newcommand{\gs}{cos((\gang))}
	\newcommand{\gsep}{-3.5}

	\foreach \g in {0,1} {

			\foreach \l in {0,...,3} {

					\foreach \k in {0,...,2} {

							\draw[thick] ({\k*\gw + \l*\gs}, {\l*\gh + \g*\gsep}) -- ({(\k+1)*\gw + \l*\gs}, {\l*\gh + \g*\gsep});

						}

					\draw[thick, dashed] ({0*\gw + \l*\gs}, {\l*\gh + \g*\gsep}) -- ({(-0.5)*\gw + \l*\gs}, {\l*\gh + \g*\gsep});
					\draw[thick, dashed] ({3*\gw + \l*\gs}, {\l*\gh + \g*\gsep}) -- ({(3.5)*\gw + \l*\gs}, {\l*\gh + \g*\gsep});

				}

			\foreach \k in {0,...,3} {

					\foreach \l in {0,...,2} {

							\draw[thick] ({\k*\gw + \l*\gs}, {\l*\gh + \g*\gsep}) -- ({\k*\gw + (\l+1)*\gs}, {(\l+1)*\gh + \g*\gsep});

						}

					\draw[thick, dashed] ({\k*\gw}, {0 + \g*\gsep}) -- ({\k*\gw+(-0.5)*\gw*\gs}, {(-0.5)*\gh + \g*\gsep});

					\draw[thick, dashed] ({\k*\gw + 3*\gs}, {3*\gh + \g*\gsep}) -- ({\k*\gw+(+0.5)*\gw*\gs + 3*\gs}, {(3.5)*\gh + \g*\gsep});

				}

			\foreach \k in {0,...,2} {
					\foreach \l in {0,2} {

							\fill [ultra thick, pattern=north west lines, pattern color=black, opacity=0.3]
							({\k*\gw + \l*\gs}, {\l*\gh + \g*\gsep}) -- ({(\k+1)*\gw + \l*\gs}, {\l*\gh + \g*\gsep}) -- ({(\k+1)*\gw + (\l+1)*\gs}, {(\l+1)*\gh + \g*\gsep}) -- ({(\k)*\gw + (\l+1)*\gs}, {(\l+1)*\gh + \g*\gsep}) -- cycle;

						}
				}
		}

	\node[] at ({0.5*\gw + 0.5*\gw*\gs}, {0.5*\gh}) {$\rho\imjm\nrm$};
	\node[] at ({1.5*\gw + 0.5*\gw*\gs}, {0.5*\gh}) {$\rho\ijm\nrm$};
	\node[] at ({2.5*\gw + 0.5*\gw*\gs}, {0.5*\gh}) {$\rho\ipjm\nrm$};

	\node[] at ({0.5*\gw + 1.5*\gw*\gs}, {1.5*\gh}) {$\rho\imj\nrm$};
	\node[] at ({1.5*\gw + 1.5*\gw*\gs}, {1.5*\gh}) {$\rho\ij\nrm$};
	\node[] at ({2.5*\gw + 1.5*\gw*\gs}, {1.5*\gh}) {$\rho\ipj\nrm$};

	\node[] at ({0.5*\gw + 2.5*\gw*\gs}, {2.5*\gh}) {$\rho\imjp\nrm$};
	\node[] at ({1.5*\gw + 2.5*\gw*\gs}, {2.5*\gh}) {$\rho\ijp\nrm$};
	\node[] at ({2.5*\gw + 2.5*\gw*\gs}, {2.5*\gh}) {$\rho\ipjp\nrm$};

	\node[] at ({0.5*\gw + 0.5*\gw*\gs}, {0.5*\gh + \gsep}) {$\rho\imjm\nr$};
	\node[] at ({1.5*\gw + 0.5*\gw*\gs}, {0.5*\gh + \gsep}) {$\rho\ijm\nr$};
	\node[] at ({2.5*\gw + 0.5*\gw*\gs}, {0.5*\gh + \gsep}) {$\rho\ipjm\nr$};

	\node[] at ({0.5*\gw + 1.5*\gw*\gs}, {1.5*\gh + \gsep}) {$\rho\imj\nr$};
	\node[] at ({1.5*\gw + 1.5*\gw*\gs}, {1.5*\gh + \gsep}) {$\rho\ij\nr$};
	\node[] at ({2.5*\gw + 1.5*\gw*\gs}, {1.5*\gh + \gsep}) {$\rho\ipj\nr$};

	\node[] at ({0.5*\gw + 2.5*\gw*\gs}, {2.5*\gh + \gsep}) {$\rho\imjp\nr$};
	\node[] at ({1.5*\gw + 2.5*\gw*\gs}, {2.5*\gh + \gsep}) {$\rho\ijp\nr$};
	\node[] at ({2.5*\gw + 2.5*\gw*\gs}, {2.5*\gh + \gsep}) {$\rho\ipjp\nr$};

	\draw[->] ({2.8*\gw + 2.8*\gw*\gs}, {2.8*\gh}) to [bend left=40] node[midway, left=-50pt, fill=white] {$j+1>r,\; \rho\ipjp\nr=\rho\ipjp\nrm$} ({2.8*\gw + 2.8*\gw*\gs}, {2.8*\gh + \gsep}) ;

	\draw[->] ({0.2*\gw + 0.2*\gw*\gs}, {0.2*\gh}) to [bend right=40] node[midway, right=-50pt, fill=white] {$j-1<r,\; \rho\imjm\nr=\rho\imjm\nrm$} ({0.2*\gw + 0.2*\gw*\gs}, {0.2*\gh + \gsep}) ;

	\draw[->] ({1.5*\gw + 1.2*\gw*\gs}, {1.2*\gh}) to [bend right=10] node[midway, fill=white] {$r=j,\quad \rho\ij\nr = \rho\ij\nrm - \frac{\Dt}{\Dx}\prt*{F\ihj\nr - F\imhj\nr}$} ({1.5*\gw + 1.8*\gw*\gs}, {1.8*\gh + \gsep}) ;

\end{tikzpicture} 		      \caption{Detail of the $r=j$ update of the sweeping dimensional splitting scheme. Only the $j=r$ row is updated; the remaining rows, corresponding to $j\neq r$, are left unchanged by this update. The convolution density is treated as in the one-dimensional setting in the $j=r$ row, using a suitable choice to guarantee energy dissipation: $\hat{\rho}\kl\nr = \rho\kl\nssr$; elsewhere, the most recent value of the density is used instead: $\hat{\rho}\kl\nr = \rho\kl\nrm$.}
		      \label{fig:sweeping}
	      \end{figure}
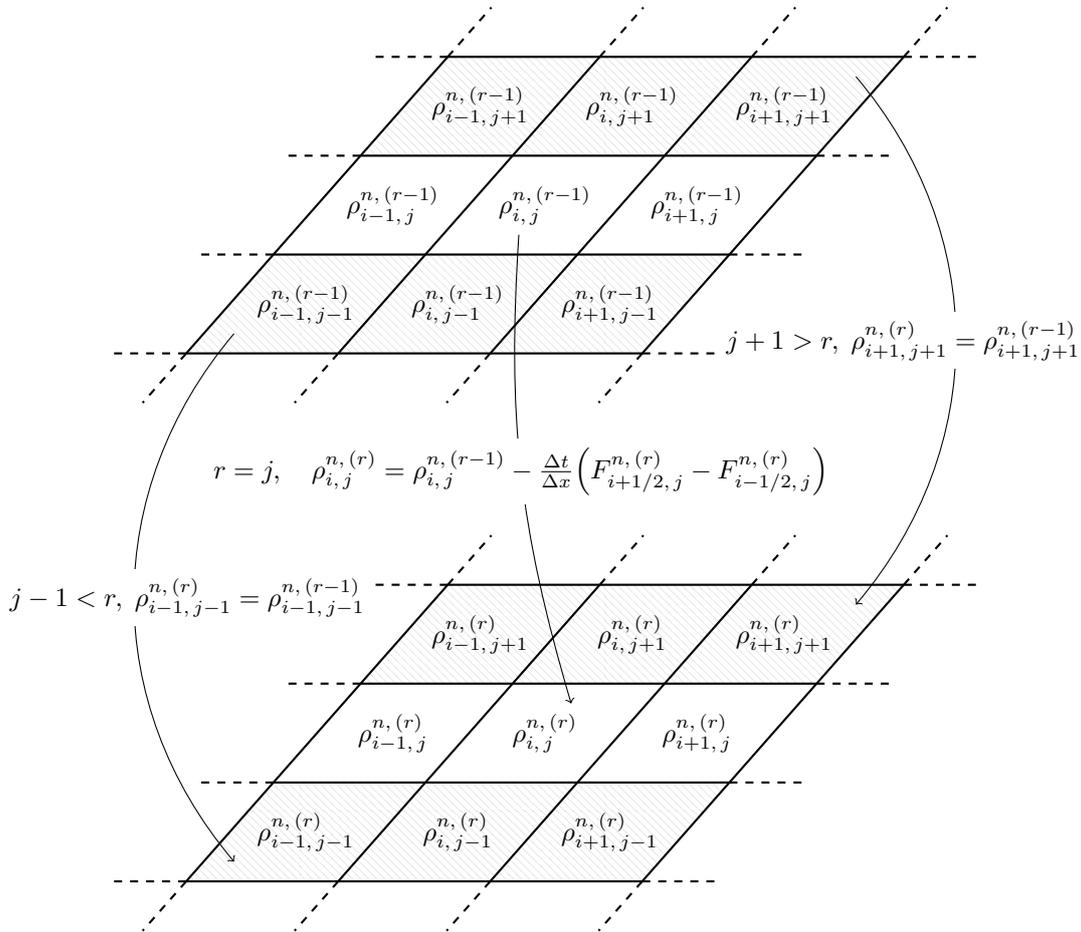

	\item \textbf{Evolution in the $y$-direction}

	      We now let $\rho\nhzero \coloneqq \rho\nh$ and define the analogous sequence
	      \begin{align}\label{eq:2DS1Y}
		      \rho\ij\nhr =
		      \begin{cases}\displaystyle
			      \rho\ij\nhrm - \frac{\Dt}{\Dy}\prt*{G\ijh\nhr - G\ijmh\nhr} & \text{if }i=r ,
			      \\ \rho\ij\nhrm & \text{otherwise} ;
		      \end{cases}
	      \end{align}
	      where
	      \begin{equation}\label{eq:2DS1Yterms}
		      \begin{aligned}
			      G\ijh\nhr & = \rho\ij\Nr \pos{v\ijh\nhr} + \rho\ijp\Sr \neg{v\ijh\nhr} ,
			      \\ v\ijh\nhr &= - \frac{\xi\ijp\nhr - \xi\ij\nhr}{\Dy} ,
			      \\ \xi\ij\nhr &= H'\prt{\rho\ij\nhr} + V\ij + \prt{W\conv\doublehat{\rho}\nhr}\ij ;
			      \\ \rho\ij\Nr &= \rho\ij\nhr + \frac{\Dy}{2}\ry\ij\nhr,
			      \qquad \rho\ij\Sr = \rho\ij\nhr - \frac{\Dy}{2}\ry\ij\nhr,
			      \\ \ry\nh\ij &= \minmod\prt*{
				      \theta \frac{\rho\ijp\nh - \rho\ij\nh}{\Dy}, \;
				      \frac{\rho\ijp\nh - \rho\ijm\nh}{2\Dy}, \;
				      \theta \frac{\rho\ij\nh - \rho\ijm\nh}{\Dy}
			      };
		      \end{aligned}
	      \end{equation}
	      and the discrete convolution is defined as in \cref{eq:2Ddiscreteconvolution}. This convolution variable is given by
	      \begin{align}
		      \doublehat{\rho}\kl\nhr =
		      \begin{cases}
			      \rho\kl\nsssr & \text{if }k = r ,
			      \\ \rho\kl\nhrm & \text{otherwise} ,
		      \end{cases}
	      \end{align}
	      where $\rho\nsssr$ is chosen as one of $\rho\nhr$, $\rho\nhrm$, or $\prt{\rho\nhr + \rho\nhrm}/{2}$.

	      The last term of the sequence defines the update: $\rho\np\coloneqq \rho\nhtwoM$.
\end{itemize}

\subsubsection{Positivity Preservation Property}

\begin{theorem}
	Scheme \eqref{eq:2DS1X}-\eqref{eq:2DS1Y} is positivity preserving provided the following CFL condition is satisfied:
	\begin{equation}\label{eq:S2CFLDSCoupled}
		\Dt \leq \frac{1}{2} \min\set*{
			\frac{\Dx}{
				\max\ij\set*{\pos{u\ihj\nj}, -\neg{u\ihj\nj}}
			}, \; \frac{\Dy}{
				\max\ij\set*{\pos{v\ijh\nhi}, -\neg{v\ijh\nhi}}
			}}.
	\end{equation}
\end{theorem}

\begin{proof}
	Using \cref{th:S1CFL}, we recover a sequence of CFL conditions for the positivity of each of the updates of \eqref{eq:2DS1X}:
	\begin{equation}
		\Dt\leq \frac{\Dx}{
			2\max\i\set*{\pos{u\ihj\nj}, -\neg{u\ihj\nj}}
		},\quad\textrm{for each }j.
	\end{equation}
	Similarly, the updates \eqref{eq:2DS1Y} require the sequence of conditions
	\begin{equation}
		\Dt\leq \frac{\Dx}{
			2\max\j\set*{\pos{v\ijh\nhi}, -\neg{v\ijh\nhi}}
		},\quad\textrm{for each }i.
	\end{equation}
	The combination of all of these conditions yields the result.
\end{proof}

\subsubsection{Energy Dissipation Property}

\begin{theorem}
	Under the CFL condition \eqref{eq:S2CFLDSCoupled}, scheme \eqref{eq:2DS1X}-\eqref{eq:2DS1Y} dissipates the discrete energy \eqref{eq:energyDiscrete2D}:
	\begin{align}
		E_\Delta\prt{\rho\np}\leq E_\Delta\prt{\rho\n} ,
	\end{align}
	for a suitable choice of convolution variables:
	\begin{enumerate}[(i)]
		\item $\rho\nssr = \rho\nr$ and $\rho\nsssr = \rho\nhr$ for a negative-definite potential $W$;
		\item $\rho\nssr = \rho\nrm$ and $\rho\nsssr = \rho\nhrm$ for a positive-definite potential $W$;
		\item $\rho\nssr = \prt{\rho\nr + \rho\nrm}/{2}$ and $\rho\nsssr = \prt{\rho\nhr + \rho\nhrm}/{2}$ for any $W$.
	\end{enumerate}
\end{theorem}

\begin{proof}
	Scheme \eqref{eq:2DS1X}, along the $j=r$ row, reads
	\begin{align}
		\rho\ir\nrm - \frac{\Dt}{\Dx}\prt*{F\ihr\nr - F\imhr\nr}.
	\end{align}
	Upon multiplication by $\xi\ij\nr$ and summation over $i$, this yields
	\begin{align}
		\sum_{i=1}^{2M} \xi\ir\nr \prt*{\rho\ir\nr - \rho\ir\nrm}
		 & = - \sum_{i=1}^{2M} \xi\ir\nr \prt*{F\ihr\nr - F\imhr\nr} \frac{\Dt}{\Dx} .
	\end{align}
	Substituting $\xi\ir\nr$ as defined in \cref{eq:2DS1Xterms}, we obtain the identity
	\begin{align}
		\sum_{i=1}^{2M} V\ir\nr \prt*{\rho\ir\nr - \rho\ir\nrm} & = - \sum_{i=1}^{2M} \xi\ir\nr \prt*{F\ihr\nr - F\imhr\nr} \frac{\Dt}{\Dx}                                 \\
		                                                        & \quad - \sum_{i=1}^{2M} H'\prt{\rho\ir\nr} \prt*{\rho\ir\nr - \rho\ir\nrm}                                \\
		                                                        & \quad - \sum_{i=1}^{2M} \sum_{k,l=1}^{2M} \Wikrl \prt*{\rho\ir\nr - \rho\ir\nrm} \hat{\rho}\kl\nr \Dx\Dy.
		\label{eq:2Dpotentialidentity}
	\end{align}
	Considering now the definition of the discrete energy, we compute the difference
	\begin{align}
		\lll E_\Delta\prt{\rho\nr} - E_\Delta\prt{\rho\nrm}                                                                                                                  \\
		 & = \sum_{i,j=1}^{2M} \prt*{H\prt{\rho\ij\nr} - H\prt{\rho\ij\nrm}} \Delta x\Delta y \quad + \sum_{i,j=1}^{2M} \prt*{\rho\ij\nr - \rho\ij\nrm}V\ij \Delta x\Delta y \\
		 & \quad + \sum_{i,j=1}^{2M} \sum_{k,l=1}^{2M} \Wikjl \prt*{\rho\ij\nr \rho\kl\nr - \rho\ij\nrm \rho\kl\nrm} \frac{\prt{\Delta x\Delta y}^2}{2},
	\end{align}
	using the fact that $\rho\ij\nr = \rho\ij\nrm$ whenever $j\neq r$. Substituting \cref{eq:2Dpotentialidentity}, the difference becomes
	\begin{align}
		\lll E_\Delta\prt{\rho\nr} - E_\Delta\prt{\rho\nrm}                                                                                             \\
		 & = \sum_{i=1}^{2M} \prt*{H\prt{\rho\ir\nr} - H\prt{\rho\ir\nrm} - H'\prt{\rho\ir\nr} \prt{\rho\ir\nr - \rho\ir\nrm}}                          \\
		 & \quad + \sum_{i,j=1}^{2M} \sum_{k,l=1}^{2M} \Wikjl \prt*{\rho\ij\nr \rho\kl\nr - \rho\ij\nrm \rho\kl\nrm} \frac{\prt{\Delta x\Delta y}^2}{2} \\
		 & \quad - \sum_{i=1}^{2M} \sum_{k,l=1}^{2M} \Wikrl \prt*{\rho\ir\nr - \rho\ir\nrm} \hat{\rho}\kl\nr \prt{\Delta x\Delta y}^2                   \\
		 & \quad - \sum_{i=1}^{2M} \xi\ir\nr \prt*{F\ihr\nr - F\imhr\nr} \Dt\Dy                                                                         \\
		 & = I + II + III,
	\end{align}
	where
	\begin{align}
		I   & \coloneqq \sum_{i=1}^{2M} \prt*{H\prt{\rho\ir\nr} - H\prt{\rho\ir\nrm} - H'\prt{\rho\ir\nr} \prt{\rho\ir\nr - \rho\ir\nrm}}                    \\
		II  & \coloneqq \sum_{i,j=1}^{2M} \sum_{k,l=1}^{2M} \Wikjl \prt*{\rho\ij\nr \rho\kl\nr - \rho\ij\nrm \rho\kl\nrm} \frac{\prt{\Delta x\Delta y}^2}{2} \\
		    & \quad - \sum_{i=1}^{2M} \sum_{k,l=1}^{2M} \Wikrl \prt*{\rho\ir\nr - \rho\ir\nrm} \hat{\rho}\kl\nr \prt{\Delta x\Delta y}^2                     \\
		III & \coloneqq - \sum_{i=1}^{2M} \xi\ir\nr \prt*{F\ihr\nr - F\imhr\nr} \Dt\Dy.
	\end{align}

	The first term, as in the one-dimensional case, is bounded using the convexity of $H$:
	\begin{align}
		H\prt{\rho\ir\nr} - H\prt{\rho\ir\nrm} - H'\prt{\rho\ir\nr} \prt{\rho\ir\nr - \rho\ir\nrm} \leq 0,
	\end{align}
	hence $I\leq 0$. The second term:
	\begin{align}
		\frac{II}{\curlyC}
		 & = \sum_{i,j=1}^{2M} \sum_{k,l=1}^{2M} \Wikjl \prt*{\rho\ij\nr \rho\kl\nr - \rho\ij\nrm \rho\kl\nrm}                                                          \\
		 & \quad - 2 \sum_{i=1}^{2M} \sum_{k,l=1}^{2M} \Wikrl \prt*{\rho\ir\nr - \rho\ir\nrm} \hat{\rho}\kl\nr                                                          \\
		 & = \sum_{\substack{i,j=1                                                                                                                                      \\ j\neq r}}^{2M} \sum_{k,l=1}^{2M} \Wikjl \rho\ij\nrm \prt*{\rho\kl\nr - \rho\kl\nrm} \\
		 & \quad + \sum_{i=1}^{2M} \sum_{k,l=1}^{2M} \Wikrl \prt*{\rho\ir\nr \rho\kl\nr - \rho\ir\nrm \rho\kl\nrm - 2 \prt{\rho\ir\nr - \rho\ir\nrm} \hat{\rho}\kl\nr},
	\end{align}
	having used $\rho\ij\nr = \rho\ij\nrm$ for $j\neq r$ again, and where $\curlyC =  \prt{\Delta x\Delta y}^2/2$. Splitting the sums further:
	\begin{align}
		\frac{II}{\curlyC} & = \sum_{\substack{i,j=1                                                                                                                                    \\ j\neq r}}^{2M} \sum_{k}^{2M} \Wikjr \rho\ij\nrm \prt*{\rho\kr\nr - \rho\kr\nrm} \\
		                   & \quad + \sum_{i=1}^{2M} \sum_{\substack{k,l=1                                                                                                              \\ l\neq r}}^{2M} \Wikrl \prt*{\prt{\rho\ir\nr - \rho\ir\nrm} \rho\kl\nrm - 2 \prt{\rho\ir\nr - \rho\ir\nrm} \hat{\rho}\kl\nr} \\
		                   & \quad + \sum_{i=1}^{2M} \sum_{k=1}^{2M} \Wikrr \prt*{\rho\ir\nr \rho\kr\nr - \rho\ir\nrm \rho\kr\nrm - 2 \prt{\rho\ir\nr - \rho\ir\nrm} \hat{\rho}\kr\nr}  \\
		                   & = 2 \sum_{i=1}^{2M} \sum_{\substack{k,l=1                                                                                                                  \\ l\neq r}}^{2M} \Wikrl \prt*{\prt{\rho\ir\nr - \rho\ir\nrm} \rho\kl\nrm - \prt{\rho\ir\nr - \rho\ir\nrm} \hat{\rho}\kl\nr} \\
		                   & \quad + \sum_{i=1}^{2M} \sum_{k=1}^{2M} \Wikrr \prt*{\rho\ir\nr \rho\kr\nr - \rho\ir\nrm \rho\kr\nrm - 2 \prt{\rho\ir\nr - \rho\ir\nrm} \hat{\rho}\kr\nr},
	\end{align}
	by swapping $i$ with $k$, taking $j=l$ in the first sum, and using the symmetry of the interaction potential: $ \Wkirl = \Wikrl$. Substituting the definition of the convolution variable, \cref{eq:2DXconvolutionvariable}, the expression reduces to
	\begin{align}
		\sum_{i=1}^{2M} \sum_{k=1}^{2M} \Wikrr \prt*{\rho\ir\nr \rho\kr\nr - \rho\ir\nrm \rho\kr\nrm - 2 \prt{\rho\ir\nr - \rho\ir\nrm} \rho\kl\nssr},
	\end{align}
	because the first sum is identically zero. The recovered summand is analogous to the one which appears in the one-dimensional energy dissipation result: it is controlled through the choice of $\rho\nssr$. Selecting $\rho\nssr = \rho\nr$ (respectively $\rho\nssr = \rho\nrm$) for a negative-definite (resp. positive-definite) potential $W$ shows $II\leq 0$; letting $\rho\nssr = \prt{\rho\nr + \rho\nrm}/{2}$ instead yields $II\equiv 0$, regardless of $W$.

	The discrete energy difference reduces thus:
	\begin{align}
		\lll E_\Delta\prt{\rho\nr} - E_\Delta\prt{\rho\nrm}                                                       \\
		 & \leq - \sum_{i=1}^{2M} \xi\ir\nr \prt*{F\ihr\nr - F\imhr\nr} \Dt\Dy                                    \\
		 & = - \sum_{i=1}^{2M-1} F\ihr\nr \prt*{\xi\ir\nr - \xi\ipr\nr} \Dt\Dy                                    \\
		 & = - \sum_{i=1}^{2M-1} F\ihr\nr u\ihr\nr \Dt\Dx\Dy                                                      \\
		 & = - \sum_{i=1}^{2M-1} \prt*{\rho\ir\Er \pos{u\ihj\nr} + \rho\ipr\Wr \neg{u\ihr\nr}} u\ihr\nr \Dt\Dx\Dy \\
		 & \leq - \sum_{i=1}^{2M-1} \min\prt*{\rho\ir\Er, \rho\ipr\Wr} \abs{u\ihr\nr}^2 \Dt\Dx\Dy \leq 0.
	\end{align}
	A similar argument yields an estimate for the update in the $y$-direction:
	\begin{align}
		\lll E_\Delta\prt{\rho\nhr} - E_\Delta\prt{\rho\nhrm}                                              \\
		 & \leq - \sum_{j=1}^{2M-1} \min\prt*{\rho\rj\Nr, \rho\rjp\Sr} \abs{v\rjh\nhr}^2 \Dt\Dx\Dy \leq 0.
	\end{align}
	Finally, we observe
	\begin{align}
		\lll E_\Delta\prt{\rho\np} - E_\Delta\prt{\rho\n}                                                                                                 \\
		 & = \prt*{E_\Delta\prt{\rho\np} - E_\Delta\prt{\rho\nh}} + \prt*{E_\Delta\prt{\rho\nh} - E_\Delta\prt{\rho\n}}                                   \\
		 & = \sum_{r=1}^{2M} \brk*{\prt{E_\Delta\prt{\rho\nhr} - E_\Delta\prt{\rho\nhrm}} + \prt{E_\Delta\prt{\rho\nr} - E_\Delta\prt{\rho\nrm}}} \leq 0,
	\end{align}
	because all the summands are non-positive, concluding the proof.
\end{proof}

\subsection{Scheme 2 (S2)}

Assume $\rho\n$ is the approximate solution at time $t\n$. The scheme is again defined through two sequences of individual row-by-row and column-by-column iterations.
\begin{itemize}
	\item \textbf{Evolution in the $x$-direction}

	      We let $\rho\nzero \coloneqq \rho\n$ and define, for $1\leq i, j, r\leq 2M$, the implicit sequence
	      \begin{align}\label{eq:2DS2X}
		      \rho\ij\nr =
		      \begin{cases}\displaystyle
			      \rho\ij\nrm - \frac{\Dt}{\Dx}\prt*{F\ihj\nr - F\imhj\nr} & \text{if }j=r ,
			      \\ \rho\ij\nrm & \text{otherwise} ;
		      \end{cases}
	      \end{align}
	      where
	      \begin{equation}\label{eq:2DS2Xterms}
		      \begin{aligned}
			      F\ihj\nr & = \rho\ij\nr \pos{u\ihj\nr} + \rho\ipj\nr \neg{u\ihj\nr} ,
			      \\ u\ihj\nr &= - \frac{\xi\ipj\nr - \xi\ij\nr}{\Dx} ,
			      \\ \xi\ij\nr &= H'\prt{\rho\ij\nr} + V\ij + \prt{W\conv\hat{\rho}\nr}\ij ;
		      \end{aligned}
	      \end{equation}
	      and the discrete convolution is defined by the sum
	      \begin{align}\label{eq:2DS2discreteconvolution}
		      \prt{W\conv\hat{\rho}\nr}\ij =
		      \sum_{k,l=1}^{2M} \Wikjl \hat{\rho}\kl\nr \Dx\Dy .
	      \end{align}
	      Furthermore, the convolution variable is given by
	      \begin{align}\label{eq:2DS2Xconvolutionvariable}
		      \hat{\rho}\kl\nr =
		      \begin{cases}
			      \rho\kl\nssr & \text{if }l = r ,
			      \\ \rho\kl\nrm & \text{otherwise} ,
		      \end{cases}
	      \end{align}
	      where once again the choice of $\rho^{**}$ depends on the properties of the interaction potential $W$: throughout the scheme, $\rho\nssr$ is chosen as one of $\rho\nr$, $\rho\nrm$, or $\prt{\rho\nr + \rho\nrm}/{2}$.

	      The last term of the sequence defines the semi-update $\rho\nh\coloneqq \rho\ntwoM$.

	\item \textbf{Evolution in the $y$-direction}

	      We now let $\rho\nhzero \coloneqq \rho\nh$ and define the analogous sequence
	      \begin{align}\label{eq:2DS2Y}
		      \rho\ij\nhr =
		      \begin{cases}\displaystyle
			      \rho\ij\nhrm - \frac{\Dt}{\Dy}\prt*{G\ijh\nhr - G\ijmh\nhr} & \text{if }i=r ,
			      \\ \rho\ij\nhrm & \text{otherwise} ;
		      \end{cases}
	      \end{align}
	      where
	      \begin{equation}\label{eq:2DS2Yterms}
		      \begin{aligned}
			      G\ijh\nhr & = \rho\ij\nhr \pos{v\ijh\nhr} + \rho\ijp\nhr \neg{v\ijh\nhr} ,
			      \\ v\ijh\nhr &= - \frac{\xi\ijp\nhr - \xi\ij\nhr}{\Dy} ,
			      \\ \xi\ij\nhr &= H'\prt{\rho\ij\nhr} + V\ij + \prt{W\conv\doublehat{\rho}\nhr}\ij ;
		      \end{aligned}
	      \end{equation}
	      and the discrete convolution is defined as in \cref{eq:2DS2discreteconvolution}. This convolution variable is given by
	      \begin{align}
		      \doublehat{\rho}\kl\nhr =
		      \begin{cases}
			      \rho\kl\nsssr & \text{if }k = r ,
			      \\ \rho\kl\nhrm & \text{otherwise} ,
		      \end{cases}
	      \end{align}
	      where $\rho\nsssr$ is chosen as one of $\rho\nhr$, $\rho\nhrm$, or $\prt{\rho\nhr + \rho\nhrm}/{2}$.

	      The last term of the sequence defines the update: $\rho\np\coloneqq \rho\nhtwoM$.
\end{itemize}

\subsubsection{Positivity Preservation Property}

\begin{theorem}
	Scheme \eqref{eq:2DS2X}-\eqref{eq:2DS2Y} is unconditionally positivity-preserving: if $\rho\i\n\geq 0$ for all $i$, then $\rho\i\np\geq 0$ for all $i$.
\end{theorem}

\begin{proof}
	Just as in the proof of positivity for S1, we apply the one-dimensional result to each update. Invoking \cref{th:S2CFL} at every step, we learn that $\rho\i\n\geq 0$ for all $i$ implies $\rho\i\np\geq 0$ for all $i$.
\end{proof}

\subsubsection{Energy Dissipation Property}

\begin{theorem}
	Scheme \eqref{eq:2DS2X}-\eqref{eq:2DS2Y} unconditionally dissipates the discrete energy \eqref{eq:energyDiscrete2D}:
	\begin{align}
		E_\Delta\prt{\rho\np}\leq E_\Delta\prt{\rho\n} ,
	\end{align}
	for a choice of convolution variables:
	\begin{enumerate}[(i)]
		\item $\rho\nssr = \rho\nr$ and $\rho\nsssr = \rho\nhr$ for a negative-definite potential $W$;
		\item $\rho\nssr = \rho\nrm$ and $\rho\nsssr = \rho\nhrm$ for a positive-definite potential $W$;
		\item $\rho\nssr = \prt{\rho\nr + \rho\nrm}/{2}$ and $\rho\nsssr = \prt{\rho\nhr + \rho\nhrm}/{2}$ for any $W$.
	\end{enumerate}
\end{theorem}

\begin{proof}
	The proof of the result in the previous section carries over except for the last part:
	\begin{align}
		\lll E_\Delta\prt{\rho\nr} - E_\Delta\prt{\rho\nrm}                                                       \\
		 & \leq - \sum_{i=1}^{2M} \xi\ir\nr \prt*{F\ihr\nr - F\imhr\nr} \Dt\Dy                                    \\
		 & = - \sum_{i=1}^{2M-1} F\ihr\nr \prt*{\xi\ir\nr - \xi\ipr\nr} \Dt\Dy                                    \\
		 & = - \sum_{i=1}^{2M-1} F\ihr\nr u\ihr\nr \Dt\Dx\Dy                                                      \\
		 & = - \sum_{i=1}^{2M-1} \prt*{\rho\ir\nr \pos{u\ihj\nr} + \rho\ipr\nr \neg{u\ihr\nr}} u\ihr\nr \Dt\Dx\Dy \\
		 & \leq - \sum_{i=1}^{2M-1} \min\prt*{\rho\ir\nr, \rho\ipr\nr} \abs{u\ihr\nr}^2 \Dt\Dx\Dy \leq 0.
	\end{align}
	A similar argument yields an estimate for the update in the $y$-direction:
	\begin{align}
		\lll E_\Delta\prt{\rho\nhr} - E_\Delta\prt{\rho\nhrm}                                                \\
		 & \leq - \sum_{j=1}^{2M-1} \min\prt*{\rho\rj\nhr, \rho\rjp\nhr} \abs{v\rjh\nhr}^2 \Dt\Dx\Dy \leq 0.
	\end{align}
	Finally, we observe
	\begin{align}
		\lll E_\Delta\prt{\rho\np} - E_\Delta\prt{\rho\n}                                                                                                 \\
		 & = \prt*{E_\Delta\prt{\rho\np} - E_\Delta\prt{\rho\nh}} + \prt*{E_\Delta\prt{\rho\nh} - E_\Delta\prt{\rho\n}}                                   \\
		 & = \sum_{r=1}^{2M} \brk*{\prt{E_\Delta\prt{\rho\nhr} - E_\Delta\prt{\rho\nhrm}} + \prt{E_\Delta\prt{\rho\nr} - E_\Delta\prt{\rho\nrm}}} \leq 0,
	\end{align}
	because all the summands are non-positive, concluding the proof.
\end{proof}
 \section{Implementation, Validation and Accuracy of the Schemes}\label{sec:validation}

The following section is concerned with the implementation of the numerical schemes, as well as their validation against equations whose analytical solutions are known. To begin, we validate the order of the schemes by solving the heat and porous medium equations, as well as a non-local analogue of the Fokker-Planck equation, and studying the error against their solutions. Furthermore, we validate the order of convergence to a stationary state on non-linear and non-local Fokker-Planck equations by comparing them against the known convergence rates.

We recall the essential properties of the schemes:
\begin{center}
	\renewcommand*{\arraystretch}{1.3}
	\begin{tabular}{C{0.27\textwidth}C{0.44\textwidth}C{0.19\textwidth}}
		                        & {\scshape Scheme 1 (S1)}                                                                    & {\scshape Scheme 2 (S2)} \\
		\toprule
		Order in time           & First                                                                                       & First                    \\
		Order in space          & Second                                                                                      & First                    \\
		\midrule
		Positivity-preservation & \multirow{2}{*}{$\displaystyle \Dt\leq \frac{\Dx}{2\max\i\set*{(u\ih\np)^+,-(u\ih\np)^-}}$}
		                        & \multirow{2}{*}{Unconditional}                                                                                         \\
		Energy-dissipation      &                                                                                             &                          \\
		\bottomrule
	\end{tabular}
\end{center}

\subsection{Implementation --- A Note on Solutions with Vacuum}

The numerical schemes were implemented using the Julia language \cite{B.E.K+2017}. \revision{The implicit-in-time formulation of S1 and S2 requires the approximation of the solution $\rho\np$ to \eqref{eq:S1} or \eqref{eq:S2}, for which we employ the Newton-Raphson method provided by the \textit{NLsolve} library \cite{NLsolve}}. Throughout Sections \ref{sec:validation} and \ref{sec:experiments} we make the choice $\rho\nss=\prt{\rho\n+\rho\np}/{2}$, which guarantees the dissipation of the discrete energy regardless of the choice of $W$.

Due to the nature of the schemes, special care should be taken when dealing with problems where vacuum is present. While the schemes perform satisfactorily in cases where parts of the solution take arbitrarily small values (the heat equation in Section \ref{sec:heatvalidation}, for instance), they sometimes fail with solutions involving segments which take the exact value of zero. Examples of this are problems with compactly supported initial datum such as the Barenblatt solution for the porous medium equation.

An explicit calculation of the Jacobian matrix of the schemes employed by the Newton solver reveals that certain terms can become ill-posed when $\rho\i=0$ in some cells. In particular, terms involving the partial derivative of $u\ih\np$ with respect to $\rho_{j}\np$ result in the second derivative of the internal energy density, $H''\prt{\rho}$, which can be singular. For instance, this is proportional to $\rho^{m-2}$ for the heat equation or the porous medium equation; under S1, the range $1\leq m<2$ is problematic, whereas S2 handles all cases except $m=1$. \revision{The issue can be easily circumvented by modifying the energy term to include an offset: $H\prt{\hat{\rho}}$, where $\hat{\rho} = \max\prt{\rho, \epsilon}$ and $\epsilon$ is the \textit{machine epsilon}}.

\subsection{heat equation}\label{sec:heatvalidation}

The first validation case is the heat equation $\pt\rho=D\Delta\rho$, i.e.,
\begin{equation}\label{eq:heat}
	H(\rho)=D\prt{\rho\log(\rho)-\rho},\quad
	V(\bx)=0,\quad
	W(\bx)=0,
\end{equation}
for $D>0$. The analytical solution $\rho^*(t,\bx)$ corresponding to a point source is given by the heat kernel
\begin{equation}\label{eq:heatkernel}
	\Phi\prt{t,\bx}=
	\prt{4\pi Dt}^{-\frac{n}{2}}
	\exp\prt*{-\frac{\abs{\bx}^2}{4Dt}}.
\end{equation}

We will solve \eqref{eq:heat} numerically for $D=1$ with initial datum $\rho_0\prt{\bx}=\Phi\prt{2.0,\bx}$ through an interval of time of unit length for various choices of $\Dx$. We will compute the $L_1$ error of the numerical solution $\rho_{\Dx}$ at the final time,
\begin{equation}
	\textrm{Error}(\Dx)=\left\|\rho_{\Dx}(\tfin,\bx)-\rho^*(\tfin,\bx)\right\|_{L_1}.
\end{equation}
The error will then be used to estimate the order of convergence of the scheme
\begin{equation}
	\textrm{Order}(\Dx)=\log_2(\textrm{Error}(2\Dx)/\textrm{Error}(\Dx)).
\end{equation}

The choice of time step is $\Dt=c\Dx$ for the S2 validation, but $\Dt=c\Dx^2$ instead for the S1 validation in order to show second-order convergence in space. The results for the S1 scheme can be found on \cref{tab:OOCHeatEquationSBCH1,tab:OOCHeatEquationSBCH2} for one and two dimensions respectively. The results for S2 follow on \cref{tab:OOCHeatEquationIBCH1,tab:OOCHeatEquationIBCH2}. Good approximations to orders $2$ and $1$ can be seen for S1 and S2 respectively.

\begin{table}[H]
	\centering
	\begin{tabularx}{0.88\textwidth}{XXXX}
		\toprule
		$\Dt$     & $\Dx$    & Error        & Order        \\
		\midrule
		$2^{-4}$  & $2^{-1}$ & 0.0042109083 & ---          \\
		$2^{-6}$  & $2^{-2}$ & 0.0010515212 & 2.0016534660 \\
		$2^{-8}$  & $2^{-3}$ & 0.0002646653 & 1.9902368023 \\
		$2^{-10}$ & $2^{-4}$ & 0.0000662580 & 1.9980028628 \\
		$2^{-12}$ & $2^{-5}$ & 0.0000165759 & 1.9990085352 \\
		$2^{-14}$ & $2^{-6}$ & 0.0000041459 & 1.9993314969 \\
		\bottomrule
	\end{tabularx}
	\caption{Errors and orders of convergence for the solution to the \textbf{heat equation} \eqref{eq:heat} in one dimension with \textbf{S1}. $D=1.0$, $\tini = 2.0$, $\tfin = 3.0$, $L = 15.0$.}
	\label{tab:OOCHeatEquationSBCH1}
\end{table}
 \begin{table}[H]
	\centering
	\begin{tabularx}{0.89\textwidth}{XXXXX}
		\toprule
		$\Dt$     & $\Dx$    & $\Dy$    & Error        & Order        \\
		\midrule
		$2^{-9}$  & $2^{-1}$ & $2^{-1}$ & 0.0289894915 & ---          \\
		$2^{-11}$ & $2^{-2}$ & $2^{-2}$ & 0.0073328480 & 1.9830844898 \\
		$2^{-13}$ & $2^{-3}$ & $2^{-3}$ & 0.0018680584 & 1.9715457943 \\
		$2^{-15}$ & $2^{-4}$ & $2^{-4}$ & 0.0004731995 & 1.9803261620 \\
		\bottomrule
	\end{tabularx}
	\caption{Errors and orders of convergence for the solution to the \textbf{heat equation} \eqref{eq:heat} in two dimensions with \textbf{S1}. $D=1.0$, $\tini = 2.0$, $\tfin = 3.0$, $L = 15.0$.}
	\label{tab:OOCHeatEquationSBCH2}
\end{table}
 
\begin{table}[H]
	\centering
	\begin{tabularx}{0.88\textwidth}{XXXX}
		\toprule
		$\Dt$    & $\Dx$    & Error        & Order        \\
		\midrule
		$2^{-1}$ & $2^{-1}$ & 0.0206792591 & ---          \\
		$2^{-2}$ & $2^{-2}$ & 0.0108726916 & 0.9274753681 \\
		$2^{-3}$ & $2^{-3}$ & 0.0056016868 & 0.9567759114 \\
		$2^{-4}$ & $2^{-4}$ & 0.0028449428 & 0.9774616438 \\
		$2^{-5}$ & $2^{-5}$ & 0.0014335008 & 0.9888569712 \\
		$2^{-6}$ & $2^{-6}$ & 0.0007196730 & 0.9941293288 \\
		\bottomrule
	\end{tabularx}
	\caption{Errors and orders of convergence for the solution to the \textbf{heat equation} \eqref{eq:heat} in one dimension with \textbf{S2}. $D=1.0$, $\tini = 2.0$, $\tfin = 3.0$, $L = 15.0$.}
	\label{tab:OOCHeatEquationIBCH1}
\end{table}
 \begin{table}[H]
	\centering
	\begin{tabularx}{0.89\textwidth}{XXXXX}
		\toprule
		$\Dt$    & $\Dx$    & $\Dy$    & Error        & Order         \\
		\midrule
		$2^{-1}$ & $2^{-1}$ & $2^{-1}$ & 0.0519120967 & ---           \\
		$2^{-2}$ & $2^{-2}$ & $2^{-2}$ & 0.0219888447 & 1.2392989536  \\
		$2^{-3}$ & $2^{-3}$ & $2^{-3}$ & 0.0099177360 & 1.14868907849 \\
		$2^{-4}$ & $2^{-4}$ & $2^{-4}$ & 0.0049797156 & 0.99394747341 \\
		\bottomrule
	\end{tabularx}
	\caption{Errors and orders of convergence for the solution to the \textbf{heat equation} \eqref{eq:heat} in two dimensions with \textbf{S2}. $D=1.0$, $\tini = 2.0$, $\tfin = 3.0$, $L = 15.0$.}
	\label{tab:OOCHeatEquationIBCH2}
\end{table}

\subsection{porous medium equation}\label{sec:PMEvalidation}

To validate a non-linear diffusion setting, we will now consider the porous medium equation $\pt\rho=D\Delta \rho^m$, i.e.,
\begin{equation}\label{eq:porousmedium}
	H(\rho)=\frac{D}{m-1}\rho^m,\quad
	V(\bx)=0,\quad
	W(\bx)=0,
\end{equation}
for $D>0,m>1.$ The Barenblatt solution $\rho^*(t,\bx)$ corresponding to a point source is given by
\begin{equation}\label{eq:PMEkernel}
	\Psi\prt{t,\bx}=
	\frac{1}{t^\alpha}\psi\prt*{\frac{\abs{\bx}}{t^\beta}},
\end{equation}
where $\psi\prt{\xi}=\prt{K-\kappa\xi^2}_{+}^{1/\prt{m-1}}$ for $\alpha=n/\prt{n(m-1)+2}$, $\beta=\alpha/n$, $\gamma=1/(m-1)+n/2$, $\kappa=\beta(m-1)/(2Dm)$
and $\prt{\cdot}_{+}=\max\set{\cdot,0}$. The normalisation constant $K>0$ is related to the total mass $M$ by $M=a\prt{m,n}K^\gamma$, see \cite[Section 17.5]{Vazquez2006}, where
\begin{equation}
	a\prt{m,n}=
	\prt*{
		\frac{\pi\prt{2Dmn}}{\alpha\prt{m-1}}
	}^{\frac{n}{2}}
	\frac{
		\Gamma\prt*{\frac{m}{m-1}}
	}{
		\Gamma\prt*{\frac{m}{m-1}+\frac{n}{2}}
	}
	,
\end{equation}
and $\Gamma$ is the Gamma function.

As before, we will solve \eqref{eq:porousmedium} numerically for $D=1$ with initial datum $\rho_0\prt{\bx}=\Psi\prt{2.0,\bx}$ and estimate the order of the scheme. \crefrange{tab:OOCPorousMediumEquationSBCH11.5}{tab:OOCPorousMediumEquationIBCH23.0} correspond to the cases $m=3/2$, $m=2$ and $m=3$ for the schemes in one and two dimensions. The approximation to the correct orders is fine for $m=3/2$ but worsens for increasing values of $m$ when compared to section \ref{sec:heatvalidation}. This phenomenon is well known in the numerical literature for non-linear diffusion, as the Barenblatt solution and compactly supported solutions in general lose regularity with increasing exponents, see for instance \cite{C.C.H2015}.

\begin{table}[H]
	\centering
	\begin{tabularx}{0.88\textwidth}{XXXX}
		\toprule
		$\Dt$     & $\Dx$    & Error        & Order        \\
		\midrule
		$2^{-2}$  & $2^{-1}$ & 0.0104485202 & ---          \\
		$2^{-4}$  & $2^{-2}$ & 0.0029208065 & 1.8388599282 \\
		$2^{-6}$  & $2^{-3}$ & 0.0007686238 & 1.9260171509 \\
		$2^{-8}$  & $2^{-4}$ & 0.0001964728 & 1.9679481374 \\
		$2^{-10}$ & $2^{-5}$ & 0.0000496005 & 1.9859042948 \\
		$2^{-12}$ & $2^{-6}$ & 0.0000124637 & 1.9926249978 \\
		\bottomrule
	\end{tabularx}
	\caption{Errors and orders of convergence for the solution to the \textbf{porous medium equation} \eqref{eq:porousmedium} with exponent $3/2$ in one dimension with \textbf{S1}. $D=1.0$, $m=1.5$, $\tini = 2.0$, $\tfin = 3.0$, $L = 6.0$.}
	\label{tab:OOCPorousMediumEquationSBCH11.5}
\end{table}
 \begin{table}[H]
	\centering
	\begin{tabularx}{0.88\textwidth}{XXXXX}
		\toprule
		$\Dt$    & $\Dx$    & $\Dy$    & Error        & Order        \\
		\midrule
		$2^{-2}$ & $2^{-1}$ & $2^{-1}$ & 0.0907205567 & ---          \\
		$2^{-4}$ & $2^{-2}$ & $2^{-2}$ & 0.0224679432 & 2.0135614318 \\
		$2^{-6}$ & $2^{-3}$ & $2^{-3}$ & 0.0055641932 & 2.0136236322 \\
		$2^{-8}$ & $2^{-4}$ & $2^{-4}$ & 0.0013852705 & 2.0060048023 \\
		\bottomrule
	\end{tabularx}
	\caption{Errors and orders of convergence for the solution to the \textbf{porous medium equation} \eqref{eq:porousmedium} with exponent $3/2$ in two dimensions with \textbf{S1}. $D=1.0$, $m=1.5$, $\tini = 2.0$, $\tfin = 3.0$, $L = 6.0$.}
	\label{tab:OOCPorousMediumEquationSBCH21.5}
\end{table}
 
\begin{table}[H]
	\centering
	\begin{tabularx}{0.88\textwidth}{XXXX}
		\toprule
		$\Dt$    & $\Dx$    & Error        & Order        \\
		\midrule
		$2^{-1}$ & $2^{-1}$ & 0.0272797400 & ---          \\
		$2^{-2}$ & $2^{-2}$ & 0.0152054561 & 0.8432408083 \\
		$2^{-3}$ & $2^{-3}$ & 0.0081299605 & 0.9032688429 \\
		$2^{-4}$ & $2^{-4}$ & 0.0042207082 & 0.9457632376 \\
		$2^{-5}$ & $2^{-5}$ & 0.0021528301 & 0.9712506201 \\
		$2^{-6}$ & $2^{-6}$ & 0.0010876434 & 0.9850288966 \\
		\bottomrule
	\end{tabularx}
	\caption{Errors and orders of convergence for the solution to the \textbf{porous medium equation} \eqref{eq:porousmedium} with exponent $3/2$ in one dimension with \textbf{S2}. $D=1.0$, $m=1.5$, $\tini = 2.0$, $\tfin = 3.0$, $L = 6.0$.}
	\label{tab:OOCPorousMediumEquationIBCH11.5}
\end{table}
 \begin{table}[H]
	\centering
	\begin{tabularx}{0.88\textwidth}{XXXXX}
		\toprule
		$\Dt$    & $\Dx$    & $\Dy$    & Error        & Order        \\
		\midrule
		$2^{-1}$ & $2^{-1}$ & $2^{-1}$ & 0.1369706965 & ---          \\
		$2^{-2}$ & $2^{-2}$ & $2^{-2}$ & 0.0511753063 & 1.4203475402 \\
		$2^{-3}$ & $2^{-3}$ & $2^{-3}$ & 0.0210733327 & 1.2800293428 \\
		$2^{-4}$ & $2^{-4}$ & $2^{-4}$ & 0.0093923443 & 1.1658612861 \\
		\bottomrule
	\end{tabularx}
	\caption{Errors and orders of convergence for the solution to the \textbf{porous medium equation} \eqref{eq:porousmedium} with exponent $3/2$ in two dimensions with \textbf{S2}. $D=1.0$, $m=1.5$, $\tini = 2.0$, $\tfin = 3.0$, $L = 6.0$.}
	\label{tab:OOCPorousMediumEquationIBCH21.5}
\end{table}
 
\begin{table}[H]
	\centering
	\begin{tabularx}{0.88\textwidth}{XXXX}
		\toprule
		$\Dt$     & $\Dx$    & Error        & Order        \\
		\midrule
		$2^{-2}$  & $2^{-1}$ & 0.0130915415 & ---          \\
		$2^{-4}$  & $2^{-2}$ & 0.0041148325 & 1.6697293690 \\
		$2^{-6}$  & $2^{-3}$ & 0.0009224433 & 2.1573015535 \\
		$2^{-8}$  & $2^{-4}$ & 0.0002336760 & 1.9809505153 \\
		$2^{-10}$ & $2^{-5}$ & 0.0000590647 & 1.9841427616 \\
		$2^{-12}$ & $2^{-6}$ & 0.0000151741 & 1.9606888670 \\
		\bottomrule
	\end{tabularx}
	\caption{Errors and orders of convergence for the solution to the \textbf{porous medium equation} \eqref{eq:porousmedium} with exponent $2$ in one dimension with \textbf{S1}. $D=1.0$, $m=2.0$, $\tini = 2.0$, $\tfin = 3.0$, $L = 6.0$.}
	\label{tab:OOCPorousMediumEquationSBCH12.0}
\end{table}
 \begin{table}[H]
	\centering
	\begin{tabularx}{0.88\textwidth}{XXXXX}
		\toprule
		$\Dt$    & $\Dx$    & $\Dy$    & Error        & Order        \\
		\midrule
		$2^{-2}$ & $2^{-1}$ & $2^{-1}$ & 0.2257693038 & ---          \\
		$2^{-4}$ & $2^{-2}$ & $2^{-2}$ & 0.0669120354 & 1.7545117122 \\
		$2^{-6}$ & $2^{-3}$ & $2^{-3}$ & 0.0190277529 & 1.8141605351 \\
		$2^{-8}$ & $2^{-4}$ & $2^{-4}$ & 0.0051266368 & 1.8920205977 \\
		\bottomrule
	\end{tabularx}
	\caption{Errors and orders of convergence for the solution to the \textbf{porous medium equation} \eqref{eq:porousmedium} with exponent $2$ in two dimensions with \textbf{S1}. $D=1.0$, $m=2.0$, $\tini = 2.0$, $\tfin = 3.0$, $L = 6.0$.}
	\label{tab:OOCPorousMediumEquationSBCH22.0}
\end{table}
 
\begin{table}[H]
	\centering
	\begin{tabularx}{0.88\textwidth}{XXXX}
		\toprule
		$\Dt$    & $\Dx$    & Error        & Order        \\
		\midrule
		$2^{-1}$ & $2^{-1}$ & 0.0332234361 & ---          \\
		$2^{-2}$ & $2^{-2}$ & 0.0186566864 & 0.8325085195 \\
		$2^{-3}$ & $2^{-3}$ & 0.0105248032 & 0.8258995178 \\
		$2^{-4}$ & $2^{-4}$ & 0.0056855204 & 0.8884289410 \\
		$2^{-5}$ & $2^{-5}$ & 0.0029958742 & 0.9243153420 \\
		$2^{-6}$ & $2^{-6}$ & 0.0015486779 & 0.9519399216 \\
		\bottomrule
	\end{tabularx}
	\caption{Errors and orders of convergence for the solution to the \textbf{porous medium equation} \eqref{eq:porousmedium} with exponent $2$ in one dimension with \textbf{S2}. $D=1.0$, $m=2.0$, $\tini = 2.0$, $\tfin = 3.0$, $L = 6.0$.}
	\label{tab:OOCPorousMediumEquationIBCH12.0}
\end{table}
 \begin{table}[H]
	\centering
	\begin{tabularx}{0.88\textwidth}{XXXXX}
		\toprule
		$\Dt$    & $\Dx$    & $\Dy$    & Error        & Order        \\
		\midrule
		$2^{-1}$ & $2^{-1}$ & $2^{-1}$ & 0.2502827038 & ---          \\
		$2^{-2}$ & $2^{-2}$ & $2^{-2}$ & 0.0916007693 & 1.4501269744 \\
		$2^{-3}$ & $2^{-3}$ & $2^{-3}$ & 0.0355814351 & 1.3642350159 \\
		$2^{-4}$ & $2^{-4}$ & $2^{-4}$ & 0.0155759666 & 1.1918030035 \\
		\bottomrule
	\end{tabularx}
	\caption{Errors and orders of convergence for the solution to the \textbf{porous medium equation} \eqref{eq:porousmedium} with exponent $2$ in two dimensions with \textbf{S2}. $D=1.0$, $m=2.0$, $\tini = 2.0$, $\tfin = 3.0$, $L = 6.0$.}
	\label{tab:OOCPorousMediumEquationIBCH22.0}
\end{table}
 
\begin{table}[H]
	\centering
	\begin{tabularx}{0.88\textwidth}{XXXX}
		\toprule
		$\Dt$     & $\Dx$    & Error        & Order        \\
		\midrule
		$2^{-2}$  & $2^{-1}$ & 0.0478347041 & ---          \\
		$2^{-4}$  & $2^{-2}$ & 0.0144860076 & 1.7233976320 \\
		$2^{-6}$  & $2^{-3}$ & 0.0039410392 & 1.8780120309 \\
		$2^{-8}$  & $2^{-4}$ & 0.0018019911 & 1.1289842459 \\
		$2^{-10}$ & $2^{-5}$ & 0.0005539346 & 1.7018042045 \\
		$2^{-12}$ & $2^{-6}$ & 0.0001794585 & 1.6260653630 \\
		\bottomrule
	\end{tabularx}
	\caption{Errors and orders of convergence for the solution to the \textbf{porous medium equation} \eqref{eq:porousmedium} with exponent $3$ in one dimension with \textbf{S1}. $D=1.0$, $m=3.0$, $\tini = 2.0$, $\tfin = 3.0$, $L = 6.0$.}
	\label{tab:OOCPorousMediumEquationSBCH13.0}
\end{table}
 \begin{table}[H]
	\centering
	\begin{tabularx}{0.88\textwidth}{XXXXX}
		\toprule
		$\Dt$    & $\Dx$    & $\Dy$    & Error        & Order        \\
		\midrule
		$2^{-2}$ & $2^{-1}$ & $2^{-1}$ & 0.4389569539 & ---          \\
		$2^{-4}$ & $2^{-2}$ & $2^{-2}$ & 0.1745831277 & 1.3301653308 \\
		$2^{-6}$ & $2^{-3}$ & $2^{-3}$ & 0.0639979688 & 1.4478161159 \\
		$2^{-8}$ & $2^{-4}$ & $2^{-4}$ & 0.0219859179 & 1.5414463482 \\
		\bottomrule
	\end{tabularx}
	\caption{Errors and orders of convergence for the solution to the \textbf{porous medium equation} \eqref{eq:porousmedium} with exponent $3$ in two dimensions with \textbf{S1}. $D=1.0$, $m=3.0$, $\tini = 2.0$, $\tfin = 3.0$, $L = 6.0$.}
	\label{tab:OOCPorousMediumEquationSBCH23.0}
\end{table}
 
\begin{table}[H]
	\centering
	\begin{tabularx}{0.88\textwidth}{XXXX}
		\toprule
		$\Dt$    & $\Dx$    & Error        & Order        \\
		\midrule
		$2^{-1}$ & $2^{-1}$ & 0.0581751903 & ---          \\
		$2^{-2}$ & $2^{-2}$ & 0.0196251683 & 1.5676989967 \\
		$2^{-3}$ & $2^{-3}$ & 0.0115953163 & 0.7591628538 \\
		$2^{-4}$ & $2^{-4}$ & 0.0071496227 & 0.6976031599 \\
		$2^{-5}$ & $2^{-5}$ & 0.0040079983 & 0.8349852132 \\
		$2^{-6}$ & $2^{-6}$ & 0.0021089620 & 0.9263487670 \\
		\bottomrule
	\end{tabularx}
	\caption{Errors and orders of convergence for the solution to the \textbf{porous medium equation} \eqref{eq:porousmedium} with exponent $3$ in one dimension with \textbf{S2}. $D=1.0$, $m=3.0$, $\tini = 2.0$, $\tfin = 3.0$, $L = 6.0$.}
	\label{tab:OOCPorousMediumEquationIBCH13.0}
\end{table}
 \begin{table}[H]
	\centering
	\begin{tabularx}{0.88\textwidth}{XXXXX}
		\toprule
		$\Dt$    & $\Dx$    & $\Dy$    & Error        & Order        \\
		\midrule
		$2^{-1}$ & $2^{-1}$ & $2^{-1}$ & 0.4476788322 & ---          \\
		$2^{-2}$ & $2^{-2}$ & $2^{-2}$ & 0.1869728167 & 1.2596355670 \\
		$2^{-3}$ & $2^{-3}$ & $2^{-3}$ & 0.0744864973 & 1.3277777102 \\
		$2^{-4}$ & $2^{-4}$ & $2^{-4}$ & 0.0345075878 & 1.1100652941 \\
		\bottomrule
	\end{tabularx}
	\caption{Errors and orders of convergence for the solution to the \textbf{porous medium equation} \eqref{eq:porousmedium} with exponent $3$ in two dimensions with \textbf{S2}. $D=1.0$, $m=3.0$, $\tini = 2.0$, $\tfin = 3.0$, $L = 6.0$.}
	\label{tab:OOCPorousMediumEquationIBCH23.0}
\end{table}

\subsection{Linear, Non-Linear and Non-local Fokker-Planck Equations}\label{sec:FPvalidation}

In order to validate our schemes for equations involving potentials, we consider the linear Fokker-Planck equation $\pt\rho=D\Delta\rho+\div\prt{\rho\bx}$, i.e.,
\begin{equation}\label{eq:linearFP}
	H(\rho)=D\prt{\rho\log(\rho)-\rho},\quad
	V(\bx)=\frac{\abs{\bx}^2}{2},\quad
	W(\bx)=0,
\end{equation}
for $D>0$. Regardless of the initial datum, there is a unique, globally stable steady state for the equation, given by the heat kernel \eqref{eq:heatkernel} at $t=1/2$, i.e.
\begin{equation}\label{eq:linearFPsteady}
	\rho_{\infty}(\bx)=
	\prt{2\pi D}^{-\frac{n}{2}}
	\exp\prt*{-\frac{\abs{\bx}^2}{2D}}.
\end{equation}
\revision{Furthermore, the evolution of an initial point source at the origin towards this equilibrium is given by
	\begin{align}\label{eq:linearFPanalytic}
		\Upsilon\prt{t,\bx} =
		\prt{2\pi D\prt{1-e^{-2t}}}^{-\frac{n}{2}}
		\exp\prt*{-\frac{\abs{\bx}^2}{2D\prt{1-e^{-2t}}}},
	\end{align}
	see \cite{Pavliotis2014} for instance.
}

\begin{figure}[ht]
	\centering
	\begin{subfigure}[t]{0.66\textwidth}
		\centering
		\begin{subfigure}[t]{0.49\textwidth}
			\centering
			\includegraphics[trim={2cm 1cm 1.25cm 2cm},clip, width=\textwidth]{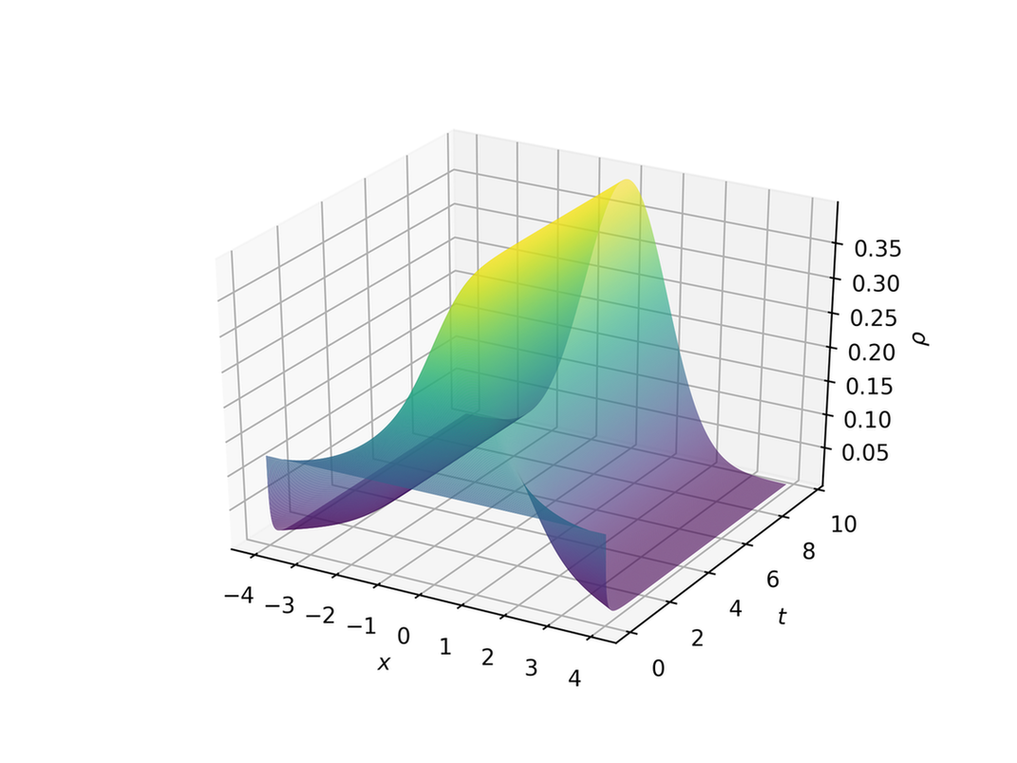}
			\caption{Convergence to $\rho_{\infty}$ in time.}
		\end{subfigure}~
		\begin{subfigure}[t]{0.49\textwidth}
			\centering
			\includegraphics[width=\textwidth]{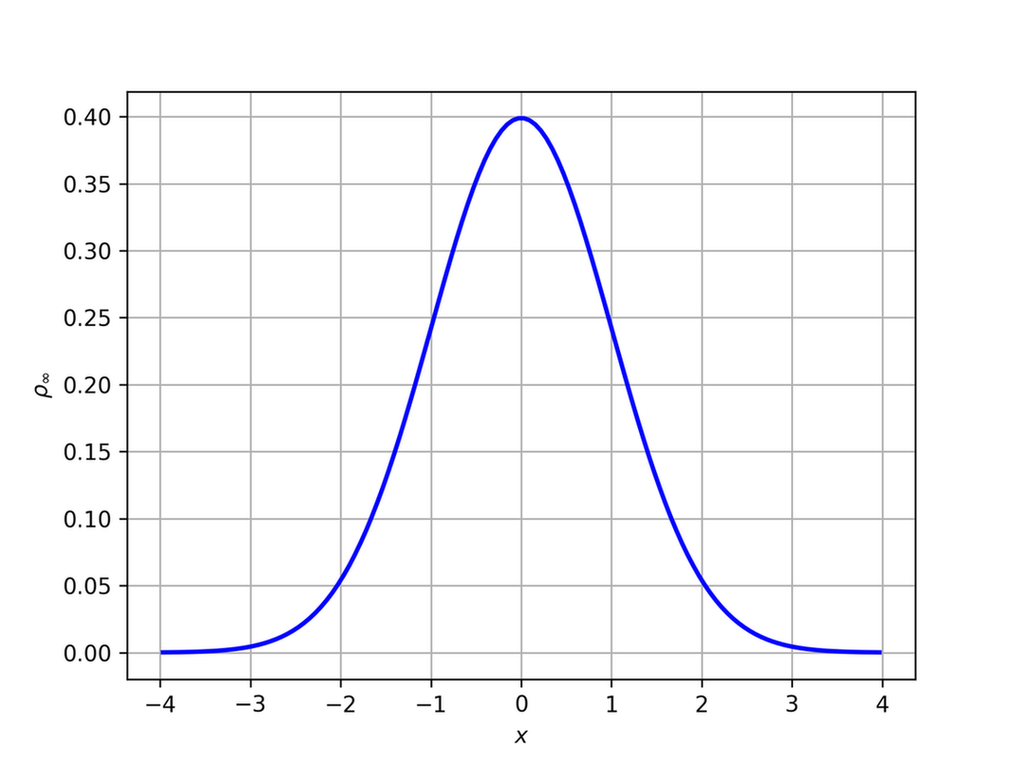}
			\caption{Stationary state $\rho_{\infty}(x)$.}
		\end{subfigure}\caption*{One dimension.}
	\end{subfigure}~
	\begin{subfigure}[t]{0.33\textwidth}
		\centering
		\begin{subfigure}[t]{1.0\textwidth}
			\centering
			\includegraphics[trim={2cm 1cm 1.25cm 2cm},clip, width=\textwidth]{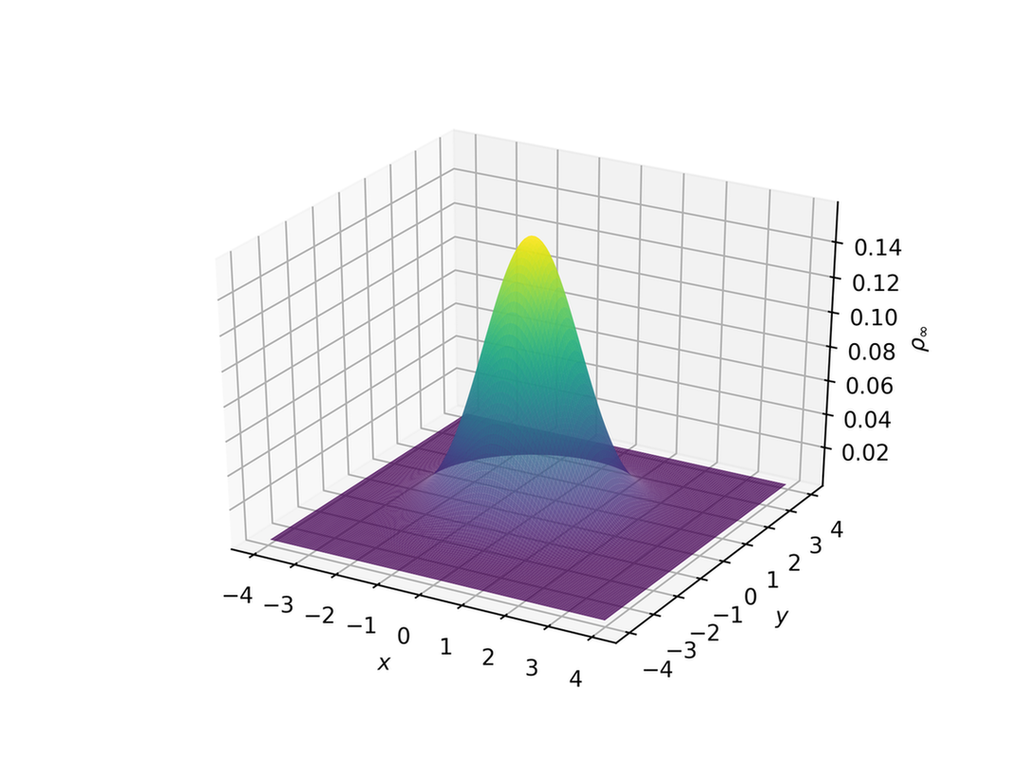}
			\caption{Stationary state $\rho_{\infty}(x)$.}
		\end{subfigure}
		\caption*{Two dimensions.}
	\end{subfigure}
	\caption{Stationary state of the \textbf{non-local Fokker-Planck equation} \eqref{eq:nonlocalFP} with $D=1$, equivalent to that of the linear Fokker-Planck equation \eqref{eq:linearFP} due to the symmetry of the initial datum about the origin.}
	\label{fig:nonlocalFP}
\end{figure}
 
The confining potential of \eqref{eq:linearFP} can be replaced by an equal interaction potential, permitting the validation of the interaction component of the schemes. The new equation involves a non-local term but will have the same solution as the linear Fokker-Planck Equation for all initial datum which is symmetric about the origin, see Figure \ref{fig:nonlocalFP}. This non-local Fokker-Planck Equation
\begin{equation}\label{eq:nonlocalFP}
	H(\rho)=D\prt{\rho\log(\rho)-\rho},\quad
	V(\bx)=0,\quad
	W(\bx)=\frac{\abs{\bx}^2}{2},
\end{equation}
for $D>0$, ought to display the same analytical solution \eqref{eq:linearFPanalytic}, the same steady state \eqref{eq:linearFPsteady} and the same order of convergence to equilibrium as the local case. In the analytic setting, with centred initial datum, the $L_1$ difference $\norm{\rho(t,\bx)-\rho_{\infty}(\bx)}_{L_1}$ is expected to decrease exponentially with order $\mathcal{O}\prt{-2t}$. Furthermore, the \textit{relative entropy} $E\prt{\rho}-E\prt{\rho_{\infty}}$ should decrease with $\mathcal{O}\prt{-4t}$ \cite{Toscani1999}.

\revision{To validate the convergence of the sweeping dimensional splitting schemes, we validate the evolution of a source solution of \cref{eq:nonlocalFP} in two dimensions against the analytical solution \eqref{eq:linearFPanalytic}, in the same fashion as \cref{sec:heatvalidation,sec:PMEvalidation}. The validation of the solution to \cref{eq:linearFP} has also been included for comparison. \cref{tab:FPSBCH2,tab:FPSBCH2W} demonstrate the second-order convergence using S1 for \eqref{eq:linearFP} and \eqref{eq:nonlocalFP} respectively. \cref{tab:FPIBCH2,tab:FPIBCH2W} show the corresponding (better than) first-order convergence using S2.
}

\begin{table}[H]
	\centering
	\begin{tabularx}{0.88\textwidth}{XXXXX}
		\toprule
		$\Dt$     & $\Dx$    & $\Dy$    & Error        & Order        \\
		\midrule
		$2^{-4}$  & $2^{-1}$ & $2^{-1}$ & 0.0130193017 & ---          \\
		$2^{-6}$  & $2^{-2}$ & $2^{-2}$ & 0.0033748735 & 1.9477467433 \\
		$2^{-8}$  & $2^{-3}$ & $2^{-3}$ & 0.0008538822 & 1.9827243883 \\
		$2^{-10}$ & $2^{-4}$ & $2^{-4}$ & 0.0002153366 & 1.9874436476 \\
		$2^{-12}$ & $2^{-5}$ & $2^{-5}$ & 0.0000542630 & 1.9885519609 \\
		\bottomrule
	\end{tabularx}
	\caption{Errors and orders of convergence for the solution to the \textbf{linear Fokker-Planck equation} \eqref{eq:linearFP} in two dimensions with \textbf{S1}. $D=1.0$, $\tini = 2.0$, $\tfin = 3.0$, $L = 5.0$.}
	\label{tab:FPSBCH2}
\end{table}
 \begin{table}[H]
	\centering
	\begin{tabularx}{0.88\textwidth}{XXXXX}
		\toprule
		$\Dt$     & $\Dx$    & $\Dy$    & Error        & Order        \\
		\midrule
		$2^{-6}$  & $2^{-1}$ & $2^{-1}$ & 0.0128997621 & ---          \\
		$2^{-8}$  & $2^{-2}$ & $2^{-2}$ & 0.0033440967 & 1.9476559875 \\
		$2^{-10}$ & $2^{-3}$ & $2^{-3}$ & 0.0008446799 & 1.9851398626 \\
		$2^{-12}$ & $2^{-4}$ & $2^{-4}$ & 0.0002133594 & 1.9851187830 \\
		$2^{-14}$ & $2^{-5}$ & $2^{-5}$ & 0.0000537499 & 1.9889523045 \\
		\bottomrule
	\end{tabularx}
	\caption{Errors and orders of convergence for the solution to the \textbf{non-local Fokker-Planck equation} \eqref{eq:nonlocalFP} in two dimensions with \textbf{S1}. $D=1.0$, $\tini = 2.0$, $\tfin = 3.0$, $L = 5.0$.}
	\label{tab:FPSBCH2W}
\end{table}
 
\begin{table}[H]
	\centering
	\begin{tabularx}{0.88\textwidth}{XXXXX}
		\toprule
		$\Dt$    & $\Dx$    & $\Dy$    & Error        & Order        \\
		\midrule
		$2^{-1}$ & $2^{-1}$ & $2^{-1}$ & 0.0126781382 & ---          \\
		$2^{-2}$ & $2^{-2}$ & $2^{-2}$ & 0.0035203530 & 1.8485509035 \\
		$2^{-3}$ & $2^{-3}$ & $2^{-3}$ & 0.0009623843 & 1.8710350515 \\
		$2^{-4}$ & $2^{-4}$ & $2^{-4}$ & 0.0002759015 & 1.8024599428 \\
		$2^{-5}$ & $2^{-5}$ & $2^{-5}$ & 0.0000804105 & 1.7786959505 \\
		\bottomrule
	\end{tabularx}
	\caption{Errors and orders of convergence for the solution to the \textbf{linear Fokker-Planck equation} \eqref{eq:linearFP} in two dimensions with \textbf{S2}. $D=1.0$, $\tini = 2.0$, $\tfin = 3.0$, $L = 5.0$.}
	\label{tab:FPIBCH2}
\end{table}
 \begin{table}[H]
	\centering
	\begin{tabularx}{0.88\textwidth}{XXXXX}
		\toprule
		$\Dt$    & $\Dx$    & $\Dy$    & Error        & Order        \\
		\midrule
		$2^{-1}$ & $2^{-1}$ & $2^{-1}$ & 0.0126777040 & ---          \\
		$2^{-2}$ & $2^{-2}$ & $2^{-2}$ & 0.0035197973 & 1.8487292513 \\
		$2^{-3}$ & $2^{-3}$ & $2^{-3}$ & 0.0009618213 & 1.8716515859 \\
		$2^{-4}$ & $2^{-4}$ & $2^{-4}$ & 0.0002753364 & 1.8045732416 \\
		$2^{-5}$ & $2^{-5}$ & $2^{-5}$ & 0.0000798789 & 1.7853077858 \\
		\bottomrule
	\end{tabularx}
	\caption{Errors and orders of convergence for the solution to the \textbf{non-local Fokker-Planck equation} \eqref{eq:nonlocalFP} in two dimensions with \textbf{S2}. $D=1.0$, $\tini = 2.0$, $\tfin = 3.0$, $L = 5.0$.}
	\label{tab:FPIBCH2W}
\end{table}
  
To validate the energy dissipation properties of the schemes, we studied the convergence in time of Gaussian initial datum in both problems to the numerical steady state, verifying the agreement between the local and the non-local settings. Convergence to the known dissipation rates upon refinement of the mesh was verified as well --- see Figure \ref{fig:nonlocalFPconvergence} for the two-dimensional case.

\begin{figure}[ht]
	\centering
	\begin{subfigure}[t]{0.42\textwidth}
		\centering
		\includegraphics[width=\textwidth]{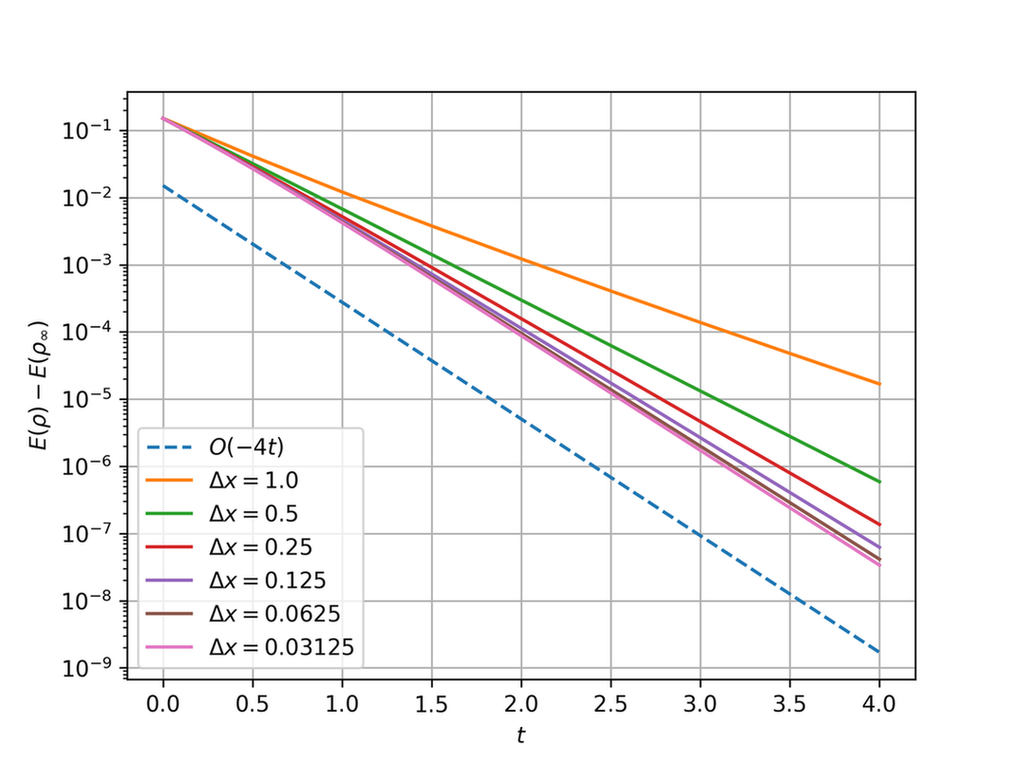}
		\caption{S1.}
	\end{subfigure}~
	\begin{subfigure}[t]{0.42\textwidth}
		\centering
		\includegraphics[width=\textwidth]{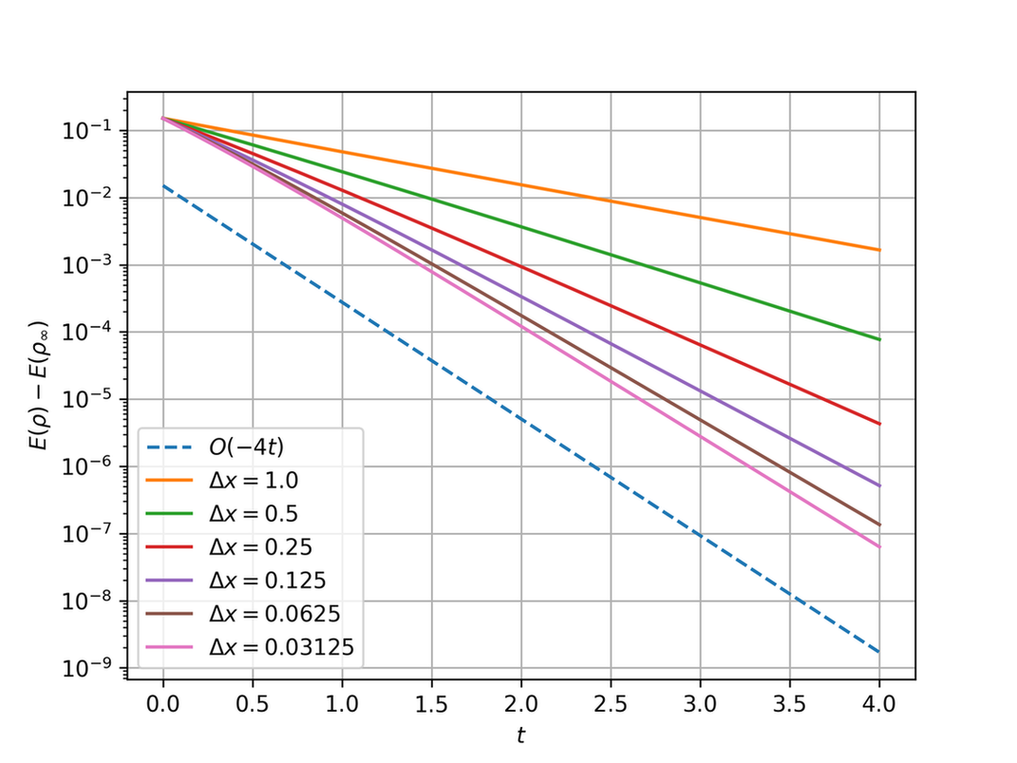}
		\caption{S2.}
	\end{subfigure}\caption{Dissipation of the discrete energy $E_\Delta$ in the convergence to the stationary state of the \textbf{non-local Fokker-Planck equation} \eqref{eq:nonlocalFP} in two dimensions. Note the slopes approach $\mathcal{O}\prt{-4t}$ as the mesh is refined. $\Delta t=\Delta x$, $D=1.0$, $L = 5.0$.}
	\label{fig:nonlocalFPconvergence}
\end{figure}
 
To further the discussion, we consider a non-linear diffusion case also. Replacing the linear term on \eqref{eq:linearFP} by the porous medium equivalent yields a non-linear Fokker-Planck Equation
\begin{equation}\label{eq:nonlinearFP}
	H(\rho)=\frac{D}{m-1}\rho^m,\quad
	V(\bx)=\frac{\abs{\bx}^2}{2},\quad
	W(\bx)=0,
\end{equation}
for $D>0,m>1$. Again, regardless of initial datum this equation exhibits a globally stable steady state, see Figure \ref{fig:nonlinearFP}.

\begin{figure}[ht]
	\centering
	\begin{subfigure}[t]{0.66\textwidth}
		\centering
		\begin{subfigure}[t]{0.49\textwidth}
			\centering
			\includegraphics[trim={2cm 1cm 1.25cm 2cm},clip, width=\textwidth]{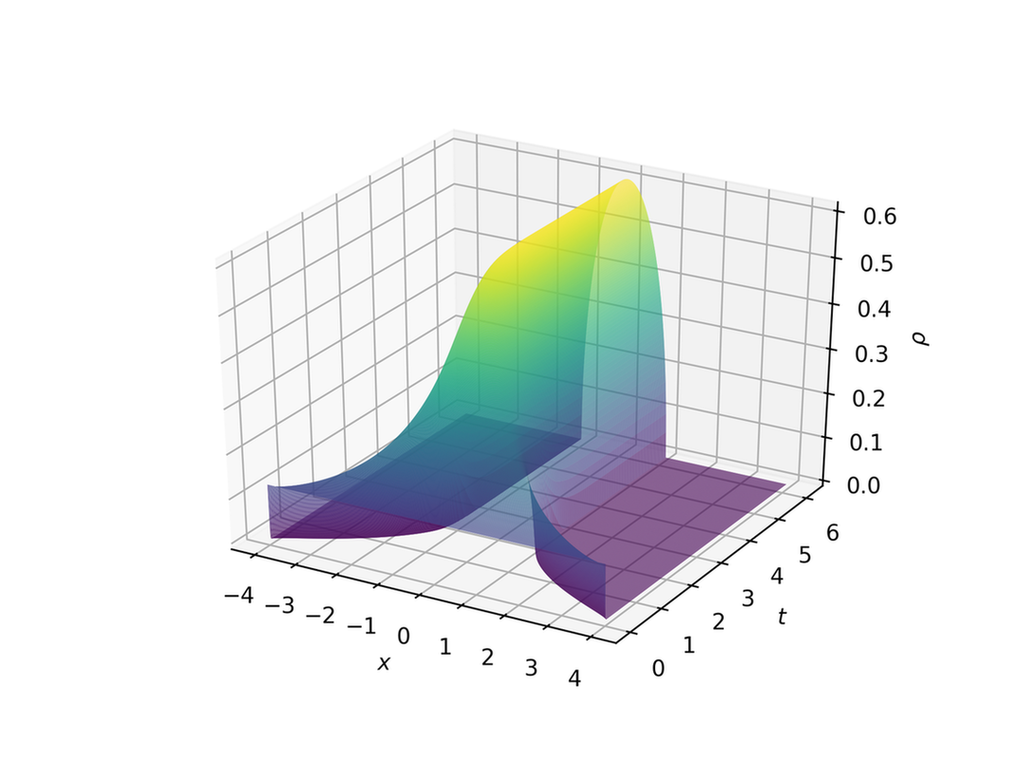}
			\caption{Convergence to $\rho_{\infty}$ in time.}
		\end{subfigure}~
		\begin{subfigure}[t]{0.49\textwidth}
			\centering
			\includegraphics[width=\textwidth]{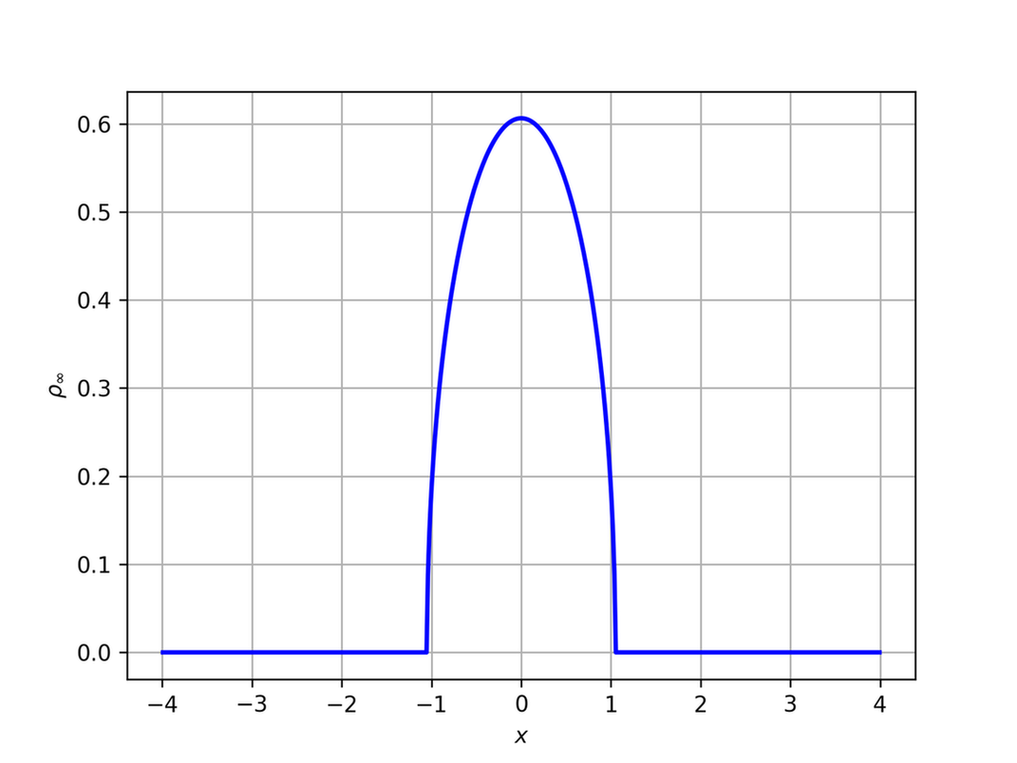}
			\caption{Stationary state $\rho_{\infty}(x)$.}
		\end{subfigure}\caption*{One dimension.}
	\end{subfigure}~
	\begin{subfigure}[t]{0.33\textwidth}
		\centering
		\begin{subfigure}[t]{1.0\textwidth}
			\centering
			\includegraphics[trim={2cm 1cm 1.25cm 2cm},clip, width=\textwidth]{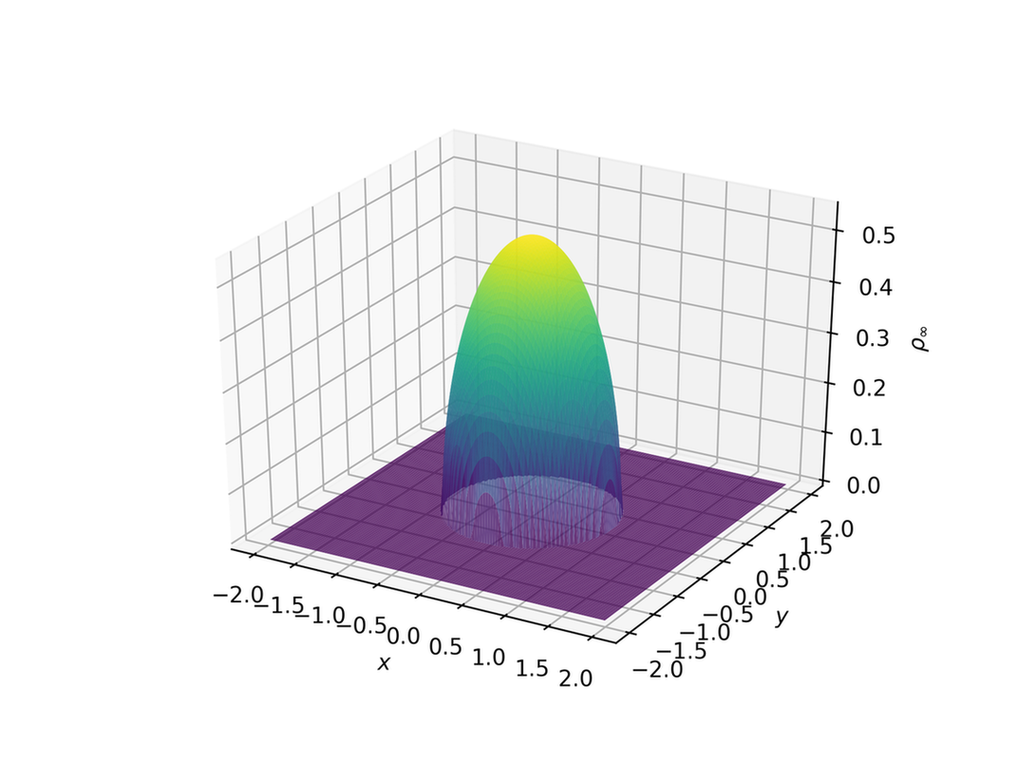}
			\caption{Stationary state $\rho_{\infty}(x)$.}
		\end{subfigure}
		\caption*{Two dimensions.}
	\end{subfigure}
	\caption{Stationary state of the \textbf{non-linear Fokker-Planck equation} \eqref{eq:nonlinearFP} with $D=1, m=3$.}
	\label{fig:nonlinearFP}
\end{figure}
 
The regularity of the steady solution is once again controlled by the exponent $m$, and so is the rate of convergence to the stationary profile. In one dimension, for symmetric initial datum, the $L_1$ difference $\norm{\rho(t,\bx)-\rho_{\infty}(\bx)}_{L_1}$ and the relative entropy $E\prt{\rho}-E\prt{\rho_{\infty}}$ dissipate with $\mathcal{O}\prt{-(m+1)t}$ and $\mathcal{O}\prt{-2(m+1)t}$ respectively \cite{C.D.T2007}. Similar verifications were performed --- see Figure \ref{fig:nonlinearFPconvergence} for $m=3$.

\begin{figure}[ht]
	\centering
	\begin{subfigure}[t]{0.42\textwidth}
		\centering
		\includegraphics[width=\textwidth]{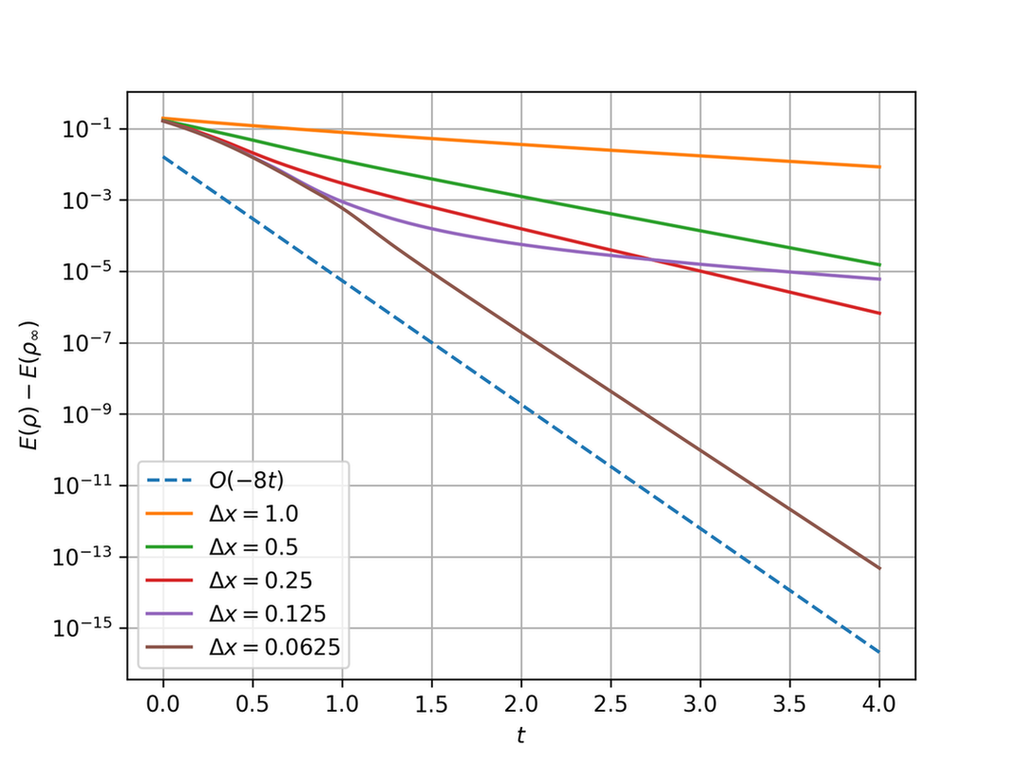}
		\caption{S1.}
	\end{subfigure}~
	\begin{subfigure}[t]{0.42\textwidth}
		\centering
		\includegraphics[width=\textwidth]{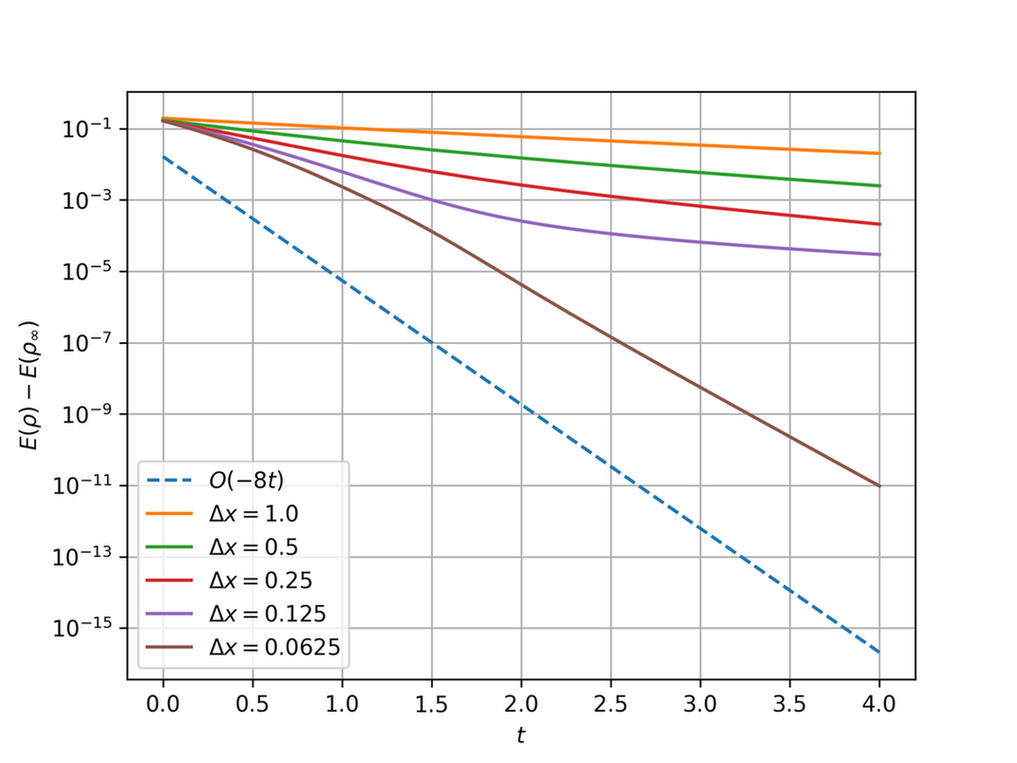}
		\caption{S2.}
	\end{subfigure}\caption{Dissipation of the discrete energy $E_\Delta$ in the convergence to the stationary state of the \textbf{non-linear Fokker-Planck equation} \eqref{eq:nonlinearFP} with exponent $3$ in one dimension. Note the slopes approach $\mathcal{O}\prt{-8t}$ as the mesh is refined. $\Delta t=\Delta x$, $D=1.0$, $m=3.0$, $L = 5.0$.}
	\label{fig:nonlinearFPconvergence}
\end{figure}
  \section{Numerical Experiments with S2}\label{sec:experiments}

This concluding section presents a selection of experiments which aim to showcase some interesting problems which can be solved with the S2 scheme. First, we consider steady state problems with a variety of confining potentials. Later on, we discuss equations whose solutions display metastability in their convergence to equilibrium. Finally, we study a phase transition driven by noise in a kinetic system by constructing the stable branch of the bifurcation diagram.

\subsection{Convergence to Steady States}

Section \ref{sec:FPvalidation} concerned the convergence to globally stable stationary solutions. Beyond the standard Fokker-Planck setting, the equivalents of \eqref{eq:linearFP} and \eqref{eq:nonlinearFP} with more intricate confining potentials may be considered. For instance, a bistable term yields
\begin{equation}\label{eq:bistableheat}
	H(\rho)=D\prt{\rho\log(\rho)-\rho},\quad
	V(\bx)=\frac{\abs{\bx}^4}{4}-\frac{\abs{\bx}^2}{2},\quad
	W(\bx)=0,
\end{equation}
for $D>0$ in the linear diffusion case, which displays a globally stable steady state characterised by maxima at $\abs{\bx}=1$; see Figure \ref{fig:bistableheat}.
\begin{figure}[ht]
	\centering
	\begin{subfigure}[t]{0.66\textwidth}
		\centering
		\begin{subfigure}[t]{0.49\textwidth}
			\centering
			\includegraphics[trim={2cm 1cm 1.25cm 2cm},clip, width=\textwidth]{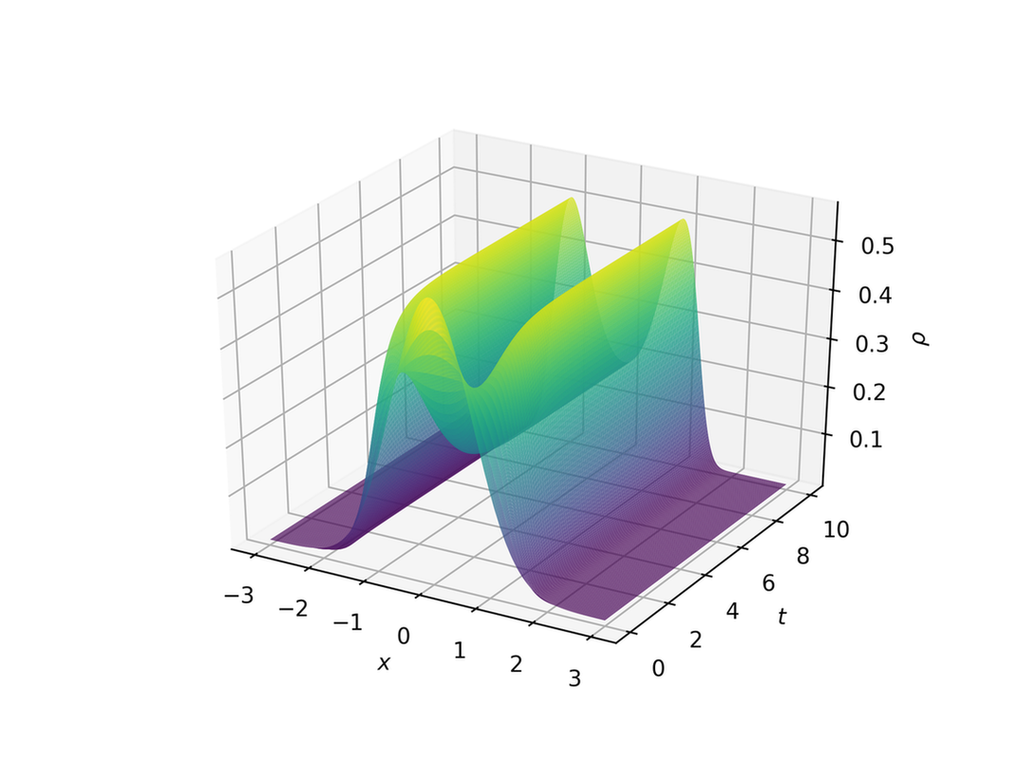}
			\caption{Convergence to $\rho_{\infty}$ in time.}
		\end{subfigure}~
		\begin{subfigure}[t]{0.49\textwidth}
			\centering
			\includegraphics[width=\textwidth]{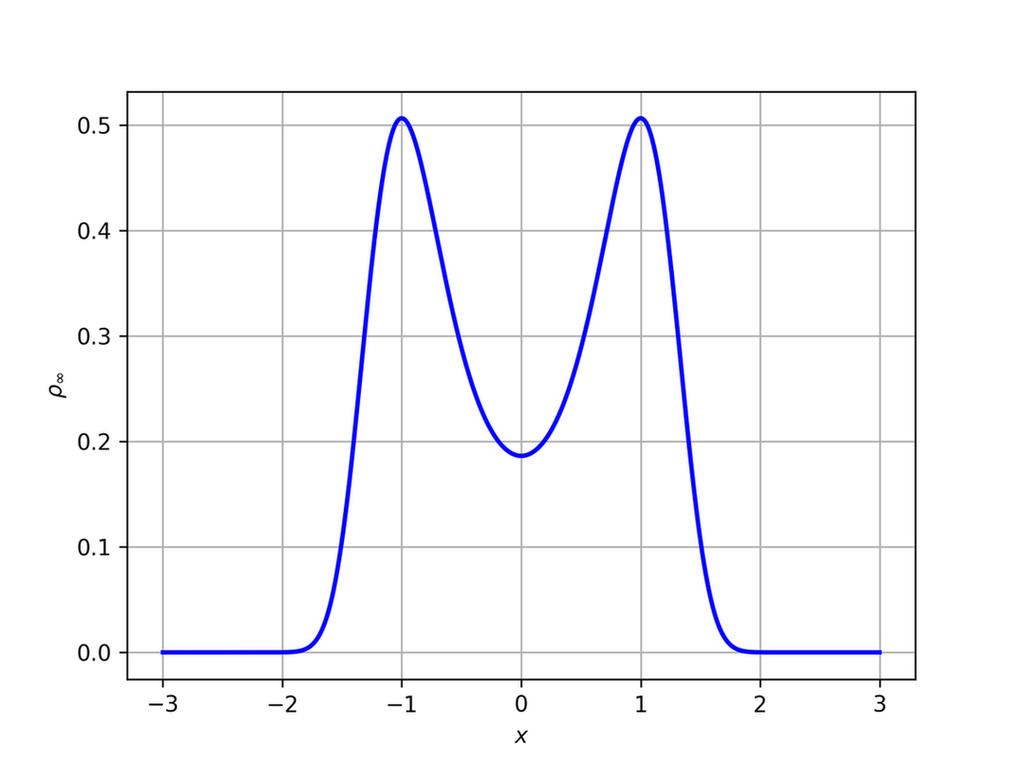}
			\caption{Stationary state $\rho_{\infty}(x)$.}
		\end{subfigure}\caption*{One dimension.}
	\end{subfigure}~
	\begin{subfigure}[t]{0.33\textwidth}
		\centering
		\begin{subfigure}[t]{1.0\textwidth}
			\centering
			\includegraphics[trim={2cm 1cm 1.25cm 2cm},clip, width=\textwidth]{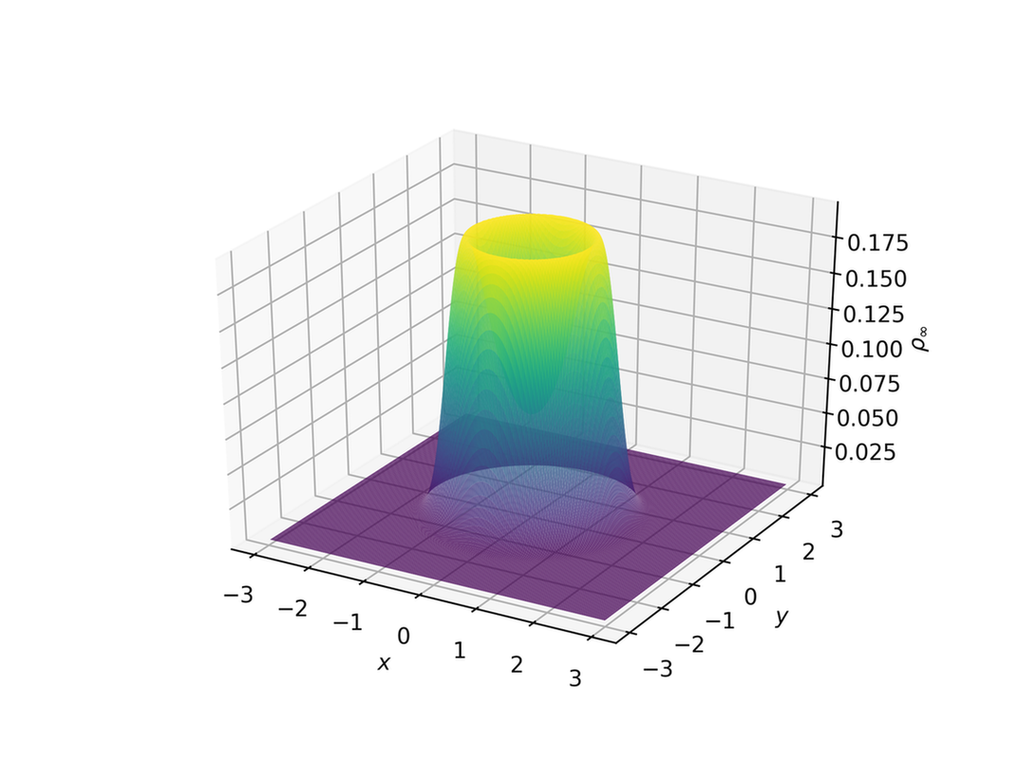}
			\caption{Stationary state $\rho_{\infty}(x)$.}
		\end{subfigure}
		\caption*{Two dimensions.}
	\end{subfigure}
	\caption{Stationary state of equation \eqref{eq:bistableheat} with $D=0.25$. Note the maxima at $\abs{\bx}=1$.}
	\label{fig:bistableheat}
\end{figure}
 \begin{figure}[ht]
	\centering
	\begin{subfigure}[t]{0.66\textwidth}
		\centering
		\begin{subfigure}[t]{0.49\textwidth}
			\centering
			\includegraphics[trim={2cm 1cm 1.25cm 2cm},clip, width=\textwidth]{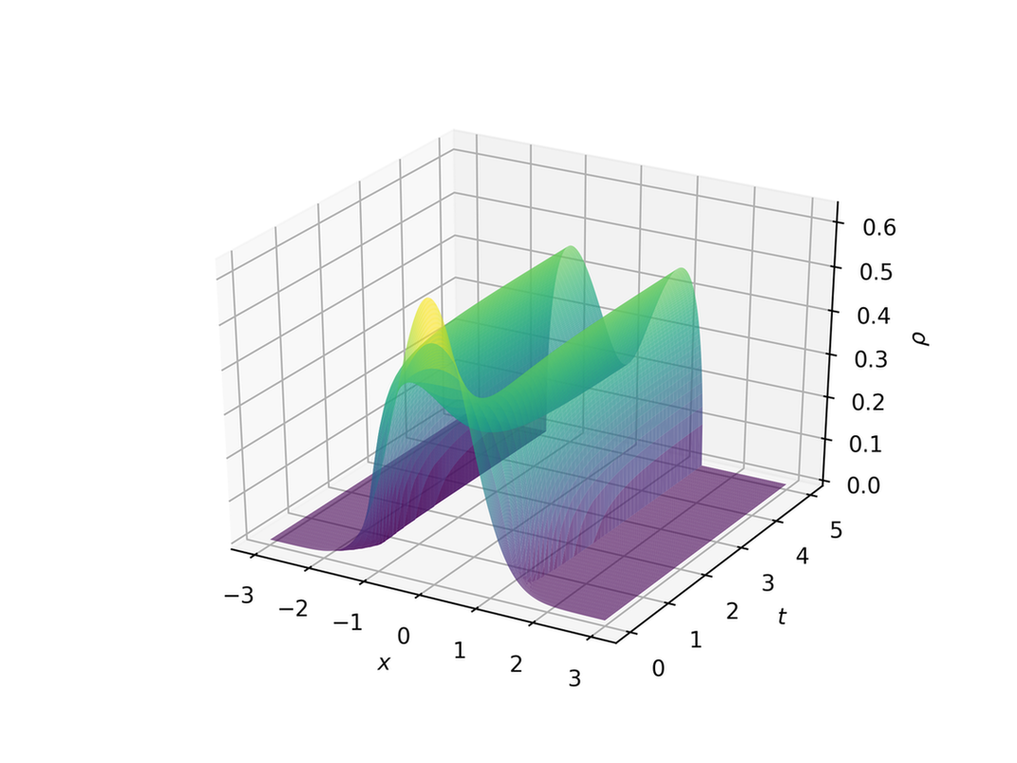}
			\caption{Convergence to $\rho_{\infty}$ in time.}
		\end{subfigure}~
		\begin{subfigure}[t]{0.49\textwidth}
			\centering
			\includegraphics[width=\textwidth]{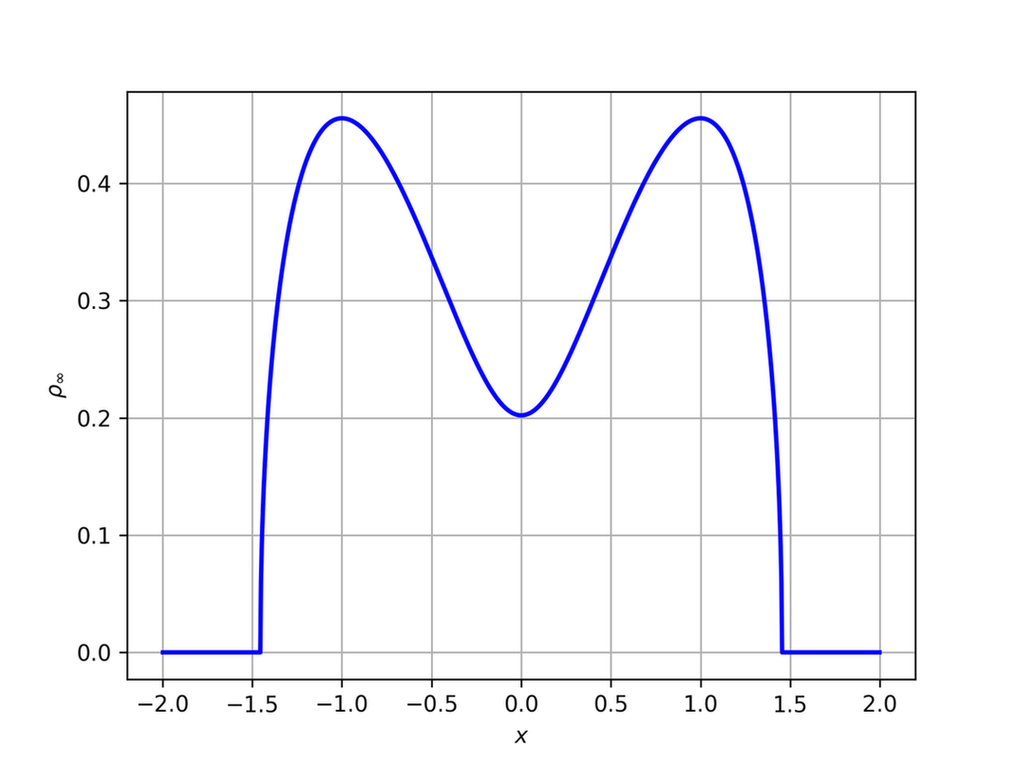}
			\caption{Stationary state $\rho_{\infty}(x)$.}
		\end{subfigure}\caption*{One dimension.}
	\end{subfigure}~
	\begin{subfigure}[t]{0.33\textwidth}
		\centering
		\begin{subfigure}[t]{1.0\textwidth}
			\centering
			\includegraphics[trim={2cm 1cm 1.25cm 2cm},clip, width=\textwidth]{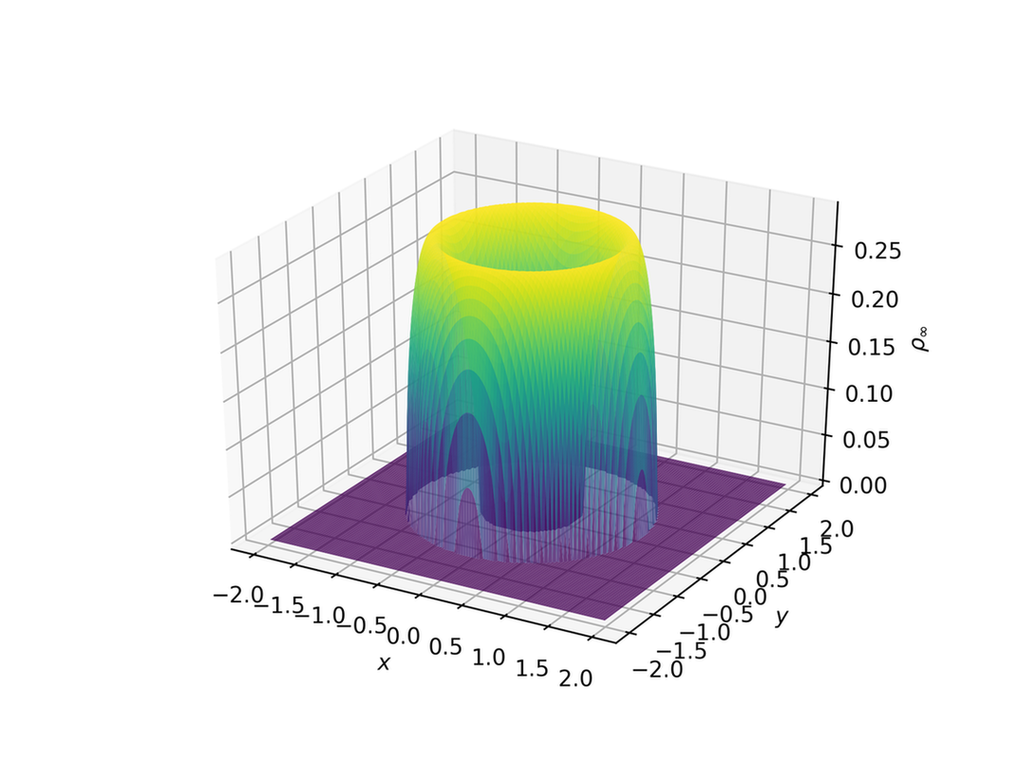}
			\caption{Stationary state $\rho_{\infty}(x)$.}
		\end{subfigure}
		\caption*{Two dimensions.}
	\end{subfigure}
	\caption{Stationary state of equation \eqref{eq:bistableporousmedium} with $D=1, m=3$.}
	\label{fig:bistableporousmedium}
\end{figure}
 
In the non-linear setting, the equation reads:
\begin{equation}\label{eq:bistableporousmedium}
	H(\rho)=\frac{D}{m-1}\rho^m,\quad
	V(\bx)=\frac{\abs{\bx}^4}{4}-\frac{\abs{\bx}^2}{2},\quad
	W(\bx)=0,
\end{equation}
for $D>0, m>1$. The non-linear diffusion equivalent of \eqref{eq:bistableheat} also has a unique stable steady state, compactly supported and characterised by maxima at $\abs{\bx}=1$ in two dimensions. In the one-dimensional setting, the steady state is only unique provided the diffusion coefficient $D$ is large \cite{C.C.H2015}. Note that in two dimensions the stationary solution might not be simply connected --- see Figure \ref{fig:bistableporousmedium}.

\subsection{Metastability}

We will now study the behaviour of a non-linear diffusion equation with an attractive interaction kernel:
\begin{equation}\label{eq:nonlocalnonlineardiffusion}
	H(\rho)=\frac{D}{m-1}\rho^m,\quad
	V(\bx)=0,\quad
	W(\bx)=-\frac{1}{\sqrt{2\pi\sigma^2}}\exp\prt*{-\frac{|\bx|^2}{2\sigma^2}},
\end{equation}
for $D>0,m>1,\sigma>0$. This equation can exhibit a many-step convergence to equilibrium: rather than converging at a fixed rate, the energy dissipates in an alternating sequence of slow and fast timescales. Whilst the true steady state consists of a simply connected, compactly supported component, intermediate aggregates which depend on the initial datum can rapidly form. These aggregates will eventually merge but the rate of convergence can be arbitrarily slow if $\sigma$ is small.

Three examples are presented: Figure \ref{fig:nonlocalnonlineardiffusion1D}, where two aggregates are formed before reaching the final equilibrium; Figure \ref{fig:nonlocalnonlineardiffusion1Dthree}, where three and then two aggregates are present before the steady state appears; and Figure \ref{fig:nonlocalnonlineardiffusion2D}, which shows the asymmetric aggregation in two dimensions. Note the intermediate plateaux on the energy landscapes, each corresponding to one of the many-aggregate states.

\begin{figure}[ht]
	\centering

	\begin{subfigure}[t]{0.32\textwidth}
		\centering
		\includegraphics[trim={2cm 1cm 1.25cm 2cm},clip, width=\textwidth]{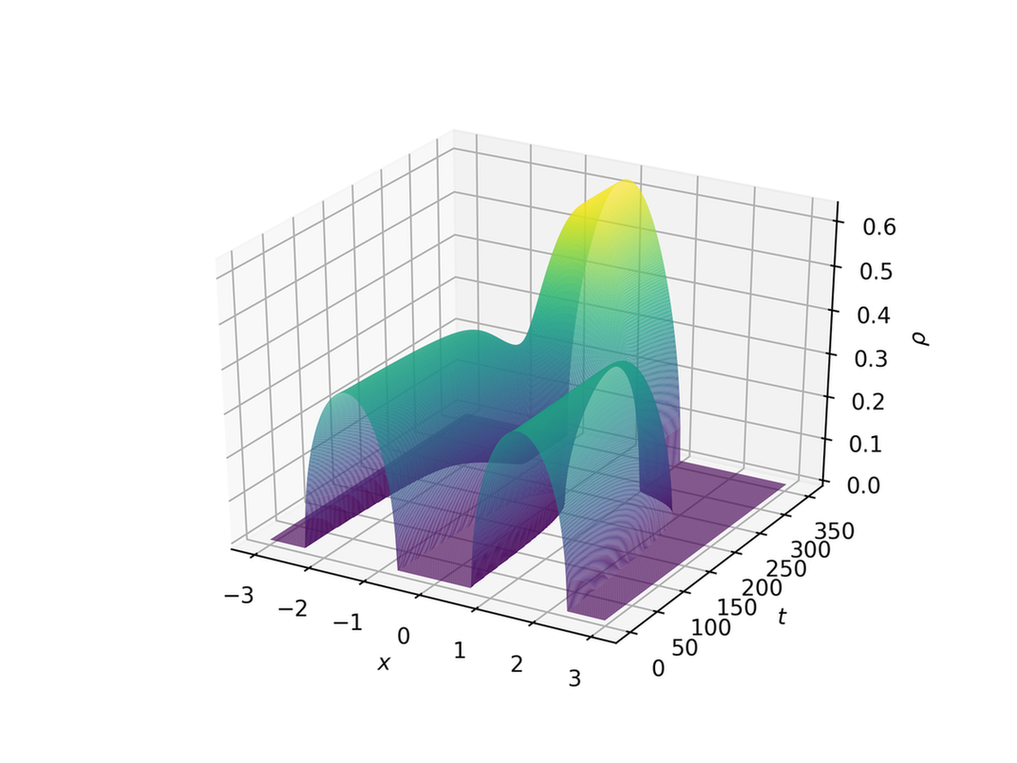}
		\caption{Convergence to $\rho_{\infty}$ in time.}
	\end{subfigure}~
	\begin{subfigure}[t]{0.32\textwidth}
		\centering
		\includegraphics[width=\textwidth]{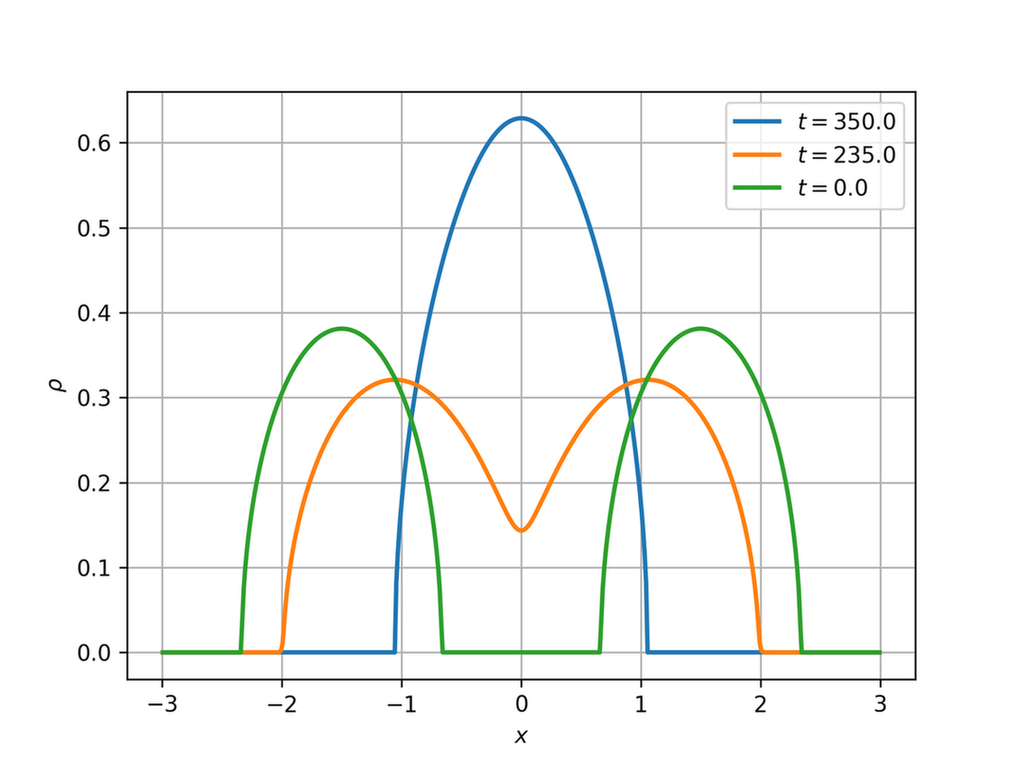}
		\caption{Intermediate and stationary states.}
	\end{subfigure}~
	\begin{subfigure}[t]{0.32\textwidth}
		\centering
		\includegraphics[width=\textwidth]{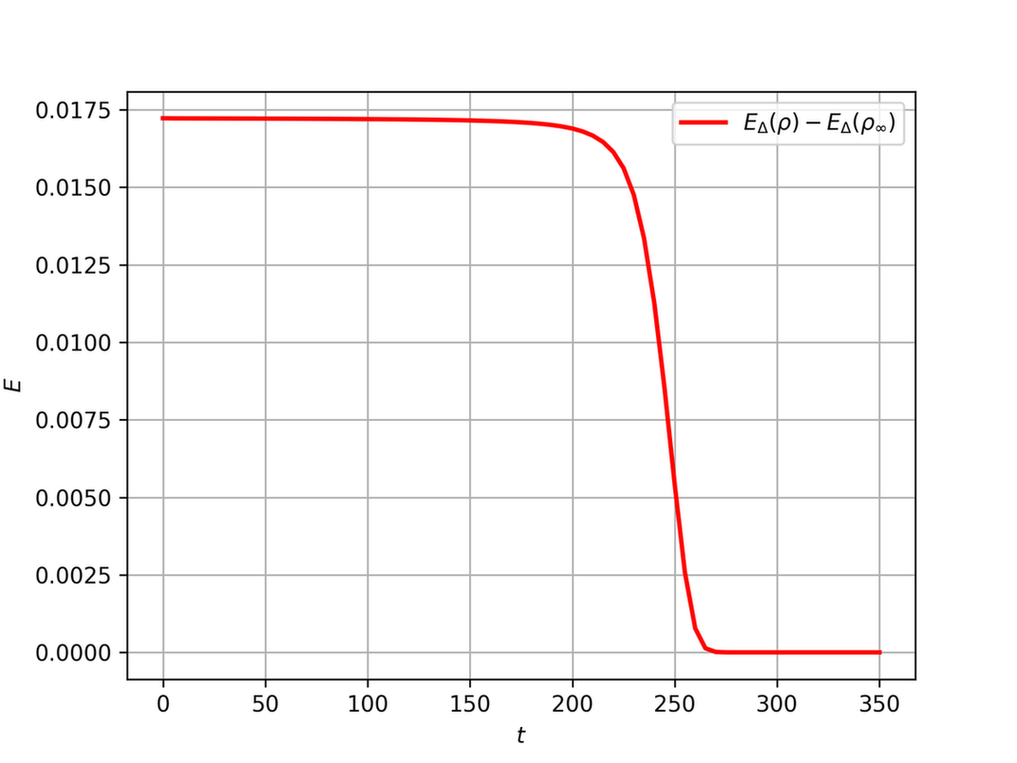}
		\caption{Dissipation of the discrete energy.}
	\end{subfigure}\caption{Two-aggregate solution of equation \eqref{eq:nonlocalnonlineardiffusion} for $D=0.1, m=3, \sigma=0.5$.}
	\label{fig:nonlocalnonlineardiffusion1D}
\end{figure}

\begin{figure}[ht]
	\centering

	\begin{subfigure}[t]{0.32\textwidth}
		\centering
		\includegraphics[trim={2cm 1cm 1.25cm 2cm},clip, width=\textwidth]{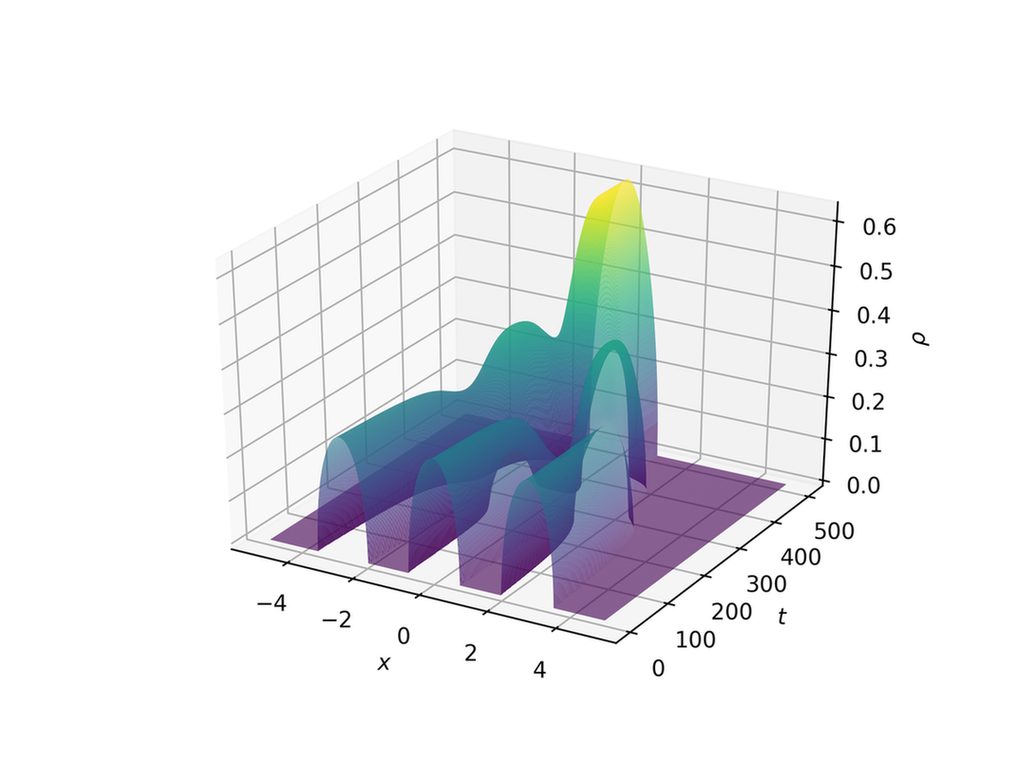}
		\caption{Convergence to $\rho_{\infty}$ in time.}
	\end{subfigure}~
	\begin{subfigure}[t]{0.32\textwidth}
		\centering
		\includegraphics[width=\textwidth]{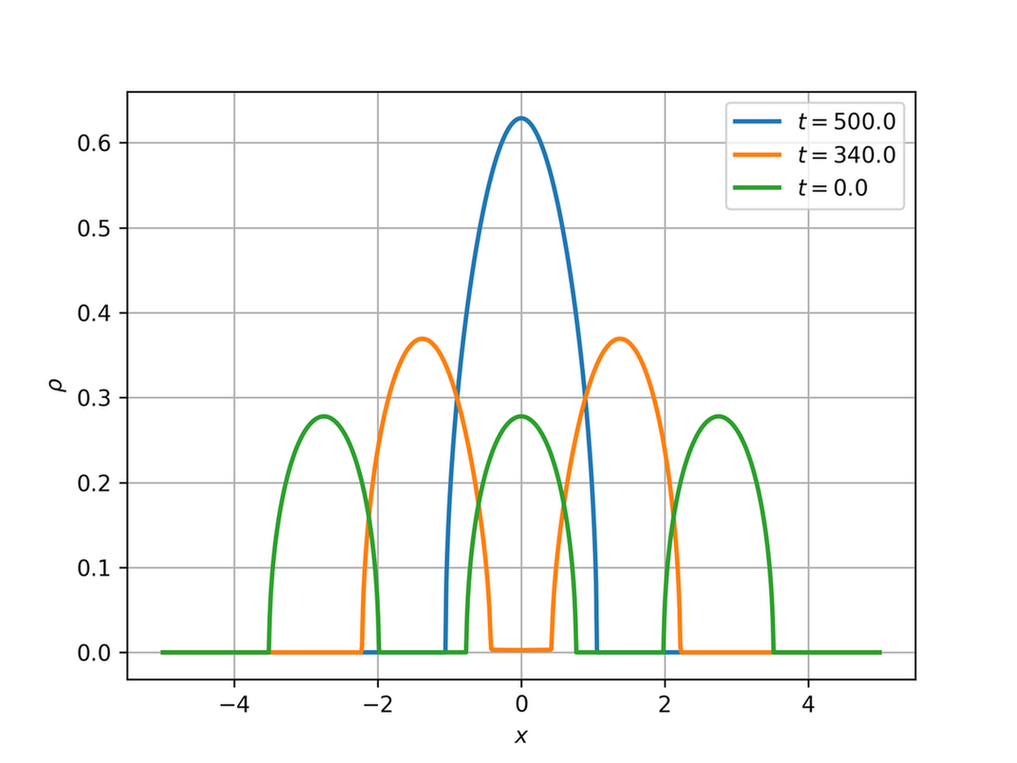}
		\caption{Intermediate and stationary states.}
	\end{subfigure}~
	\begin{subfigure}[t]{0.32\textwidth}
		\centering
		\includegraphics[width=\textwidth]{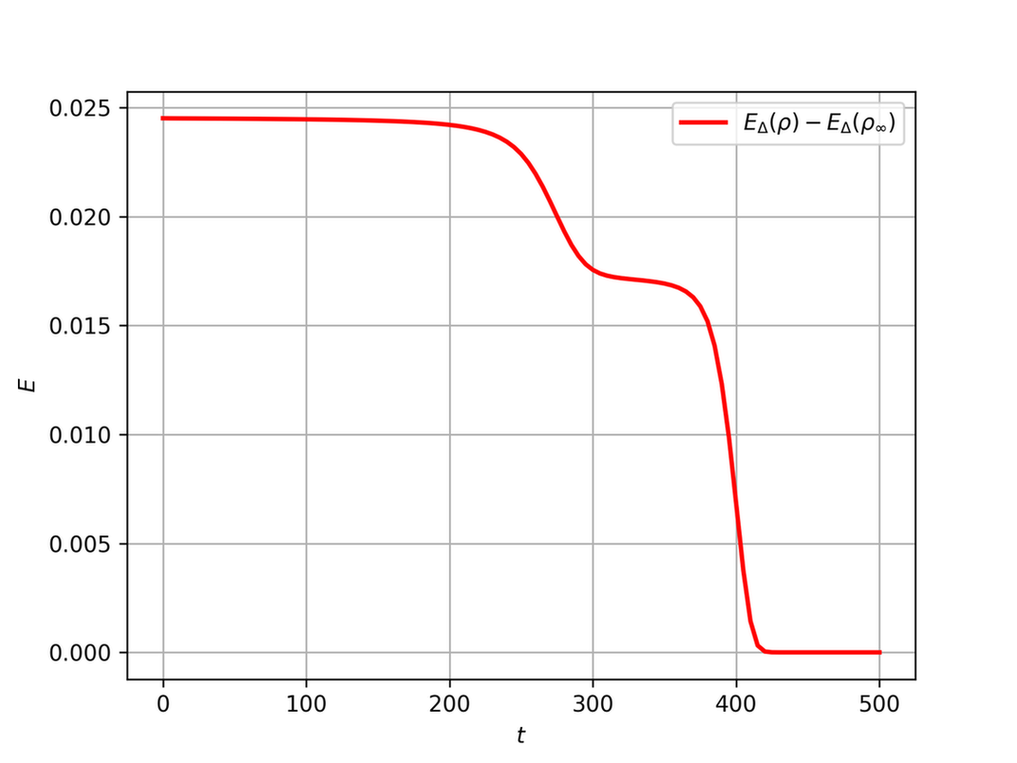}
		\caption{Dissipation of the discrete energy.}
	\end{subfigure}\caption{Three-then-two-aggregate solution of equation \eqref{eq:nonlocalnonlineardiffusion} for $D=0.1, m=3, \sigma=0.5$.}
	\label{fig:nonlocalnonlineardiffusion1Dthree}
\end{figure}
 
\begin{figure}[ht!]
	\centering

	\begin{subfigure}[t]{0.49\textwidth}
		\centering
		\includegraphics[trim={2cm 1cm 1.25cm 2cm},clip, width=\textwidth]{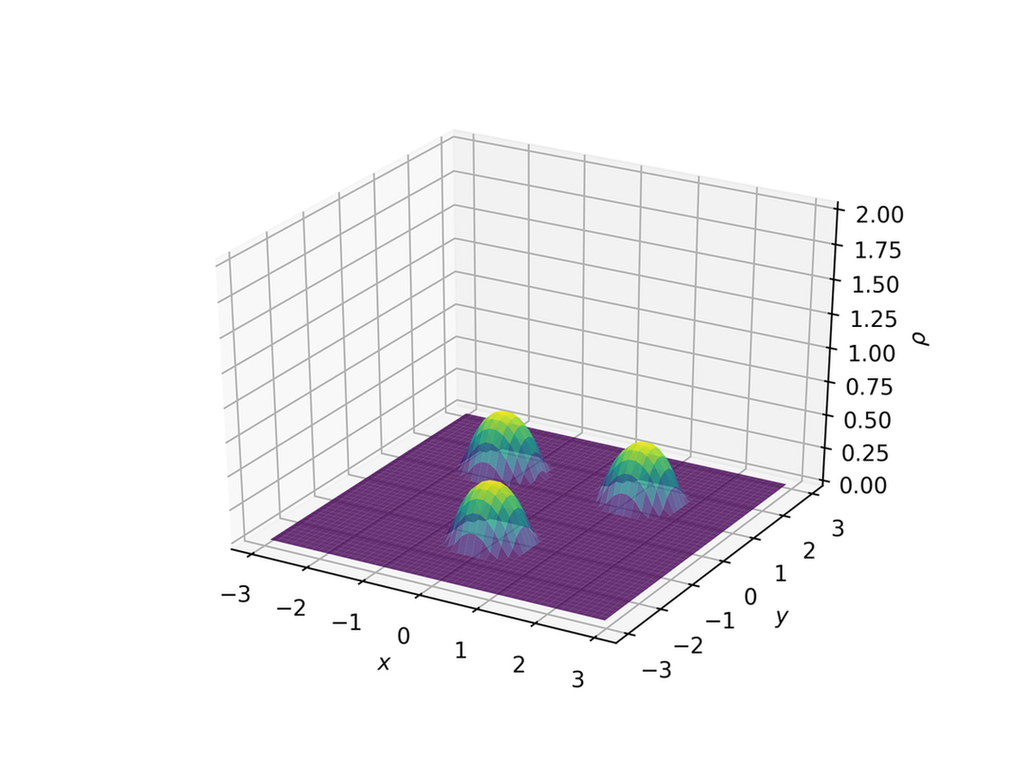}
		\caption{$t=0$.}
	\end{subfigure}~
	\begin{subfigure}[t]{0.49\textwidth}
		\centering
		\includegraphics[trim={2cm 1cm 1.25cm 2cm},clip, width=\textwidth]{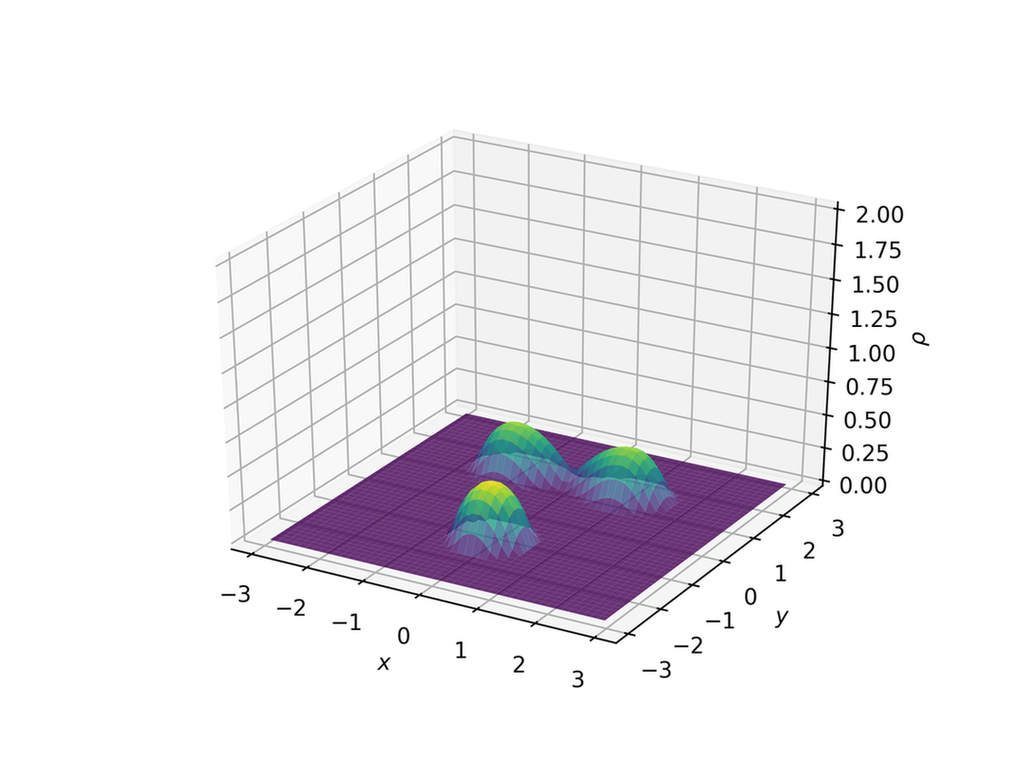}
		\caption{$t=150$.}
	\end{subfigure}

	\begin{subfigure}[t]{0.49\textwidth}
		\centering
		\includegraphics[trim={2cm 1cm 1.25cm 2cm},clip, width=\textwidth]{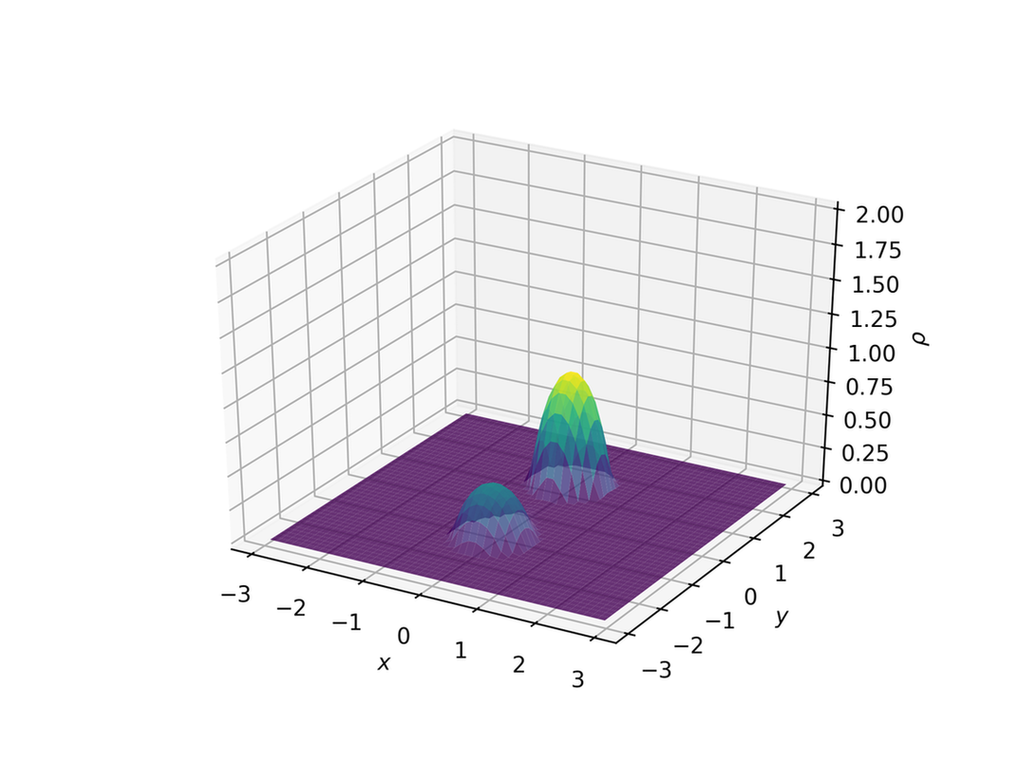}
		\caption{$t=230$.}
	\end{subfigure}~
	\begin{subfigure}[t]{0.49\textwidth}
		\centering
		\includegraphics[trim={2cm 1cm 1.25cm 2cm},clip, width=\textwidth]{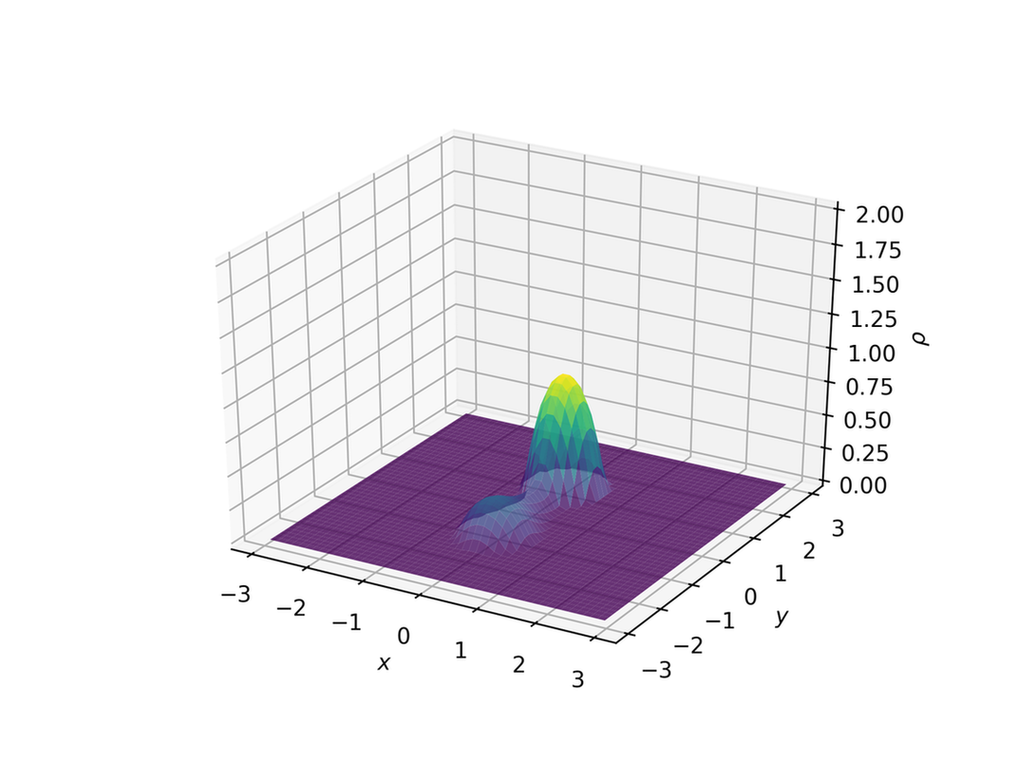}
		\caption{$t=275$.}
	\end{subfigure}

	\begin{subfigure}[t]{0.49\textwidth}
		\centering
		\includegraphics[trim={2cm 1cm 1.25cm 2cm},clip, width=\textwidth]{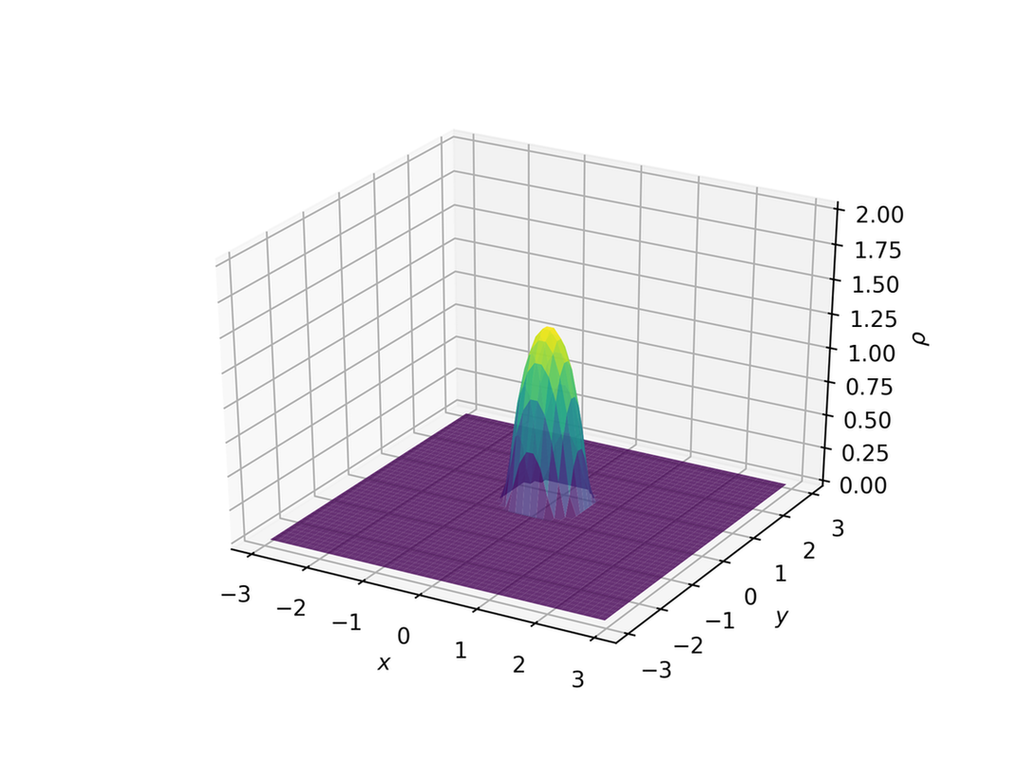}
		\caption{$t=500$.}
	\end{subfigure}~
	\begin{subfigure}[t]{0.49\textwidth}
		\centering
		\includegraphics[width=\textwidth]{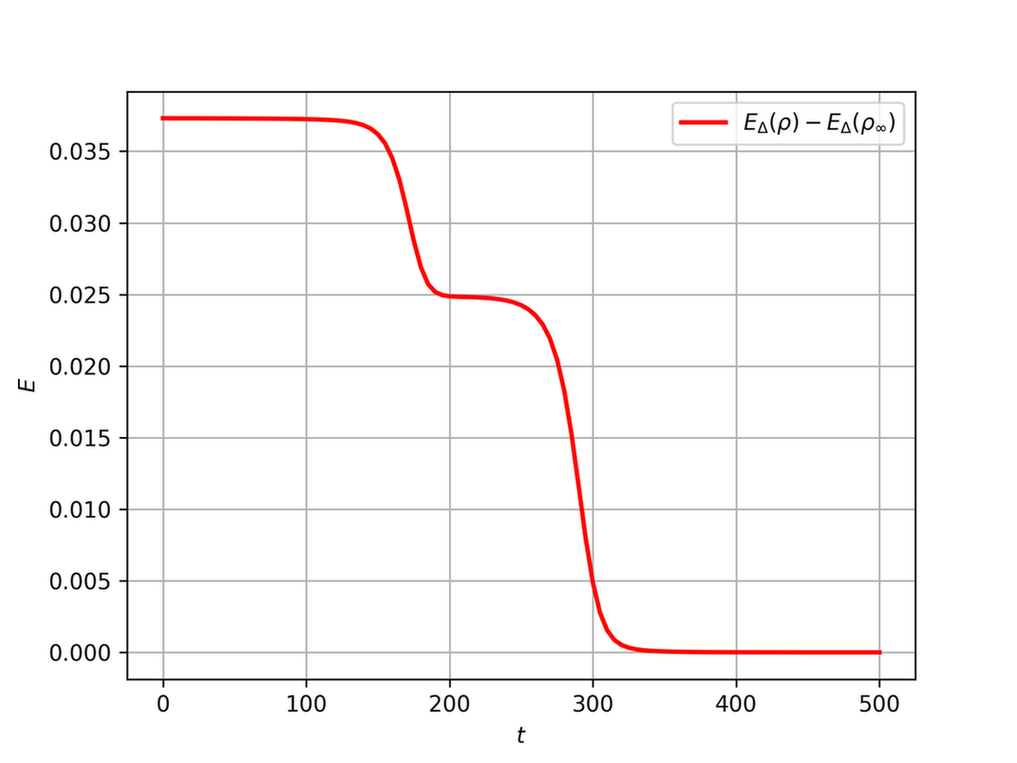}
		\caption{Dissipation of the discrete energy.}
	\end{subfigure}

	\caption{Three-then-two-aggregate solution of equation \eqref{eq:nonlocalnonlineardiffusion} for $D=0.01, m=2, \sigma=0.5$ in two dimensions.}
	\label{fig:nonlocalnonlineardiffusion2D}
\end{figure}

\subsection{Homogeneous Noisy Kinetic Flocking}

For the last example we will discuss a kinetic model for the velocity of self-propelled agents with a noisy tendency to flock:
\begin{equation}\label{eq:kineticflocking}
	H(\rho)=\sigma\prt{\rho\log\prt{\rho}-\rho},\quad
	V(\bx)=\alpha\prt*{\frac{\abs{\bx}^4}{4}-\frac{\abs{\bx}^2}{2}},\quad
	W(\bx)=\frac{\abs{\bx}^2}{2},
\end{equation}
for $\sigma\geq0, \alpha\geq 0$. For the sake of simplicity we retain the notation $\bx$ even though the equation concerns velocities. The confinement potential represents the preference of the agents to move with speed one. The interaction kernel models the alignment tendency, and the diffusion component accounts for the noise in the system.

This model was studied at length in \cite{Tugaut2014, B.C.C+2016, C.G.P+2020}. Among other things, the authors prove the existence of a phase transition in the system. Low values of $\sigma$ allow asymmetric initial conditions to flock, resulting in polarised steady states; the equation admits a symmetric steady state which is unstable and only realised for symmetric initial datum. Increasing the parameter beyond a critical threshold plunges the system into isotropic symmetry regardless of the initial condition.

The S2 scheme allowed us to solve the steady state problem of \eqref{eq:kineticflocking} for a large range of values of $\sigma$. The first moment of the steady state $\rho_{\infty}$,
\begin{equation}
	\langle \bx \rangle
	=\int \bx \rho_{\infty} d\bx,
\end{equation}
can be studied as a function of the noise strength $\sigma$, revealing whether the system is polarised or not. A sharp transition from the asymmetric polarised steady states to the isotropic setting can be seen on Figure \ref{fig:kinetic} for the one-dimensional case. The centre of mass of the initial datum was shifted along the positive axis, resulting in the polarisation in that direction. By symmetry there is always another polarised steady state in the opposite direction.

\begin{figure}[ht]
	\centering
	\begin{subfigure}[t]{0.49\textwidth}
		\centering
		\includegraphics[width=\textwidth]{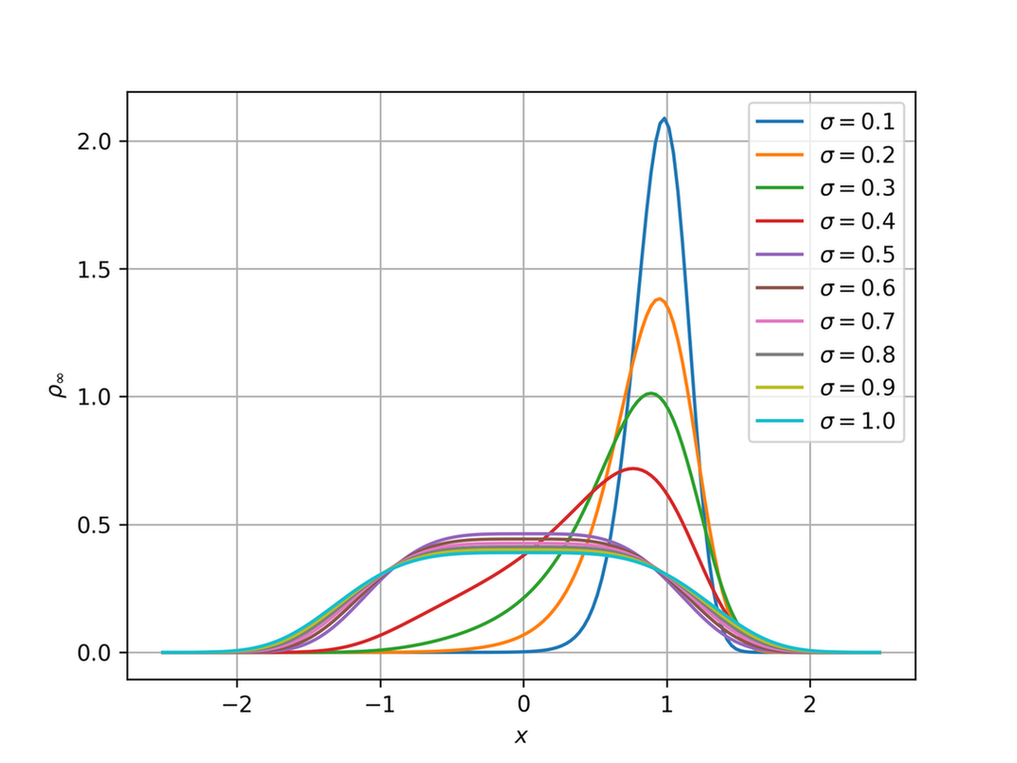}
		\caption{Stationary state $\rho_{\infty}(x)$ for different values of $\sigma$.}
	\end{subfigure}~
	\begin{subfigure}[t]{0.49\textwidth}
		\centering
		\includegraphics[width=\textwidth]{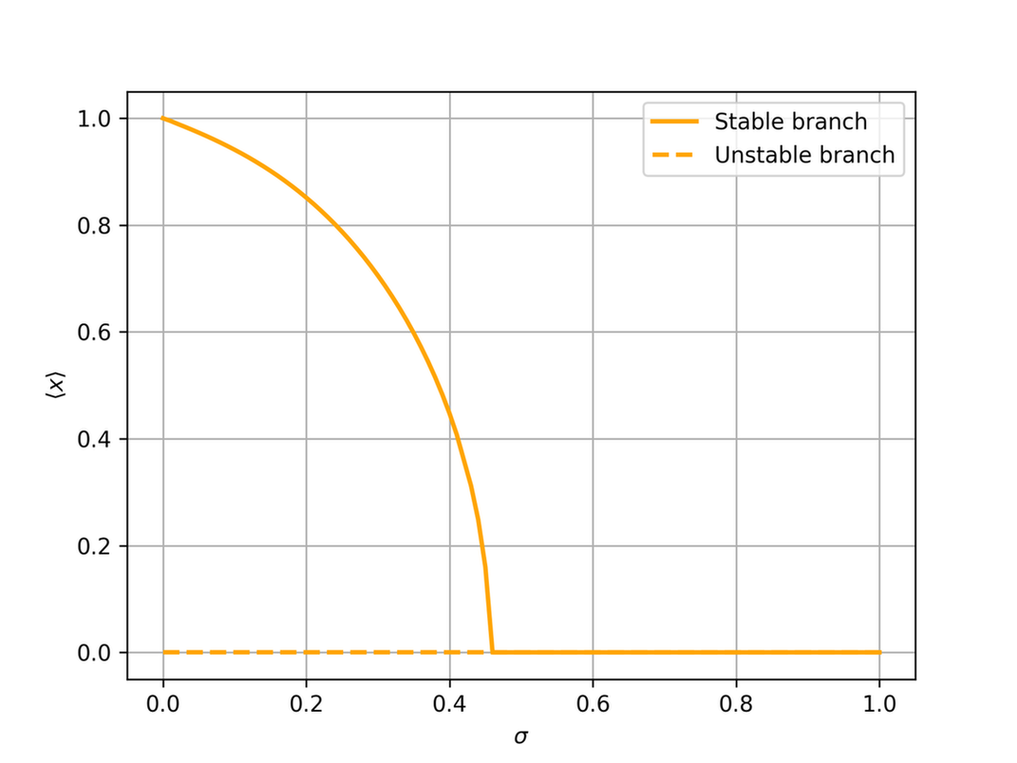}
		\caption{Bifurcation diagram.}
	\end{subfigure}\caption{Stationary states and phase transition of \eqref{eq:kineticflocking} for $\alpha=1$ in one dimension.}
	\label{fig:kinetic}
\end{figure}
 
The same phenomenon is observed in the two-dimensional setting, see Figure \ref{fig:kinetic2D} and \cite{B.C.C+2016} for the analysis. The initial datum was shifted along the positive $x$ axis, resulting in the corresponding polarisation. Note that there is a rotationally symmetric family of polarised steady states. These states resemble a von Mises--Fisher distribution obtained for the Vicsek model ($\alpha=\infty$), see \cite{D.F.L2015}.

\begin{figure}[ht]
	\centering
	\begin{subfigure}[t]{1.0\textwidth}
		\begin{subfigure}[t]{0.49\textwidth}
			\centering
			\includegraphics[trim={2cm 1cm 1.25cm 2cm},clip, width=\textwidth]{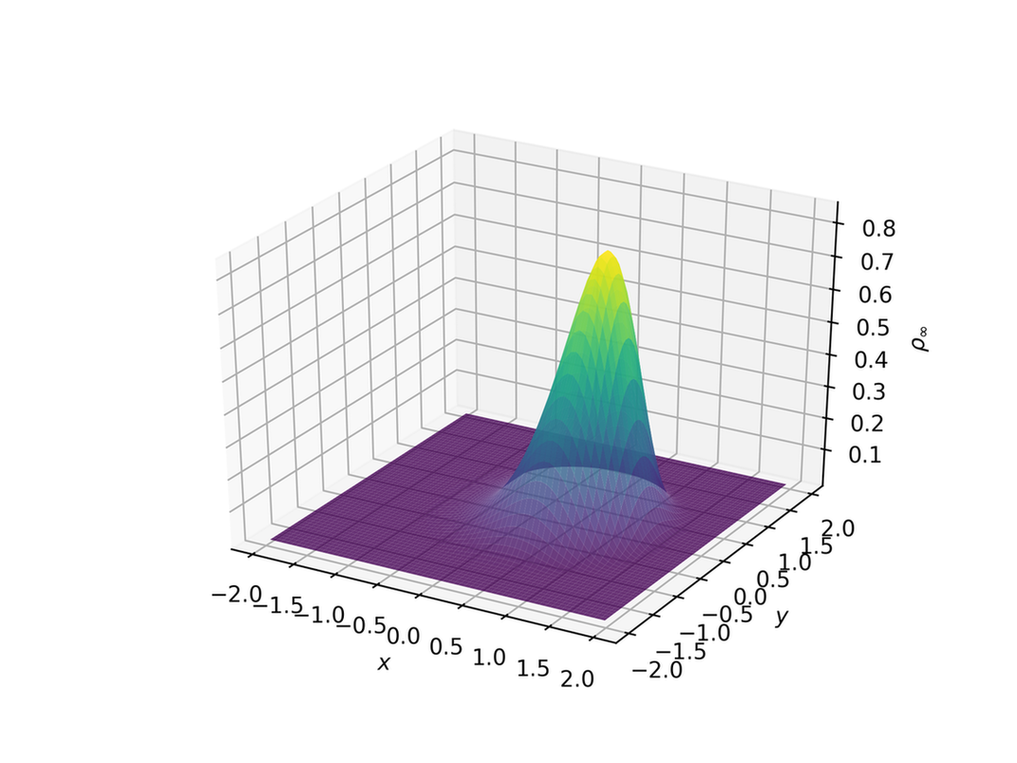}
		\end{subfigure}~
		\begin{subfigure}[t]{0.49\textwidth}
			\centering
			\includegraphics[width=\textwidth]{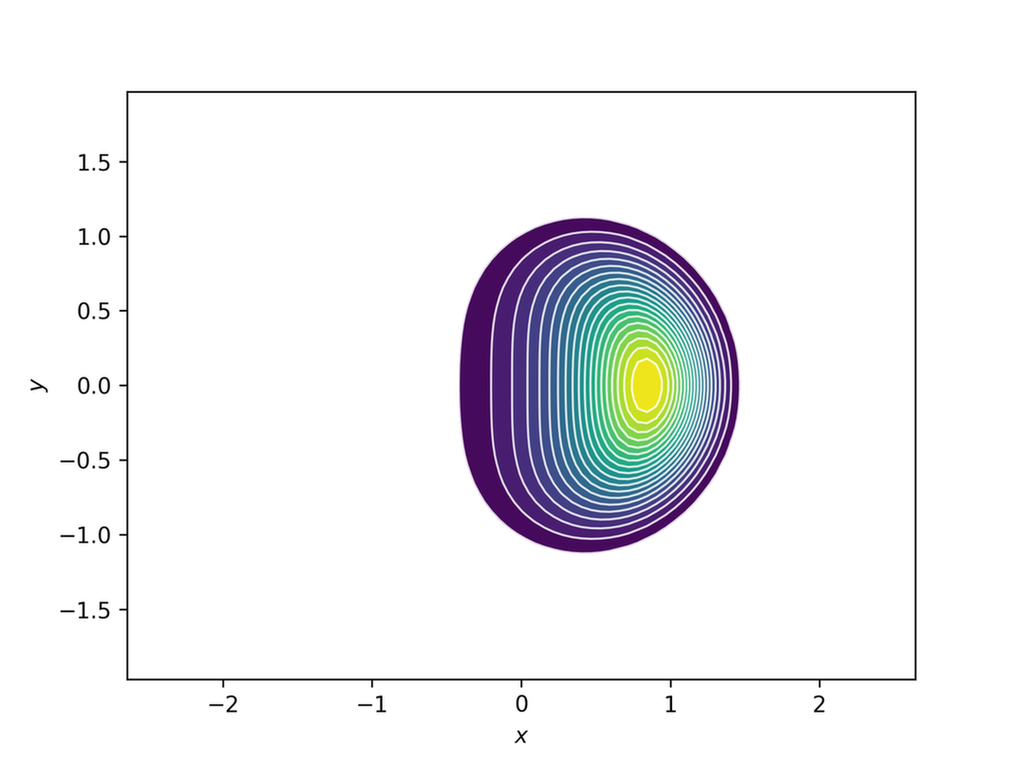}
		\end{subfigure}\caption{Stationary state $\rho_{\infty}(x)$.}
	\end{subfigure}

	\begin{subfigure}[t]{0.5\textwidth}
		\centering
		\includegraphics[width=\textwidth]{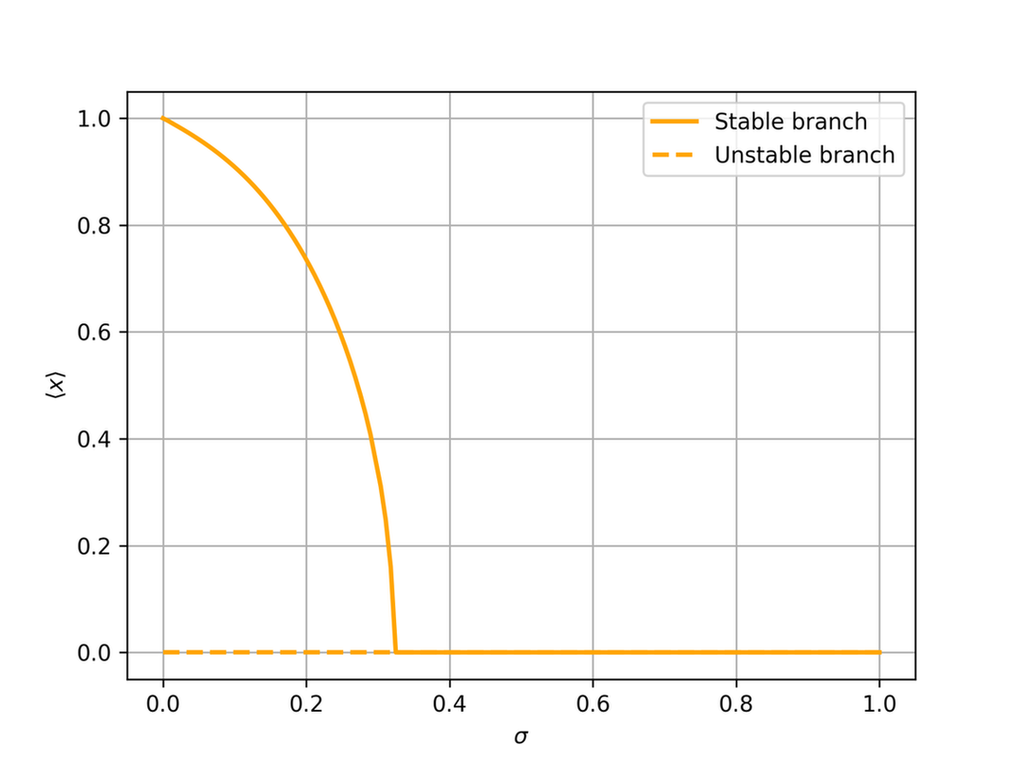}
		\caption{Bifurcation diagram.}
	\end{subfigure}\caption{Stationary states and phase transition of \eqref{eq:kineticflocking} for $\alpha=1$ in two dimensions.}
	\label{fig:kinetic2D}
\end{figure}
 
Finally, we discuss the phase transition for the non-linear diffusion case with and without a linear diffusion regularisation. This corresponds to:
\begin{equation}\label{eq:kineticflockingPME}
	H_\varepsilon(\rho)=\sigma\prt*{
		\frac{\rho^m}{m-1}+
		\varepsilon\prt{\rho\log(\rho)-\rho}
	},\quad
	V(\bx)=\alpha\prt*{\frac{\abs{\bx}^4}{4}-\frac{\abs{\bx}^2}{2}},\quad
	W(\bx)=\frac{\abs{\bx}^2}{2},
\end{equation}
for $\epsilon\geq0, \sigma\geq0, m>1, \alpha\geq0$. Figure \ref{fig:kineticPME} shows the stationary states without regularisation as well as the bifurcation diagrams for $\varepsilon=0$ and $\varepsilon>0$. The case shown, $m=2$, leads to compactly supported stationary states with Lipschitz regularity at the boundary of the support, see Figure \ref{fig:kineticPME} (a). The regularisation numerically compensates the loss of spatial accuracy of the scheme due to the lack of smoothness of the solution, requiring fewer mesh points to adequately capture the behaviour around the critical point. Numerically we observe that the bifurcation diagram is continuous with respect to the regularisation parameter $\varepsilon$.

\begin{figure}[ht]
	\centering
	\begin{subfigure}[t]{0.5\textwidth}
		\centering
		\includegraphics[width=\textwidth]{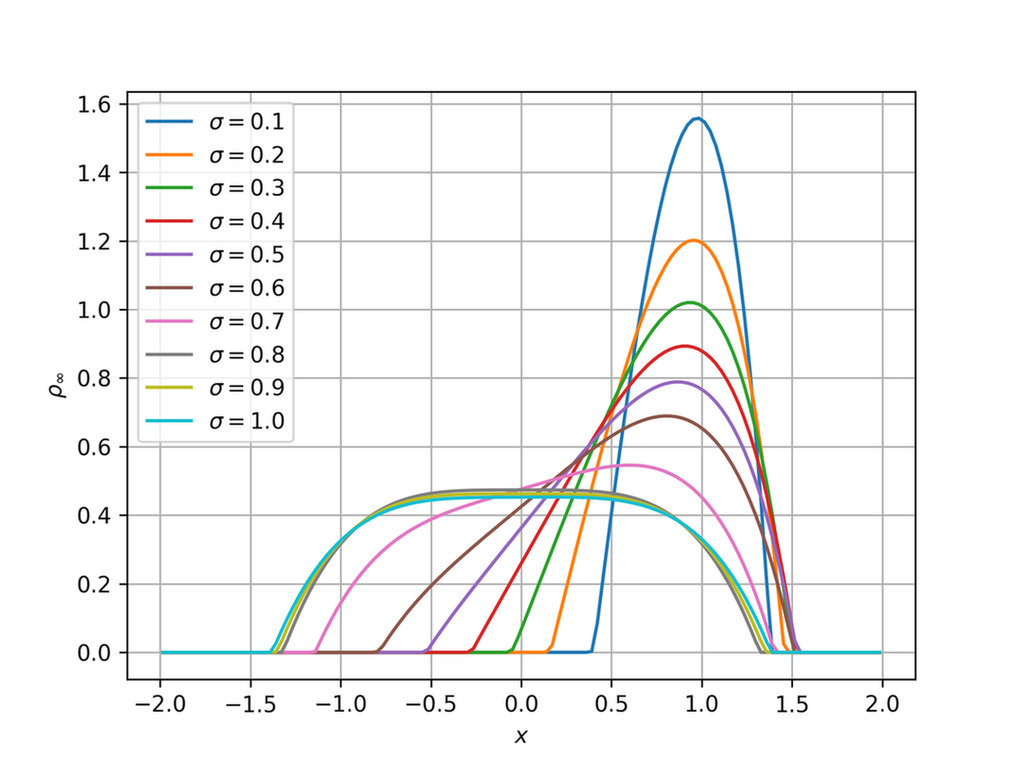}
		\caption{Stationary states $\rho_{\infty}$ for $H_0\prt{\rho}$ and varying $\sigma$.}
		\label{fig:kineticPME_A}
	\end{subfigure}

	\begin{subfigure}[t]{0.49\textwidth}
		\centering
		\includegraphics[width=\textwidth]{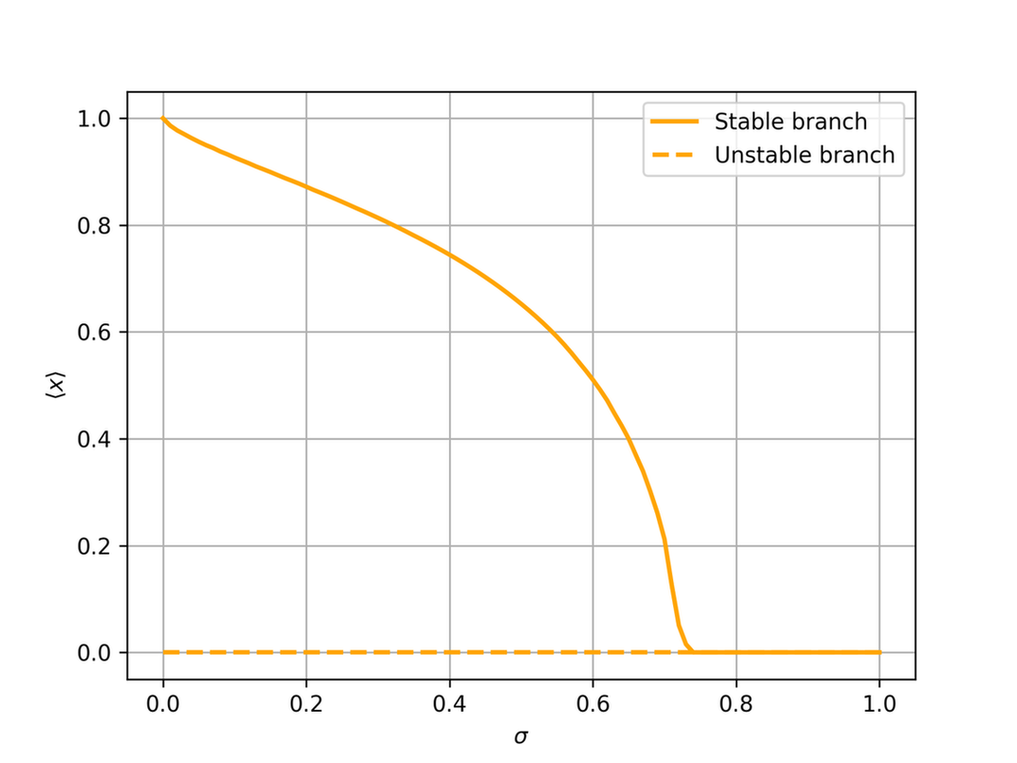}
		\caption{Bifurcation diagram for $H_0\prt{\rho}$.}
	\end{subfigure}~
	\begin{subfigure}[t]{0.49\textwidth}
		\centering
		\includegraphics[width=\textwidth]{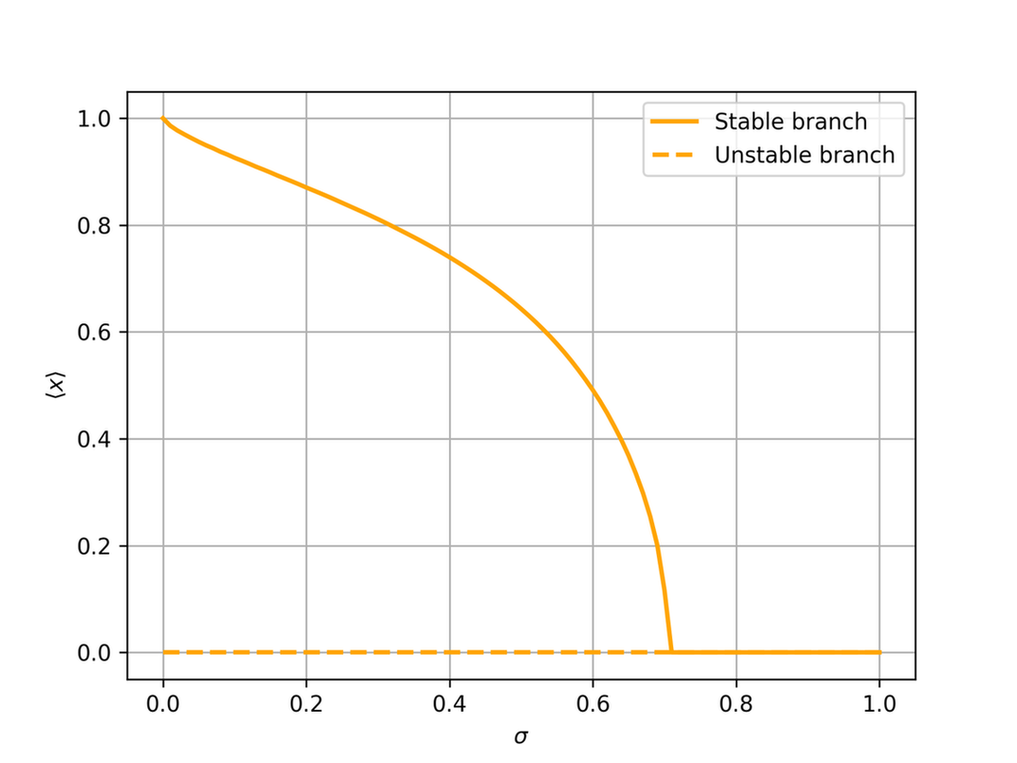}
		\caption{Bifurcation diagram for $H_{\varepsilon}\prt{\rho}$, with $\varepsilon=0.01$.}
	\end{subfigure}

	\caption{Stationary states and phase transition of \eqref{eq:kineticflockingPME} for $\alpha=1, m=2$ in one dimension.}
	\label{fig:kineticPME}
\end{figure}
  
\section*{Acknowledgements}

JAC was partially supported by the EPSRC grant number EP/P031587/1. JH was supported by NSF grant DMS-1620250 and NSF CAREER grant DMS-1654152. The authors are grateful to the Mittag-Leffler Institute for providing a fruitful working environment during the special semester \emph{Mathematical Biology} where this work was completed. JAC and JH thank the University of Texas at Austin for their invitation in September 2017 where this project started.
 \bibliographystyle{abbrv}
\bibliography{./main.bib}

\end{document}